\documentclass[11pt,reqno]{amsart}
\usepackage{amsmath,amssymb,amsthm,amsfonts,mathrsfs,latexsym,times,color,bm}
\usepackage{amsmath,amsfonts}
\usepackage{amsthm}

\usepackage{hyperref}

\usepackage{graphicx}
\usepackage{color}
\usepackage{multicol}

\usepackage[notref,notcite]{showkeys}

\newtheorem{thm}{Theorem}[section]
\newtheorem{definition}{Definition}[section]

\newtheorem{Lemma}[thm]{Lemma}
\newtheorem{remark}{Remark}[section]
\newtheorem{theorem}[thm]{Theorem}
\newtheorem{proposition}[thm]{Proposition}

\newtheorem{corollary}[thm]{Corollary}

\numberwithin{equation}{section}

\newcommand{\beq}{\begin{equation}}
\newcommand{\eeq}{\end{equation}}
\newcommand{\ben}{\begin{eqnarray}}
\newcommand{\een}{\end{eqnarray}}
\newcommand{\beno}{\begin{eqnarray*}}
\newcommand{\eeno}{\end{eqnarray*}}

\numberwithin{equation}{section}
\begin{document}
\title[Three-dimensional compressible Euler equations]{On the rough solutions of 3D compressible Euler equations: an alternative proof }


\author{Huali  Zhang}
\address{Changsha University
of Science and Technology, Changsha, 410114, P.R. of China.}
\email{zhlmath@yahoo.com}

\author{Lars Andersson}
\address{Max-Planck Institute for Gravitational Physics, Am M\"uhlenberg 1, D-14476, Potsdam, Germany}
\email{lars.andersson@aei.mpg.de}

\date{\today}


\begin{abstract}
The well-posedness of Cauchy problem of 3D compressible Euler equations is studied. By using Smith-Tataru's approach \cite{ST}, we prove the local existence, uniqueness and stability of solutions for Cauchy problem of 3D compressible Euler equations, where the initial data of velocity, density, specific vorticity $v, \rho \in H^s, \varpi \in H^{s_0} (2<s_0<s)$. It's an alternative and simplified proof of the result given by Q. Wang in \cite{WQEuler}.
\end{abstract}

\maketitle

\section{Introduction and Results}
\subsection{Overview}
In this paper we consider the Cauchy problem of compressible Euler equations in three dimensions, which is described as
\begin{equation}\label{CEE}
	\begin{cases}
	\rho_t+\text{div}\left(\rho v \right)=0,  \quad (t,x) \in \mathbb{R}^+ \times \mathbb{R}^3,
	\\
	v_t + \left(v\cdot \nabla \right)v+\frac{1}{\rho}\nabla p=0,
\end{cases}
\end{equation}
with the state function
\begin{equation*}
  p=p(\rho)=\rho^\gamma (\gamma>0),
\end{equation*}
and the initial data
\begin{equation}\label{id}
	\rho|_{t=0}=\rho_0, \ v|_{t=0}=v_0,
\end{equation}
where $v=(v_1,v_2,u_3)^{\text{T}}, \rho$, and $p$, respectively, denote the fluid velocity, density, and pressure. The object is to answer the question that: \textit{For which $s$, the problem \eqref{CEE}-\eqref{id} is well-posed if $(v_0,\rho_0) \in H^s(\mathbb{R}^n)$}.

This question has been well studied for the irrotational case and the incompressible case in two or three dimensions. For the 3D compressible Euler equations \eqref{CEE}, the best known result of the corresponding question is proved by Q. Wang in \cite{WQEuler}. We will discuss it in a different view.
\subsection{Background}
The compressible Euler equations is a classical system in physics to describe the motion of an idea fluid. The phenomena displayed in the interior of a fluid fall into two broad classes, the phenomena of acoustics waves and the phenomena of vortex motion.  The sound
phenomena depend on the compressibility of a fluid, while the vortex phenomena occur even in a
regime where the fluid may be considered to be incompressible.

For the Cauchy problem of $n$-D incompressible Euler equations:
\begin{equation}\label{IEE}
	\begin{cases}
	v_t + \left(v\cdot \nabla \right)v+\nabla p=0,  \quad (t,x) \in \mathbb{R}^+ \times \mathbb{R}^n,
	\\
	\mathrm{div} v=0,
\\
v|_{t=0}=v_0,
\end{cases}
\end{equation}
Kato and Ponce in \cite{KP} proved the local well-posedness of \eqref{IEE} if $v_0 \in W^{s,p}(\mathbb{R}^n), s>1+\frac{n}{p}$. Chae in \cite{Chae, Ch2} proved the local existence of solutions by setting $v_0$ in Triebel-Lizorkin spaces. On the opposite direction, the ill-posedness of solutions was solved by Bourgain and Li \cite{BL, BL2}, in which they proved that the solution will blow up instantaneously for some $v_0 \in H^{1+\frac{n}{2}}(\mathbb{R}^n), n=2,3$.

In the irrotational case, compressible Euler equations can be reduced to a special quasilinear wave equation (see Lemma \ref{wte} and taking $\varpi=0$). For general quasilinear wave equations:
\begin{equation}\label{qwe}
\begin{cases}
  &\square_{h(\phi)} \phi=q(d \phi, d \phi), \quad (t,x) \in \mathbb{R}^+ \times \mathbb{R}^{n},
  \\
  & \phi|_{t=0}=\phi_0, \partial_t \phi|_{t=0}=\phi_1,
  \end{cases}
\end{equation}
$\phi$ is a scalar function and $h(\phi)$ is a Lorentzian metric depending on $\phi$, and $q$ is a quadratic term of $d \phi$. Set the initial data $(\phi_0, \phi_1) \in H^s(\mathbb{R}^n) \times H^{s-1}(\mathbb{R}^n)$. By using classical energy methods and Sobolev imbeddings, Hughes-Kato-Marsden in \cite{HKM} proved the local well-posedness of the problem \eqref{qwe} for $ s>\frac{n}{2}+1$. On the other side, Lindblad in \cite{L} constructed some counterexamples for \eqref{qwe} when $s={\frac{7}{4}}, n=2$ or $s=2,n=3$. In mathematics, to lower the regularity, 
one may seek a type of space-time estimates of $d \phi$, namely,
Strichartz estimates. Rigorous mathematical study of Strichartz estimates was initiated by Bahouri-Chemin \cite{BC1,BC2} and Tataru \cite{T1,T2} respectively, who established the local well-posedness of \eqref{qwe} when $s > \frac{n}{2} + \frac{7}{8}, n=2$ or $s > \frac{n}{2} + \frac{3}{4}, n\geq3$. The same result was also obtained by Klainerman in \cite{K}. Through introducing a vector-field approach and a decomposition of curvature, the 3D result of \cite{BC1, BC2, T1, T2} was later improved by Klainerman and Rodnianski, where $s>2+\frac{2-\sqrt{3}}{2}$. Based on Klainerman and Rodnianski's vector-field methods, Geba in \cite{Geba} studied the local well-posedness of the 2D case for $s > \frac{7}{4} + \frac{5-\sqrt{22}}{4}$. In 2005, a sharp result was established by Smith and Tataru in \cite{ST}; they proved that the local solution of \eqref{qwe} is well-posed if the regularity of initial data satisfies $s>\frac{7}{4}, n=2$ or $s>2, n=3$ or $s>\frac{n+1}{2}, 4 \leq n \leq 6$. An alternative proof of the 3D result was also obtained through vector-field approach by Wang \cite{WQSharp}. Besides, we should also mention substantial significant progress which has been made on low
regularity solutions of Einstein vacuum equations, membrane equations or quasi-linear wave equations due to Andersson and Moncreif \cite{AM}, Ettinger and Lindblad \cite{BL}, Klainerman and Rodnianski \cite{KR, KR1}, Q. Wang \cite{WQRough, WQ1, WQ2}, Allen-Andersson-Restuccia \cite{AAR}, C.B. Wang \cite{WCB} and so on.

In the general case, concerning to the compressible Euler equations in $n$-D, the Cauchy problem \eqref{CEE}-\eqref{id} is well-posed if $v_0, \rho_0 \in H^{s}, s>1+\frac{n}{2}$ and the density is far away from vacuum, please see Majda's book \cite{M}.
Recently, a wave-transport structure of the 2D and 3D compressible Euler equations is introduced by Luk and Speck \cite{LS1,LS2}, in which their central theme is to describe a sharp asymptotic behavior of the singularity formation. Based on \cite{LS2}, Disconzi-Luo-Mazzone-Speck in \cite{DLS} proved
the existence of solutions with non-zero vorticity $\varpi$ and entropy $S$, and the assumptions on the
initial data of velocity $v$, density $\rho$ and $\varpi$ is in $H^{2+}(\mathbb{R}^3)$ and $S\in H^{3+}, \mathrm{curl} \varpi \in C^{0, \delta}$ with $0<\delta <1$. Independently, without the H\"older norm of $\mathrm{curl} \varpi$, Wang in \cite{WQEuler} proved
the local well-posedness of the 3D compressible Euler equations with initial data of $(v,\rho,\varpi) \in H^{s}(\mathbb{R}^3) \times H^{s}(\mathbb{R}^3) \times H^{s'}(\mathbb{R}^3), \ 2<s'<s $. Recently, some related 2D result is obtained by Zhang in \cite{Z1,Z2}. We should also mention substantial significant progress which has been made on shock formation and free boundary problem due to Sideris \cite{S}, Christodoulou-Miao \cite{C}, Luk-Speck \cite{LS1,LS2}, Coutand-Lindblad-Shkoller \cite{CLS}, Coutand-Shkoller \cite{CLS,CS2}, Jang-Masmoudi \cite{JM}, Ifrim-Tataru \cite{IT} and so on.
\subsection{Motivation and contribution}
In \cite{DLS, WQEuler}, they proved the low-regularity well-posedness of \eqref{CEE} in 3D by vector field approach. There is also a different and significant view to consider the low regularity problem, which is established by Smith and Tataru in \cite{ST}. This motivates us to generalize Smith-Tataru's result to 3D compressible Euler equations with non-trivial vorticity. It's a non-trivial process for us to obtain the same result established by Q. Wang in \cite{WQEuler}. Moreover, compared with \cite{WQEuler}, our paper is much simplified. Our contribution is in the following,

(i) combining two formulations(a hyperbolic system and a wave-transport system) for compressible Euler equations, which shows a direct way to discuss some energy estimates.

(ii) deriving a modified transport equation for vorticity.

(iii) obtaining characteristic energy estimates for vorticity along null hypersurfaces, which is different from basic energy estimates and non-trivial.

(iv) introducing a modified Strichartz estimate for linear wave equation endowed with the acoustic metric, which is adapting the structure of compressible Euler equations.

Before stating our result, let us introduce some quantities and notations.
\subsection{Quantities and notations}
\begin{definition}\label{pw}
Let $\bar{\rho}$ be a constant background density and $\bar{\rho}>0$. We denote the logarithmic density $\boldsymbol{\rho}$ and specific vorticity $\varpi$
\begin{equation}\label{pw1}
\boldsymbol{\rho} :=\ln \left( \frac{\rho}{{\bar{\rho}}}\right), \quad \varpi:={\rho}^{-1}{\mathrm{curl}v}={\rho}^{-1}\omega.
\end{equation}
\end{definition}
\begin{definition}\label{shengsu}
We denote the speed of sound
\begin{equation}\label{ss}
c_s:=\sqrt{\text{d}p/\text{d}\rho}.
\end{equation}
In view of \eqref{pw1}, we have
\begin{equation}\label{ss1}
c_s=c_s(\boldsymbol{\rho})
\end{equation}
and
\begin{equation}\label{ssd}
c'_s=c'_s(\boldsymbol{\rho}):=\frac{\text{d}c_s}{\text{d} \boldsymbol{\rho}}.
\end{equation}
\end{definition}
\begin{definition}\label{metricd}
We define the acoustical metric $g$ and the inverse acoustical metric $g^{-1}$ relative to the Cartesian coordinates as follows:
\begin{align}
&g:=-dt\otimes dt+c_s^{-2}\sum_{a=1}^{3}\left( dx^a-v^adt\right)\otimes\left( dx^a-v^adt\right),
\\
&g^{-1}:=-(\partial_t+v^a \partial_a)\otimes (\partial_t+v^b \partial_b)+c_s^{2}\sum_{i=1}^{3}\partial_i \otimes \partial_i.
\end{align}	
\end{definition}
We are ready to introduce a reduction of \eqref{CEE} by using the logarithmic density in Lemma \ref{pw}, which is first proposed by Luk-Speck \cite{LS1}.
\begin{Lemma}\label{wte} \cite{LS1}
Let $\varpi$ and $\boldsymbol{\rho}$ are defined in \eqref{pw1}. For 3D compressible Euler equations, we can reduce \eqref{CEE} to the following equations:
\begin{equation}\label{fc0}
\begin{cases}
\mathbf{T} v^i=-c^2_s \partial^i \boldsymbol{ \rho},
\\
\mathbf{T} \boldsymbol{\rho}=-\mathrm{div} v,
\end{cases}
\end{equation}
where $\mathbf{T}=\partial_t + v \cdot \nabla $. 
\end{Lemma}
For brevity, we set $\partial=(\partial_{1}, \partial_{2}, \partial_{3})^\mathrm{T}$, $d=(\mathbf{T}, \partial_{1}, \partial_{2}, \partial_{3})^\mathrm{T}$ in this paper. We denote the operator
\begin{equation}\label{Box}
  \square_g := g^{\alpha\beta}\partial^2_{\alpha\beta}=-\mathbf{T}\mathbf{T}+c^2_s \Delta ,
\end{equation}
where $\Delta=\partial^2_1+ \partial^2_2+\partial^2_3$. Introduce a decomposition for the velocity
\begin{equation}\label{dvc}
  v^i=v_{+}^i+ \eta^i,
\end{equation}
where the vector $\eta=(\eta^1, \eta^2, \eta^3)^{\mathrm{T}}$ is defined by
\begin{equation}\label{etad}
  -\Delta\eta^{i}:=\mathrm{e}^{\boldsymbol{\rho}}\mathrm{curl} \varpi^i .
\end{equation}
Let us denote
\begin{equation*}
\left< \xi \right>=(1+|\xi|^2)^{\frac{1}{2}}, \ \xi \in \mathbb{R}^3.
\end{equation*}
Denote by $\left< \partial \right>$ the corresponding Bessel potential multiplier. The symbol $\epsilon^{ijk}(i, j, k=1, 2, 3)$ denote the standard volume form on $\mathbb{R}^3$.

Following the book \cite{BCD}, the symbol ${\Delta}_j$ denotes the homogeneous frequency localized operator with frequency $2^j, j \in \mathbb{Z}$. For $f \in H^s(\mathbb{R}^3)$, we denote $\|f\|_{H^s}:= \|f\|_{L^2}+\|f\|_{\dot{H}^s}$, where $\|f\|^2_{\dot{H}^s} := {\sum_{j \geq -1}} 2^{2js}\|{\Delta}_j f\|^2_{L^2} $. We also denote $\|f\|^r_{\dot{B}^s_{p,r}} := {\sum_{j \geq -1}}2^{jsr}\|{\Delta}_j f\|^r_{L^p}$. We set $\eta_0:= \eta(t,x)|_{t=0}$ and
\begin{equation}\label{a1}
\delta_0=s_0-2, \quad \delta\in (0, \delta_0).
\end{equation}

The notation $X \lesssim Y$ means $X \leq CY$, where $C$ is a universal constant. The notation $X \simeq Y$ means $C_1Y \leq X \leq C_2Y$, where $C_1$ and $C_2$ are universal constants. We use the notation $X \ll Y$ to mean that $X \leq CY$ with a sufficiently large constant $C$.

We also assume $2< s_0 < s< \frac{5}{2}$ and use four small parameters
\begin{equation}\label{a0}
\epsilon_3 \ll \epsilon_2 \ll \epsilon_1 \ll \epsilon_0 \ll 1.
\end{equation}

Now, we are ready to state the result in this paper.
\subsection{Statement of result.}
\begin{theorem}\label{dingli}
	Consider the Cauchy problem \eqref{CEE}-\eqref{id}. Assume that
\begin{equation}\label{HE}
  c_s|_{t=0}>c_0>0,
\end{equation}
where $c_0$ is a positive constant. Let $s>s_0>2$. Let $\boldsymbol{\rho}$ and $\varpi$ be defined in \eqref{pw1}. For any $M_0>0$ and initial data $(v_0, \rho_0)$ satisfying
	\begin{equation}\label{chuzhi}
	\| v_0\|_{H^{s}} +
	\| \boldsymbol{\rho}_0\|_{H^{s}} + \| \varpi_0\|_{H^{s_0}}
 \leq M_0,
	\end{equation}
there exist positive constants $T_*$ and $M_1$ such that \eqref{CEE}-\eqref{id} has a unique solution $(v,\boldsymbol{\rho}) \in C([0,T_*],H^s)$, $\varpi \in C([0,T_*],H^{s_0})$.
To be precise, 

	 $\mathrm{(1)}$ the solution $v, \boldsymbol{\rho}$ and $\varpi$ satisfy the energy estimate
\begin{equation*}
  \|v, \boldsymbol{\rho}\|_{L^\infty_tH^s}+ \|\varpi\|_{L^\infty_tH^{s_0}} \leq M_1,
\end{equation*}

 $\mathrm{(2)}$ the solution $v, \boldsymbol{\rho}$ and $v_+$ satisfy the Strichartz estimate
\begin{equation*}
  \|dv, d\boldsymbol{\rho}, dv_{+}\|_{L^2_tL_x^\infty}+ \|dv, d\boldsymbol{\rho}\|_{L^2_t \dot{B}^{s_0-2}_{\infty,2}}+ \|\partial v_{+}\|_{L^2_t \dot{B}^{s_0-2}_{\infty,2}} \leq M_1.
\end{equation*}

 $\mathrm{(3)}$ for any $1 \leq r \leq s_0+1$, and for each $t_0 \in [0,T]$, the linear equation
	\begin{equation}\label{linear}
	\begin{cases}
	& \square_g f=\mathbf{T}G+B, \qquad (t,x) \in [0,T]\times \mathbb{R}^3,
	\\
	&f(t_0,\cdot)=f_0 \in H^r(\mathbb{R}^3), \quad \mathbf{T} f(t_0,\cdot)=f_1 \in H^{r-1}(\mathbb{R}^3),
	\end{cases}
	\end{equation}
admits a solution $f \in C([0,T],H^r) \times C^1([0,T],H^{r-1})$ and the following estimates hold:
\begin{equation}\label{E0}
\| f\|_{L_t^\infty H^r}+ \|\partial_t f\|_{L_t^\infty H^{r-1}} \lesssim \|f_0\|_{H^r}+ \|f_1\|_{H^{r-1}}+\| G\|_{L^\infty_tH^{r-1} \cap L^1_tH^r}+\|B\|_{L^1_tH^{r-1}}.
\end{equation}
Additionally, the following estimates hold, provided $k<r-1$,
\begin{equation}\label{SE1}
\| \left<\partial \right>^k f\|_{L^2_{t}L^\infty_x} \lesssim  \|f_0\|_{H^r}+ \|f_1\|_{H^{r-1}}+ \| G\|_{L^\infty_tH^{r-1} \cap L^1_tH^r}+\|B\|_{L^1_tH^{r-1}} .
\end{equation}
\end{theorem}
\begin{remark}
The condition \eqref{HE} is used to satisfy the hyperbolicity condition of the system \eqref{CEE}.
\end{remark}
\begin{remark}
The type of \eqref{linear} is devised for adapting to the structure of the acoustic metric.
\end{remark}
\subsection{A sketch of the proof}
The first step is to obtain energy estimates. Different from Q. Wang's paper \cite{WQRough}, we use the hyperbolic system \eqref{sq} to derive the basic energy of velocity and density
\begin{equation}\label{EE}
  \|v, \boldsymbol{\rho}\|_{H^a} \leq \|v_0, \boldsymbol{\rho}_0\|_{H^a}\exp( {\| dv, d\boldsymbol{\rho}\|_{L^1_tL^\infty_x}}), \quad a \geq 0,
\end{equation}
which shows that $\|v, \boldsymbol{\rho}\|_{H^a}$ is independent of the vorticity. If $a > \frac{5}{2}$, then the classical commutator estimates and continuity method can be used to prove the well-posedness of the problem \eqref{CEE}-\eqref{id}, one can refer Majda's book \cite{M}. However, there is no uniform space-time estimates of ${\| dv, d\boldsymbol{\rho}\|_{L^1_tL^\infty_x}}$ for \eqref{CEE} if $a \leq \frac{5}{2}$. 
However, with trival vorticity($\mathrm{curl}v=0$), the compressible Euler equations can be reduced to a special quasilinear wave equation, and there is a type of Strichartz estimates, i.e. ${\| dv, d\boldsymbol{\rho}\|_{L^2_tL^\infty_x}}$, which can have $\frac12$-regularity decrease compared to the classical result by Majda. Then the regularity requires greater than $2$ in this case, please refer Smith-Tataru's paper \cite{ST}(also in Q. Wang's paper \cite{WQSharp} for an alternative proof). What's the situation with non-trivial vorticity. Let us see Luk-Speck's result \cite{LS2}, where they introduced a wave-transport system \eqref{wtz} for 3D compressible Euler equations. In some extent, we can expect to lower the regularity of velocity and density if seeing it as a "disturbed" wave equation, and the disturbance is described by a transport equation. In precise,
the wave-transport system is the following
\begin{equation}\label{wtz}
\begin{cases}
&\square_g v= -\mathrm{e}^{\boldsymbol{\rho}}c_s^2 \mathrm{curl} \varpi+\mathrm{quadratic \ terms},
\\
&\square_g \boldsymbol{\rho}=\mathrm{quadratic \ terms},
\\
& \mathbf{T}\varpi=(\varpi \cdot \nabla)v.
\end{cases}
\end{equation}
Based on the wave-transport system \eqref{wtz}, the wave equation of $v$ and $\boldsymbol{\rho}$ play a crucial role for it's character. Then, it returns back to require some information  $\varpi$ or $\mathrm{curl} \varpi$. Because the transport equation for $\varpi$ is not good as the wave equation, so we expect a Strichartz estimates ${\| dv, d\boldsymbol{\rho}\|_{L^2_tL^\infty_x}}$ with a smallest regularity for $\varpi$. A key observation is that, seeing from the best known results \cite{ST, WQSharp}, we can hope that there is some Strichartz estimates if the regularity of $v, \rho$, and $W$ is greater than $2$. We then set the initial data $(v_0, \boldsymbol{\rho}_0) \in H^s,s>2$ and $\varpi_0 \in H^{s_0},s_0>2$. Therefore, we need to establish the energy estimates for $\varpi$ in the frame $H^{s_0}(s_0>2)$. Followed by \cite{LS2, WQEuler}, we also set $\Omega=\mathrm{e}^{-\boldsymbol{\rho}}\mathrm{curl}\varpi$. Different from \cite{WQEuler}, we derive a modified transport equation for $\mathrm{curl} \Omega$ (see Lemma \ref{PW}):
\begin{equation}\label{OS}
\begin{split}
\mathbf{T} \big( \mathrm{curl} \Omega^i -2 \mathrm{e}^{-\boldsymbol{\rho}}  \partial_a \boldsymbol{\rho}  \partial^i \varpi^a \big)
= &\partial^i \big( 2 \mathrm{e}^{-\boldsymbol{\rho}}  \partial_n v^a \partial^n \varpi_a \big)
\\
&+(\partial v, \partial \boldsymbol{\rho}) \cdot \partial^2 \varpi+(\partial v, \partial \boldsymbol{\rho}) \cdot \partial^2 v+\partial \varpi \cdot \partial \varpi.
\end{split}
\end{equation}
We consider the term $\mathrm{curl} \Omega-2 \mathrm{e}^{-\boldsymbol{\rho}}  \partial_a \boldsymbol{\rho}  \partial \varpi^a$ as a whole part, and then transfer the goal to obtain the estimates $\|\mathrm{curl} \Omega -2 \mathrm{e}^{-\boldsymbol{\rho}}  \partial_a \boldsymbol{\rho}  \partial \varpi^a\|_{H^{s_0-2}}$. In this process, we need to handle the trouble term
\begin{equation*}
  \int_{\mathbb{R}^3}\Lambda^{s_0-2}\partial^i \big( 2 \mathrm{e}^{-\boldsymbol{\rho}}  \partial_n v^a \partial^n \varpi_a \big)\cdot \Lambda^{s_0-2}\big( -2 \mathrm{e}^{-\boldsymbol{\rho}}  \partial_a \boldsymbol{\rho}  \partial_i \varpi^a \big)dx.
\end{equation*}
By Plancherel formula, we can transfer derivatives in the way
\begin{equation*}
  \int_{\mathbb{R}^3} \Lambda^{s_0-\frac{5}{2}}\partial^i \big( 2 \mathrm{e}^{-\boldsymbol{\rho}}  \partial_n v^a \partial^n \varpi_a \big)\cdot \Lambda^{s_0-\frac{3}{2}}\big( -2 \mathrm{e}^{-\boldsymbol{\rho}}  \partial_a \boldsymbol{\rho}  \partial_i \varpi^a \big)dx.
\end{equation*}
After getting the desired energy estimates, we then use Young's inequality to handle the lower order term $2 \mathrm{e}^{-\boldsymbol{\rho}}  \partial_a \boldsymbol{\rho}  \partial \varpi^a$. Based on this and \eqref{EE}, we can obtain the following energy estimates
\begin{equation*}
\begin{split}
 & \| (\boldsymbol{\rho}, v)(t)\|_{H^s}+\|\varpi(t)\|_{H^{s_0}}
 \lesssim (\| (\boldsymbol{\rho}_0, v_0)\|_{H^s}+\|\varpi_0\|_{H^{s_0}})  \exp \big ( {\int^t_0} (\|(dv, d\boldsymbol{\rho})\|_{L^\infty_x}+\|\partial v\|_{\dot{B}^{s_0-2}_{\infty,2}})d\tau \big).
\end{split}
\end{equation*}
One can see Theorem \ref{PW}, Theorem \ref{dv}, and Theorem \ref{ve} for details. In a word, the problem is concluded to bound the Strichartz estimate
\begin{equation}\label{SO}
  {\int^t_0} \|(dv, d\boldsymbol{\rho})\|_{L^\infty_x}d\tau+{\int^t_0}\|\partial v\|_{\dot{B}^{s_0-2}_{\infty,2}}d\tau ,
\end{equation}
for some $t>0$. Here, we take $2<s_0<s$. Because, if we take $s=s_0$, then \eqref{SO} cannot be proved if we refer \cite{ST, WQSharp}.

The second difficult step is the Stricharz estimate \eqref{SO}. To do that, we extend Smith-Tataru's method \cite{ST} to compressible Euler equations, which is totally different from Q. Wang's work \cite{WQRough} by using vector-field approach. Based on Smith-Tataru's work, we first reduce the problem of establishing an existence result for small, supported initial data. Next, by the continuity method, we can give a bootstrap argument on the regularity of the solutions to the nonlinear equation. Then, by introducing null hypersurfaces, the key is transformed to prove characteristic energy estimates of solutions along null hypersurfaces, and the enough regularity of null hypersurfaces is crucial to prove the Strichartz estimate. To establish characteristic energy estimates, we go back to see the wave-transport system. Note the regularity of $\varpi$ is only $s_0$. Then, we can only get the same level regularity $s_0$ on characteristic hypersurfaces(see Section 6 for details). We use the wave equation of $(v, \boldsymbol{\rho})$ to get these characteristic energy estimates on characteristic hypersurfaces. As for $\varpi$, the characteristic energy estimate is difficult. Let us explain it as follows. On the Cauchy slice $\{t=\tau\}\times \mathbb{R}^3$, we can use elliptic estimates to get the energy estimate of all derivatives of $\varpi$ by using $\mathrm{div}\varpi$ and $\mathrm{curl}\varpi$. However, on the characteristic hypersurface, these type of elliptic energy estimates don't work. To go through this difficulty, we use Hodge decomposition to recover some transport equations for derivatives of $\varpi$, where it involves Riesz operator. Then, some commutator estimated for Riesz operator and $v \cdot \nabla$ concerning to compressible fluid is required. Please refer Lemma \ref{ceR}, \ref{LPE}, \ref{ce}, \ref{te3}, and \ref{te20} for detials.

After obtaining enough regularity of null hypersurfaces and coefficients from null frame, we can reduce the problem to proving the Strichartz estimate of a linear wave equation endowed with the acoustical metric $g$. Precisely, for any $1 \leq r \leq s_0+1$, and for each $t_0 \in [0,T]$, the linear equation \eqref{linear}
admites a solution $f \in C([0,T],H^r) \times C^1([0,T],H^{r-1})$ and the following estimates holds:
\begin{equation}\label{lw0}
\| f\|_{L_t^\infty H^r}+ \|\partial_t f\|_{L_t^\infty H^{r-1}} \leq C \big( \|f_0\|_{H^r}+ \|f_1\|_{H^{r-1}}+\| G\|_{L^\infty_tH^{r-1}\cap L^1_tH^r}+\|B\|_{L^1_tH^{r-1}}\big).
\end{equation}
Additionally, the following estimates hold, provided $k<r-1$,
\begin{equation}\label{lw1}
\| \left<\partial \right>^k f\|_{L^2_{t}L^\infty_x} \leq C \big( \|f_0\|_{H^r}+ \|f_1\|_{H^{r-}}+\| G\|_{L^\infty_tH^{r-1}\cap L^1_tH^r}+\|B\|_{L^1_tH^{r-1}}\big).
\end{equation}
We still need to prove
\begin{equation*}
  {\int^t_0}\|\partial v\|_{\dot{B}^{s_0-2}_{\infty,2}}d\tau.
\end{equation*}
It can not be obtained by using the wave equation of velocity, for the regularity of source term $\mathrm{curl} \varpi$ is only $s_0-1$.
By using the wave equation for $v_{+}$ (see Lemma \ref{wte}),
\begin{equation}\label{l4}
\begin{split}
&\square_g \boldsymbol{\rho}=\mathcal{D}  ,
\\
&\square_g v^i_{+}=\mathbf{T} \mathbf{T} \eta^i+Q^i,
\end{split}
\end{equation}
we can bound \begin{equation*}
  {\int^t_0}\|\partial \boldsymbol{\rho},  \partial v_{+}\|_{\dot{B}^{s_0-2}_{\infty,2}}d\tau
\end{equation*}
by Littlewood-Paley decomposition. It crucially relies on the regularity of right hand terms can be reached to $s-1$. More precisely, operating $\Delta_j$ on \eqref{l4}, we can derive that $\Delta_j \boldsymbol{\rho}$ and $\Delta_j v^i_{+}$ satisfy
\begin{equation*}
\begin{cases}
 &\square_g \Delta_j \boldsymbol{\rho}= \Delta_j \mathcal{D}+ [\square_g, \Delta_j]\boldsymbol{\rho},
 \\
  & \Delta_j \boldsymbol{\rho}|_{t=0}= \Delta_j \boldsymbol{\rho}_0,
\end{cases}
\end{equation*}
and
\begin{equation*}
\begin{cases}
& \square_g \Delta_j v^i_{+}=
\mathbf{T}\Delta_j \mathbf{T} \eta^i+ \Delta_j Q^i+ [\square_g, \Delta_j]v^i_{+}+[\Delta_j, \mathbf{T}]\mathbf{T} \eta^i,
\\
& \Delta_j v^i_{+}|_{t=0}=\Delta_j (v_0-\eta_0), \quad \Delta_j \mathbf{T}v^i_{+}|_{t=0}=\Delta_j \mathbf{T}(v_0-\eta_0).
\end{cases}
\end{equation*}
By using the Strichartz estimate in \eqref{SE1} (taking $r=s-s_0+1, k=0$) 
, we obtain
\begin{equation}\label{isE}
\begin{split}
  \| (\Delta_j \boldsymbol{\rho}, \Delta_j v_{+})\|_{L^2_t L^\infty_x} \lesssim \ & \mathrm{"Right \ hand \ side"}
\end{split}
\end{equation}
Multiplying $2^{(s_0-1)j}$ on \eqref{isE}, and taking square of it and summing it over $j\geq 1$, we get
\begin{equation*}
\|\partial \boldsymbol{\rho}, \partial v_{+}\|^2_{L^2_t \dot{B}^{s_0-2}_{\infty, 2}}  \lesssim \| \boldsymbol{\rho}, v_{+}\|^2_{L^2_t \dot{B}^{s_0-1}_{\infty, 2}}
   \lesssim \| \boldsymbol{\rho}_0\|^2_{H^s}+\| v_0\|^2_{H^s}+ \| \varpi_0\|^2_{H^{s_0}}.
\end{equation*}
Based on the above estimate, by using \eqref{etad}, we can conclude the proof. As for the precise proofs, please refer Section 7. The last step is to prove \eqref{lw0} and \eqref{lw1}, which is presented in the Appendix part.


\subsection{Outline of the paper}
The organization of the remainder of this paper is as follows. 
In Section 2, we give a treatment of hyperbolic-transport system and introduce energy estimates and stability theorem. In Section 3, we reduce our problem to the case of smooth initial data by using compactness methods. In the subsequent Section, using physical localized technique, we reduce the problem to the case of smooth, small, compacted supported initial data. In section 5, we give a bootstrap argument based on continuous functional.
In Section 6 we derive self-contained characteristic energy estimates along null hypersurfaces, which is used to prove the regularity of null hypersurfaces. In the next step, we give the proof of Strichartz estimates in Section 7. In the Appendix, referring \cite{ST}, we give a glue proof of Strichartz estimates of linear wave equations endowed with the acoustic metric.
\section{Preliminaries: commutator and energy estimates}
\subsection{Some formulations to \eqref{CEE}}
Let us first introduce a symmetric hyperbolic system to 3D compressible Euler equations.
\begin{Lemma} \label{sh}
{Let $v, \boldsymbol{\rho}$ be a solution of \eqref{fc0}}. Then it also satisfies the following hyperbolic system
\begin{equation}\label{sq}
  U_t +\sum_{i=1}^3 {A}_i(U)U_{x_i}=0,
\end{equation}
where $U=(\boldsymbol{\rho}, v_1,v_2,v_3)^\mathrm{T}$ and
\begin{equation*}
A_1=\left(
\begin{array}{cccc}
v_1 & 1 & 0 & 1\\
c_s^2 &  v_1 & 0&0 \\
0 & 0 &  v_1&0 \\
0 & 0 & 0 & v_1
\end{array}
\right ),\quad
A_2=\left(
\begin{array}{cccc}
0 & 0 & 1 & 0\\
0 & v_2 & 0& 0\\
c^2_s & 0 &  v_2 &0\\
0 & 0 & 0 & v_2
\end{array}
\right ),
\quad
A_3=\left(
\begin{array}{cccc}
0 & 0 & 0 & 1\\
0 &  v_3 & 0& 0\\
0 & 0 & v_3 &0\\
c^2_s & 0& 0 & v_3
\end{array}
\right ).
\end{equation*}
\end{Lemma}
We are ready to introduce a wave-transport reduction.
\begin{Lemma}\label{wte} \cite{LS1}
Let $\varpi$ and $\boldsymbol{\rho}$ are defined in \eqref{pw1}. We can reduce \eqref{fc0} to
\begin{equation}\label{fc1}
\begin{cases}
\square_g v^i=-\mathrm{e}^{\boldsymbol{\rho}}c_s^2 \mathrm{curl} \varpi^i+Q^i,
\\
\square_g \boldsymbol{\rho}=\mathcal{D},
\\
 \mathbf{T}\varpi^i=\varpi^a \partial_a v^i.
\end{cases}
\end{equation}
Above, $Q^i$ and $\mathcal{D}$ are null forms relative to $g$, which are defined by
\begin{equation}\label{DDi}
\begin{split}
 Q^i:=& 2e^{\boldsymbol{\rho}} \epsilon^i_{ab} \mathbf{T} v^a \varpi^b-\left( 1+c_s^{-1}c'_s\right)g^{\alpha \beta} \partial_\alpha \boldsymbol{\rho} \partial_\beta v^i,
\\
\mathcal{D}:=& -3c_s^{-1}c'_sg^{\alpha \beta} \partial_\alpha \boldsymbol{\rho} \partial_\beta \boldsymbol{\rho}+2 \textstyle{\sum_{1 \leq a < b \leq 3} }\big\{ \partial_a v^a \partial_b v^b-\partial_a v^b \partial_b v^a \big\}.
\end{split}
\end{equation}
\end{Lemma}
For brevity, we set $Q=(Q^1,Q^2,Q^3)^\mathrm{T}$.
\begin{remark}
The equation \eqref{fc0} is derived from \eqref{CEE} via \eqref{pw1}. The system \eqref{sq}, \eqref{fc0} and \eqref{fc1} is equivalent to \eqref{CEE} respectively. 
We will switch these equivalent systems from one to another without explanation.
\end{remark}
\begin{Lemma}\label{wte1}
Let $v$ and ${\rho}$ be a solution of \eqref{CEE}. Let $\boldsymbol{\rho}$, $v_{+}$, and $\eta$ be described in \eqref{pw1}, \eqref{etad}, and \eqref{dvc} respectively. Then $v_{+}$ satisfies
\begin{equation}\label{fc}
\begin{split}
&\square_g v^i_{+}=\mathbf{T}\mathbf{T} \eta^i+Q^i. 
\end{split}
\end{equation}
\end{Lemma}
\begin{proof}
First, by Lemma \ref{wte}, we have
\begin{equation*}
  \square_g v^i=-\mathrm{e}^{\boldsymbol{\rho}}c_s^2 \mathrm{curl} \varpi^i+Q^i.
\end{equation*}
Substituting \eqref{dvc} and \eqref{etad} to the above equality, we then get
\begin{equation*}
  \begin{split}
  \square_g v_{+}^i=&-\square_g \eta^i-\mathrm{e}^{\boldsymbol{\rho}}c_s^2 \mathrm{curl} \varpi^i+Q^i
  \\
  =& \mathbf{T}\mathbf{T}\eta^i-c^2_s \Delta \eta^i-\mathrm{e}^{\boldsymbol{\rho}}c_s^2 \mathrm{curl} \varpi^i+Q^i
  \\
  =& \mathbf{T}\mathbf{T}\eta^i+c^2_s \left( - \Delta \eta^i- \mathrm{e}^{\boldsymbol{\rho}}\mathrm{curl} \varpi^i \right) +Q^i
  \\
  =& \mathbf{T}\mathbf{T}\eta^i+Q^i.
  \end{split}
\end{equation*}
\end{proof}
\begin{Lemma}\label{PW}
Let $\varpi$ be defined in \eqref{pw1}. Then $\varpi$ satisfy
\begin{equation}\label{W0}
\mathbf{T} \varpi =  (\varpi \cdot \nabla )v.
\end{equation}
If setting $\Omega=e^{-\boldsymbol{\rho}}\mathrm{curl}\varpi$, then
\begin{equation}\label{W1}
\mathbf{T}  \Omega^i = -2\epsilon^{imn}\mathrm{e}^{-\boldsymbol{\rho}}\partial_m v^a \partial_n \varpi_a+\epsilon^{amn} \mathrm{e}^{-\boldsymbol{\rho}} \partial_a v^i \partial_m \varpi_n,
\end{equation}
and
\begin{equation}\label{W01}
\mathrm{div} \varpi= -\varpi^i \partial_i \boldsymbol{\rho} ,
\end{equation}
hold. Furthermore, the quantity $\mathrm{curl} \Omega$ satisfies
\begin{equation}\label{W2}
\begin{split}
& \mathbf{T} \big( \mathrm{curl} \Omega^i -2\mathrm{e}^{-\boldsymbol{\rho}} \partial_a \boldsymbol{\rho}  \partial^i \varpi^a \big)
= \ \partial^i \big( 2 \mathrm{e}^{-\boldsymbol{\rho}}  \partial_n v_a \partial^n \varpi^b \big) + \sum^6_{j=1}R^{i}_j,
\end{split}
\end{equation}
where
\begin{equation}\label{rF}
\begin{split}
R_1^i:= & -2\mathrm{e}^{-\boldsymbol{\rho}}\epsilon_{kmn}\epsilon^{ijk} \partial^m v^a \partial_{j} (\partial^n\varpi_a)+ \mathrm{e}^{-\boldsymbol{\rho}}\epsilon^{amn} \epsilon^{ijk} \partial_a v_k \partial^2_{mj}\varpi_n
\\
&  -2\mathrm{e}^{-\boldsymbol{\rho}}  \partial_j v^a \partial^{ij} \varpi^b+ \epsilon^{ijk} \partial_j v^m \partial_m \Omega_k,
\\
R_2^i:=& \ 2\mathrm{e}^{-\boldsymbol{\rho}}\epsilon_{kmn}\epsilon^{ijk} \partial^m v^a \partial^{n}\varpi_a \partial_j \boldsymbol{\rho}- \mathrm{e}^{-\boldsymbol{\rho}}\epsilon^{amn} \epsilon^{ijk} \partial_a v_k \partial_{m}\varpi_n \partial_j \boldsymbol{\rho}
\\
& \ -2\mathrm{e}^{-\boldsymbol{\rho}}  \partial^a v^k \partial_k \boldsymbol{\rho}  \partial^i \varpi_i+ 2\mathrm{e}^{-\boldsymbol{\rho}} \partial^a \boldsymbol{\rho}  \partial^i\varpi_a
\\
& \ - 2\mathrm{e}^{-\boldsymbol{\rho}} \partial^a \boldsymbol{\rho} \partial^i \varpi^m \partial_m v_a+2\mathrm{e}^{-\boldsymbol{\rho}}  \partial^i \boldsymbol{\rho} \partial^{n} v^a \partial_n \varpi_a,
\\
R_3^i:=& \epsilon^{amn} \varpi^i  \partial_a \boldsymbol{\rho} \partial_{m}\varpi_n+2 \epsilon^{ajk} \partial_j \boldsymbol{\rho} \varpi_k  \partial_n\varpi_b,
\\
R_4^i:=& \epsilon^{amn} \partial_a \varpi^i  \partial_{m}\varpi_n+2\mathrm{e}^{\boldsymbol{\rho}} \Omega^a  \partial^i \varpi_a,
\\
R_5^i:= & 2\mathrm{e}^{-\boldsymbol{\rho}} \partial^a \boldsymbol{\rho} \partial^i v^k \partial_k \varpi_a,
 \\
 R_6^i:=  &- 2\mathrm{e}^{-\boldsymbol{\rho}} \partial^a \boldsymbol{\rho}  \varpi^m \partial^i( \partial_{m}v_a).
\end{split}
\end{equation}
\end{Lemma}
\begin{proof}
The equations \eqref{W0}, \eqref{W1}, and \eqref{W01} are derived in Luk-Speck's paper \cite{LS2}. Then, we only need to derive \eqref{W2}. The idea is mainly from Wang's paper \cite{WQRough}. Here, the difference is that we give a modified transport equation, which is more convenient for proving energy estimates and character energy estimates in this paper. Let us first calculate $\mathrm{curl} \mathbf{T}  \Omega^i$. Taking the operator $\mathrm{curl}$ on \eqref{W1}, then we have
\begin{equation}\label{ct}
\begin{split}
  \mathrm{curl} \mathbf{T}  \Omega^i &=
  \epsilon^{ijk} \partial_j \mathbf{T}  \Omega_k
  \\
  &=\epsilon^{ijk} \partial_j\left( -2\epsilon_{kmn}\mathrm{e}^{-\boldsymbol{\rho}}\partial^m v^a \partial^n \varpi_a+\epsilon^{amn} \mathrm{e}^{-\boldsymbol{\rho}} \partial_a v_k \partial_m \varpi_n \right).
\end{split}
\end{equation}
Operating derivatives on \eqref{ct}, we have
\begin{equation*}
\begin{split}
 \mathrm{curl} \mathbf{T}  \Omega^i=& -2\mathrm{e}^{-\boldsymbol{\rho}}\epsilon_{kmn}\epsilon^{ijk}\partial^m v^a \partial_{j}(\partial^n \varpi_a)+ \mathrm{e}^{-\boldsymbol{\rho}}\epsilon^{amn} \epsilon^{ijk} \partial_a v_k \partial^2_{mj}\varpi_n
 \\
 & + 2\mathrm{e}^{-\boldsymbol{\rho}}\epsilon_{kmn}\epsilon^{ijk}\partial_m v^a \partial_{n}\varpi_a \partial_j \boldsymbol{\rho}- \mathrm{e}^{-\boldsymbol{\rho}}\epsilon^{amn} \epsilon^{ijk} \partial_a v_k \partial_{m}\varpi_n \partial_j \boldsymbol{\rho}
 \\
 & -2\mathrm{e}^{-\boldsymbol{\rho}} \epsilon^{ijk} \epsilon^{kmn}  \partial_{j}(\partial^m v^a) \partial^n \varpi_a+\mathrm{e}^{-\boldsymbol{\rho}}\epsilon^{amn} \epsilon^{ijk} \partial^2_{aj} v_k \partial_m \varpi_n.
\end{split}
\end{equation*}
For the last term, we use \eqref{pw1} to get
\begin{equation*}
  \mathrm{e}^{-\boldsymbol{\rho}}\epsilon^{amn} \epsilon^{ijk} \partial^2_{aj} v_k \partial_m \varpi_n =\epsilon^{amn} \mathrm{e}^{-\boldsymbol{\rho}}\partial_a \omega^i \partial_m \varpi_n= \epsilon^{amn} \mathrm{e}^{-\boldsymbol{\rho}}\partial_a (\mathrm{e}^{\boldsymbol{\rho}}\varpi^i) \partial_m \varpi_n.
\end{equation*}
Hence, we can rewrite
\begin{equation}\label{ct0}
 \mathrm{curl} \mathbf{T}  \Omega^i= A + R^{(1)}_1+ R^{(1)}_2 +R^{(1)}_3+R^{(1)}_4,
\end{equation}
where
\begin{equation*}
\begin{split}
  R^{(1)}_1& =-2\mathrm{e}^{-\boldsymbol{\rho}}\epsilon_{kmn}\epsilon^{ijk}\partial^m v^a \partial_{j}(\partial^n \varpi_a)+ \mathrm{e}^{-\boldsymbol{\rho}}\epsilon^{amn}\epsilon^{ijk} \partial_a v_k \partial^2_{mj}\varpi_n,
\\
  R^{(1)}_2&=2\mathrm{e}^{-\boldsymbol{\rho}}\epsilon_{kmn}\epsilon^{ijk}\partial^m v^a \partial^{n}\varpi_a \partial_j \boldsymbol{\rho}- \mathrm{e}^{-\boldsymbol{\rho}}\epsilon^{amn}\epsilon^{ijk} \partial_a v_k \partial_{m}\varpi_n \partial_j \boldsymbol{\rho},
\\
  R^{(1)}_3&=\epsilon^{amn} \varpi^i  \partial_a \boldsymbol{\rho} \partial_{m}\varpi_n, \quad
   A =-2 \epsilon^{ijk} \epsilon_{kmn} \mathrm{e}^{-\boldsymbol{\rho}} \partial_{j} (\partial^m v^a) \partial^n \varpi_a,
\\
   R^{(1)}_4&= \epsilon^{amn} \partial_a \varpi^i  \partial_{m}\varpi_n.
   \end{split}
\end{equation*}
Above, the term $A$ is included the second-order derivatives of velocity, then we need to decomposition it. By using $\epsilon^{ijk} \epsilon_{kmn}=\delta^j_m \delta^i_n- \delta^i_m \delta^j_n$, we update $A$ as
\begin{equation}\label{A}
\begin{split}
  A& =-2\mathrm{e}^{-\boldsymbol{\rho}} (\delta^j_m \delta^i_n- \delta^i_m \delta^j_n) \partial_{j}(\partial^m v^a) \partial_n \varpi_a
  \\
  & =A_1+A_2,
\end{split}
\end{equation}
where
\begin{equation}\label{A12}
\begin{split}
  A_1:= -2\mathrm{e}^{-\boldsymbol{\rho}} \Delta v^a \partial^i \varpi_a, \quad A_2:= 2\mathrm{e}^{-\boldsymbol{\rho}} \partial^i( \partial_j v^a) \partial^j \varpi_a.
  \end{split}
\end{equation}
By Hodge decomposition, it tells us
\begin{equation*}
  \begin{split}
  \Delta v = \partial \mathrm{div} v - \mathrm{curl}^2 v
  & =- \partial \mathbf{T} \boldsymbol{\rho}+ \mathrm{curl} \big( \mathrm{e}^{\boldsymbol{\rho}}\varpi \big)
  \\
  & = -\mathbf{T} (\partial \boldsymbol{\rho})+ \partial v^k \partial_k \boldsymbol{\rho}++ \mathrm{curl} \big( \mathrm{e}^{\boldsymbol{\rho}}\varpi \big).
  \end{split}
\end{equation*}
Substituting the above equality to $A_1$, we can update $A_1$ as
\begin{equation}\label{A1F}
\begin{split}
  A_1 =&-2\mathrm{e}^{-\boldsymbol{\rho}} \big[ -\mathbf{T} (\partial^a \boldsymbol{\rho})+ \partial^a v^k \partial_k \boldsymbol{\rho}-\mathrm{e}^{\boldsymbol{\rho}} \epsilon^{ajk} \partial_j \boldsymbol{\rho} \varpi_k- \mathrm{e}^{2\boldsymbol{\rho}}\Omega^a \big] \partial^i \varpi_a
  \\
   = &\mathbf{T} \big( 2\mathrm{e}^{-\boldsymbol{\rho}}  \partial^a \boldsymbol{\rho} \partial^i \varpi_a  \big)- 2\mathrm{e}^{-\boldsymbol{\rho}} \partial^a \boldsymbol{\rho} \mathbf{T} ( \partial^i \varpi_a)
  \\
  & +  2\mathrm{e}^{-\boldsymbol{\rho}} \partial^a \boldsymbol{\rho}  \partial^i \varpi_a+R_2^{(2)}+R_3^{(2)}+R_4^{(2)}
  \\
  =& \mathbf{T} \big( 2\mathrm{e}^{-\boldsymbol{\rho}} \partial^a \boldsymbol{\rho}  \partial^i \varpi_a  \big)- 2\mathrm{e}^{-\boldsymbol{\rho}} \partial^a \boldsymbol{\rho}  \mathbf{T} (\partial^i \varpi_a)
  + R_2^{(2)}+R_3^{(2)}+R_4^{(2)}+R_2^{(3)},
\end{split}
\end{equation}
where
\begin{align*}
  R_2^{(2)}  &=-2\mathrm{e}^{-\boldsymbol{\rho}}   \partial^a v^k \partial_k \boldsymbol{\rho}  \partial^i \varpi_a, \quad R_4^{(2)}=2 \mathrm{e}^{\boldsymbol{\rho}} \Omega^a  \partial^i \varpi_a,
   \\
  R_3^{(2)} & =2 \epsilon^{ajk} \partial_j \boldsymbol{\rho} \varpi_k  \partial^i\varpi_a, \qquad \ \
  R_2^{(3)}=2\mathrm{e}^{-\boldsymbol{\rho}} \partial^a \boldsymbol{\rho}  \partial^i \varpi_a.
\end{align*}
In \eqref{A1F}, it remains for us to write $- 2\mathrm{e}^{-\boldsymbol{\rho}} \partial^a \boldsymbol{\rho}  \mathbf{T} (\partial^i \varpi_a)$ in a suitable way. Note $\mathbf{T} \varpi_a= \varpi^m \partial_m v_a$. Then we can calculate that
\begin{equation*}
\begin{split}
   \mathbf{T} \partial^i \varpi_a & = \partial^i (\mathbf{T} \varpi_a)- [\partial^i,  \mathbf{T}]\varpi_a
 = \partial^i \varpi^m \partial_m v_a + \varpi^m \partial^i(\partial_{m}v_a) - \partial^i v^k \partial_k \varpi_a.
\end{split}
\end{equation*}
We therefore derive that
\begin{equation}\label{A1R}
\begin{split}
 & - 2\mathrm{e}^{-\boldsymbol{\rho}}\partial^a \boldsymbol{\rho}  \mathbf{T}( \partial^i\varpi_a ) \\
  =& - 2\mathrm{e}^{-\boldsymbol{\rho}} \partial^a \boldsymbol{\rho} \big( \partial^i \varpi^m \partial_m v_a + \varpi^m \partial^i(\partial_{m}v_a) - \partial^i v^k \partial_k \varpi_a \big)
  \\
 = & R_2^{(4)}+R_5+R_6,
 \end{split}
\end{equation}
where\begin{align*}
       R_2^{(4)} &= - 2\mathrm{e}^{-\boldsymbol{\rho}} \partial^a \boldsymbol{\rho} \partial^i \varpi^m \partial_m v_a,
       \\
       R_5 &=  2\mathrm{e}^{-\boldsymbol{\rho}} \partial^a \boldsymbol{\rho} \partial^i v^k \partial_k \varpi_a,
       \\
       R_6 & =- 2\mathrm{e}^{-\boldsymbol{\rho}} \partial^a \boldsymbol{\rho}  \varpi^m \partial^i( \partial_{m}v_a),
     \end{align*}
Substituting \eqref{A1R} to \eqref{A1F}, we have
\begin{equation}\label{A1n}
\begin{split}
  A_1
  =\ & \mathbf{T} \big( 2\mathrm{e}^{-\boldsymbol{\rho}}   \partial^a \boldsymbol{\rho} \partial^i \varpi_a  \big)+R_2^{(4)}+R_6+R_5
  +R_2^{(2)}+R_3^{(2)}+R_4^{(2)}+R_2^{(3)},
\end{split}
\end{equation}
At last, let us consider $A_2$. By using chain rule, we obtain
\begin{equation}\label{A2}
\begin{split}
  A_2 = \ &  2\mathrm{e}^{-\boldsymbol{\rho}} \partial^i ( \partial_j v^a) \partial^j  \varpi_a
  \\
  = \ &  \partial^i (2\mathrm{e}^{-\boldsymbol{\rho}} \partial_{j} v^a \partial^j \varpi_a)+R_2^{(5)}+R_1^{(2)},
\end{split}
\end{equation}
where
\begin{align*}
  R_2^{(5)}  =2\mathrm{e}^{-\boldsymbol{\rho}} \partial^i \boldsymbol{\rho} \partial_{j}v^a \partial^j \varpi_a, \quad
  R_1^{(2)}  =-2\mathrm{e}^{-\boldsymbol{\rho}}  \partial_{j} v^a \partial^j(\partial^i  \varpi_a).
\end{align*}
The commutator rule tells us
\begin{align*}
 \mathbf{T} \mathrm{curl} \Omega^i   =\mathbf{T} \big( \epsilon^{ijk}\partial_j \Omega_k \big)
 &=\epsilon^{ijk} \partial_j \mathbf{T}  \Omega_k +\epsilon^{ijk} [\mathbf{T}, \partial_j]\Omega_k
 \\
 & = \mathrm{curl} \mathbf{T}  \Omega^i -\epsilon^{ijk} \partial_j v^m \partial_m \Omega_k
 \\
 & = \mathrm{curl} \mathbf{T}  \Omega^i + R^{(3)}_1,
\end{align*}
where
\begin{equation*}
  R^{(3)}_1=-\epsilon^{ijk} \partial_j v^m \partial_m \Omega_k.
\end{equation*}
Adding \eqref{A1n} and \eqref{A2} to \eqref{A}, and then combining with \eqref{ct0}, we can conclude that
\begin{align*}
 \mathbf{T} \mathrm{curl} \Omega^i
 =&  \mathrm{curl} \mathbf{T}  \Omega^i -\epsilon^{ijk} \partial_j v^m \partial_m \Omega_k.
 \\
 =& \mathbf{T}(2\mathrm{e}^{- \boldsymbol{\rho}}\partial^a \boldsymbol{\rho} \partial^i \varpi_a)+ \partial_j \big( 2 \mathrm{e}^{-\boldsymbol{\rho}} \partial_n v^a \partial^n \varpi_a \big)
 \\
 & + R_1+ R_2+ R_3+ R_4+ R_5+ R_6,
\end{align*}
where
\begin{equation*}
\begin{split}
R^i_1=& R_1^{(1)}+R_1^{(2)}+R_1^{(3)}
\\
 = & -2\mathrm{e}^{-\boldsymbol{\rho}}\epsilon_{kmn}\epsilon^{ijk} \partial^m v^a \partial_{j} (\partial^n\varpi_a)+ \mathrm{e}^{-\boldsymbol{\rho}}\epsilon^{amn} \epsilon^{ijk} \partial_a v_k \partial^2_{mj}\varpi_n
\\
&  -2\mathrm{e}^{-\boldsymbol{\rho}}  \partial_j v^a \partial^{ij} \varpi^b+ \epsilon^{ijk} \partial_j v^m \partial_m \Omega_k,
\\
R_2^i=& R_2^{(1)}+R_2^{(2)}+R_2^{(3)}+R_2^{(4)}+R_2^{(5)}
\\
=& \ 2\mathrm{e}^{-\boldsymbol{\rho}}\epsilon_{kmn}\epsilon^{ijk} \partial^m v^a \partial^{n}\varpi_a \partial_j \boldsymbol{\rho}- \mathrm{e}^{-\boldsymbol{\rho}}\epsilon^{amn} \epsilon^{ijk} \partial_a v_k \partial_{m}\varpi_n \partial_j \boldsymbol{\rho}
\\
& \ -2\mathrm{e}^{-\boldsymbol{\rho}}  \partial^a v^k \partial_k \boldsymbol{\rho}  \partial^i \varpi_i+ 2\mathrm{e}^{-\boldsymbol{\rho}} \partial^a \boldsymbol{\rho}  \partial^i\varpi_a
\\
& \ - 2\mathrm{e}^{-\boldsymbol{\rho}} \partial^a \boldsymbol{\rho} \partial^i \varpi^m \partial_m v_a+2\mathrm{e}^{-\boldsymbol{\rho}}  \partial^i \boldsymbol{\rho} \partial^{n} v^a \partial_n \varpi_a,
\\
R_3^i=&R_3^{(1)}+R_3^{(2)}=\epsilon^{amn} \varpi^i  \partial_a \boldsymbol{\rho} \partial_{m}\varpi_n+2 \epsilon^{ajk} \partial_j \boldsymbol{\rho} \varpi_k  \partial_n\varpi_b,
\\
R_4^i=& R_4^{(1)}+R_4^{(2)}=\epsilon^{amn} \partial_a \varpi^i  \partial_{m}\varpi_n+2\mathrm{e}^{\boldsymbol{\rho}} \Omega^a  \partial^i \varpi_a,
\\
R_5^i= & 2\mathrm{e}^{-\boldsymbol{\rho}} \partial^a \boldsymbol{\rho} \partial^i v^k \partial_k \varpi_a, \quad R^i_6:=  - 2\mathrm{e}^{-\boldsymbol{\rho}} \partial^a \boldsymbol{\rho}  \varpi^m\partial^i( \partial_{m}v_a).
\end{split}
\end{equation*}
At this stage, we complete the proof of \eqref{W2}.
\end{proof}
In the following, we will introduce some commutator estimates, product estimates.
\subsection{Classical commutator estimates and product estimates}
Firstly, let us give some commutator estimates and product estimates, which has been introduced in references \cite{KP,ST,WQRough}.
\begin{Lemma}\label{jh}\cite{KP}
	Let $\Lambda=(-\Delta)^{\frac{1}{2}}, s \geq 0$. Then for any scalar function $h, f$, we have
	\begin{equation*}
		\|\Lambda^s(hf)-(\Lambda^s h)f\|_{L_x^2} \lesssim \|\Lambda^{s-1}h\|_{L^{2}_x}\|\partial f\|_{L_x^\infty}+ \|h\|_{L^p_x}\|\Lambda^sf\|_{L_x^{q}},
	\end{equation*}
	where $\frac{1}{p}+\frac{1}{q}=\frac{1}{2}$.
\end{Lemma}
\begin{Lemma}\label{cj}\cite{KP}
	Let $a \geq 0$. For any scalar function $h$ and $f$, we have
	\begin{equation*}
		\|hf\|_{H_x^a} \lesssim \|h\|_{L_x^{\infty}}\| f\|_{H_x^a}+ \|f\|_{L_x^{\infty}}\| h \|_{H_x^a}.
	\end{equation*}
\end{Lemma}
\begin{Lemma}\label{jh0}\cite{KP}
Let $F(u)$ be a smooth function of $u$, $F(0)=0$ and $u \in L^\infty_x$. For any $s \geq 0$, we have
	\begin{equation*}
		\|F(u)\|_{H^s} \lesssim  \|u\|_{H^{s}}(1+ \|u\|_{L^\infty_x}).
	\end{equation*}
\end{Lemma}

\begin{Lemma}\label{ps}\cite{ST}
	Suppose that $0 \leq r, r' < \frac{n}{2}$ and $r+r' > \frac{n}{2}$. Then
	\begin{equation*}
		\|hf\|_{H^{r+r'-\frac{n}{2}}(\mathbb{R}^n)} \leq C_{r,r'} \|h\|_{H^{r}(\mathbb{R}^n)}\|h\|_{H^{r'}(\mathbb{R}^n)}.
	\end{equation*}
	If $-r \leq r' \leq r$ and $r>\frac{n}{2}$ then
\begin{equation*}
		\|hf\|_{H^{r'}(\mathbb{R}^n)} \leq C_{r,r'} \|h\|_{H^{r}(\mathbb{R}^n)}\|h\|_{H^{r'}(\mathbb{R}^n)}.
	\end{equation*}
\end{Lemma}

\begin{Lemma}\label{lpe}\cite{WQRough}
	Let $0 \leq \alpha <1 $. Then
	\begin{equation*}
		\|\Lambda^\alpha(hf)\|_{L^{2}_x(\mathbb{R}^3)} \lesssim \|h\|_{\dot{B}^{\alpha}_{\infty,2}(\mathbb{R}^3)}\|f\|_{L^2_x(\mathbb{R}^3)}+ \|h\|_{L^\infty_x(\mathbb{R}^3)}\|f\|_{\dot{H}^\alpha_x(\mathbb{R}^3)}.
	\end{equation*}
\end{Lemma}
\begin{Lemma}\label{wql}\cite{WQRough}
	Let $0 < \alpha <1 $. Then
	\begin{equation*}
		\|\Lambda^\alpha(f_1f_2f_3)\|_{L^{2}_x(\mathbb{R}^3)} \lesssim \|f_i\|_{H^{1+\alpha}}\textstyle{\prod_{j\neq i}}\|f_j\|_{H^1}.
	\end{equation*}
\end{Lemma}
\begin{Lemma}\cite{ML}\label{Miao}
Given $p \in [1,\infty]$
such that $p \geq m'$ with $m'$ the conjugate exponent
of $m$. Let $f, g$ and $\Phi$ belong to the suitable functional spaces. Then
\begin{equation*}
  \| \Phi * (hf)-h *(\Phi f) \|_{L^p_x} \lesssim \| x \Phi \|_{L^1} \| \nabla f\|_{L^\infty_x} \| h \|_{L^p}.
\end{equation*}
\end{Lemma}
\subsection{Useful lemmas}
Let us give some useful product estimates and commutator estimates, which play a crucial role in the paper.
\begin{Lemma}\label{LD}
	Let the metric $g$ is defined in \eqref{metricd}. If $f$ is the solution of
\begin{equation*}
\begin{split}
  \square_{g}f&=0, \quad t> \tau,
  \\
  f|_{t=\tau}&=-G(\tau,x), \quad \mathbf{T} f|_{t=\tau}=-B(\tau,x),
\end{split}
\end{equation*}
then
\begin{equation*}
  V(t,x)=\int^t_0 f(t,x;\tau)d\tau
\end{equation*}
is the solution of the linear wave equation
\begin{equation*}
\begin{split}
  \square_{g}V&=\mathbf{T}G+B,
  \\
  V|_{t=0}&=0, \quad \mathbf{T}V|_{t=0}=-G(0,x).
  \end{split}
\end{equation*}
\begin{proof}
We first note $\square_g= -\mathbf{T} \mathbf{T}+c^2_s \Delta$.  By calculating, we get
\begin{equation*}
\begin{split}
\mathbf{T} V &= \int^t_0 \mathbf{T} f(t,x;\tau)d\tau+f(t,x;t)
\\
&=\int^t_0 \mathbf{T} f(t,x;\tau)d\tau-G(t,x).
\end{split}
\end{equation*}
Taking the operator $\mathbf{T}$ again, we have
\begin{equation}\label{d1}
\begin{split}
-\mathbf{T}\mathbf{T} V &= -\int^t_0\mathbf{T} \mathbf{T} f(t,x;\tau)d\tau-\mathbf{T}f(t,x;t)+\mathbf{T}G
\\
&= -\int^t_0\mathbf{T} \mathbf{T} f(t,x;\tau)d\tau+B(t,x)+\mathbf{T}G.
\end{split}
\end{equation}
On the other hand,
\begin{equation}\label{d2}
  c^2_s \Delta V= \int^t_0 c^2_s \Delta f(t,x;\tau)d\tau.
\end{equation}
Adding \eqref{d1} and \eqref{d2}, we can derive that
\begin{equation*}
  \square_{g}V=\mathbf{T}G+B.
\end{equation*}
Furthermore, it satisfies
\begin{equation*}
  V|_{t=0}=0, \quad \mathbf{T}V|_{t=0}=-G(0,x).
\end{equation*}
We complete the proof of this lemma.
\end{proof}
\end{Lemma}
\begin{Lemma}\label{LPE}
	{Let $0 < \alpha <1 $. Let $\beta> \alpha$ and sufficiently close to $\alpha$. Then}
	\begin{equation}\label{HF}
		\| hf\|_{\dot{B}^{\alpha}_{\infty, 2}(\mathbb{R}^3)} \lesssim \|h\|_{L^{\infty}_x} \|f\|_{\dot{B}^{\alpha}_{\infty, 2}(\mathbb{R}^3)}+\|h\|_{C^{\beta}(\mathbb{R}^3)} \|f\|_{L^\infty}.
	\end{equation}
\end{Lemma}
\begin{proof}
Firstly, we have
\begin{equation*}
  \| hf\|^2_{\dot{B}^{\alpha}_{\infty, 2}(\mathbb{R}^3)}= \textstyle{\sum}_{j\geq -1} 2^{2j\alpha}\| \Delta_j (hf) \|^2_{L^\infty}.
\end{equation*}
By using Bony decomposition, we get
\begin{equation*}
  \Delta_j (hf)= \textstyle{\sum}_{|k-j|\leq 2 }\Delta_j (\Delta_kh S_{k-1}f)+\textstyle{\sum}_{|k-j|\leq 2 }\Delta_j (\Delta_k f S_{k-1}h)+\textstyle{\sum}_{k \geq j-1 }\Delta_j (\Delta_k h \Delta_{k}f).
\end{equation*}
By H\"older inequality, we derive that
\begin{equation}\label{HF1}
  \begin{split}
  & \textstyle{\sum}_{j\geq -1} 2^{2j\alpha}\| \Delta_j (hf) \|^2_{L^\infty}
  \\
   \lesssim & \ \textstyle{\sum}_{j\geq -1} 2^{2j\alpha} (\textstyle{\sum}_{|k-j|\leq 2 } \| \Delta_j \Delta_k h \|_{L^\infty} \| S_{k-1} f \|_{L^\infty})^2
   \\
   & \ + \textstyle{\sum}_{j\geq -1} 2^{2j\alpha} (\textstyle{\sum}_{|k-j|\leq 2 } \| \Delta_j \Delta_k f \|_{L^\infty} \| S_{k-1} h \|_{L^\infty})^2
  \\
    & \ + \textstyle{\sum}_{j\geq -1} 2^{2j\alpha} (\textstyle{\sum}_{k \geq j-1} \| \Delta_j \Delta_k f \|_{L^\infty} \| \Delta_{k} h \|_{L^\infty})^2
    \\
   \lesssim & \   \textstyle{\sum}_{j\geq -1} 2^{2j\alpha} \| \Delta_j h \|^2_{L^\infty} \| f \|^2_{L^\infty}+\|h\|^2_{L^{\infty}_x} \|f\|^2_{\dot{B}^{\alpha}_{\infty, 2}(\mathbb{R}^3)}
     \\
    & + \textstyle{\sum}_{j\geq -1}\| \Delta_j f \|^2_{L^\infty}  (\textstyle{\sum}_{k \geq j-1}2^{2(-\beta+\alpha)}2^{k\beta}  \| \Delta_{k} h \|_{L^\infty})^2 .
  \end{split}
\end{equation}
It suffices for us to give the bounds of the last right terms on \eqref{HF1}. Note
\begin{equation}\label{HF2}
\begin{split}
  \textstyle{\sum}_{j\geq -1} 2^{2j\alpha} \| \Delta_j h \|^2_{L^\infty} \| f \|^2_{L^\infty} \lesssim & \textstyle{\sum}_{j\geq -1} 2^{2j (-\beta+\alpha) } 2^{2j \beta }  \| \Delta_j f \|^2_{L^\infty} \| f \|^2_{L^\infty}
  \\
 \lesssim & \| f \|^2_{L^\infty} \{ 2^{2j (-\beta+\alpha) } \}_{l^1_{j}} (\{2^{j \beta }  \| \Delta_j h \|_{L^\infty}\}_{l^\infty_{j}} )^2
  \\
  \lesssim & \| f \|^2_{L^\infty} \| h \|^2_{\dot{B}^{\beta}_{\infty,\infty}}
  \lesssim  \| f \|^2_{L^\infty} \| h \|^2_{C^{\beta}}.
\end{split}
\end{equation}
We also note
\begin{equation}\label{HF3}
\begin{split}
& \textstyle{\sum}_{j\geq -1}\| \Delta_j f \|^2_{L^\infty}  (\textstyle{\sum}_{k \geq j-1}2^{2(-k\beta+j\alpha)}2^{k\beta}  \| \Delta_{k} h \|_{L^\infty})^2
\\
\lesssim & \ \textstyle{\sum}_{j\geq -1}\| \Delta_j f \|^2_{L^\infty}  (\textstyle{\sum}_{k \geq j-1}2^{2(-k\beta+j\alpha)}2^{k\beta}  \| \Delta_{k} h \|_{L^\infty})^2
\\
\lesssim & \ \| h \|^2_{\dot{B}^\beta_{\infty,\infty}}\textstyle{\sum}_{j\geq -1} 2^{2j(\alpha-\beta)}\| \Delta_j f \|^2_{L^\infty}
\\
\lesssim & \ \| f \|^2_{L^\infty}  \| h \|^2_{C^\beta}.
\end{split}
\end{equation}
Substituing \eqref{HF2} and \eqref{HF3} to \eqref{HF1}, we can get \eqref{HF}.
\end{proof}
The next commutator estimate is concerning to Riesz operators.
\begin{Lemma}\label{ceR}
	{Let $0 \leq \alpha <1$. Denote the Riesz operator $\mathbf{R}:=\partial^2(-\Delta)^{-1}$. Then}
	\begin{equation*}
		\| [\mathbf{R}, v \cdot \nabla]f\|_{\dot{H}^{\alpha}_x(\mathbb{R}^3)} \lesssim \| v\|_{\dot{B}^{1+\alpha}_{\infty, \infty}} \|f\|_{L^2_{x}(\mathbb{R}^3)}+ \| v\|_{\dot{B}^1_{\infty, \infty}} \|f\|_{\dot{H}^\alpha_x(\mathbb{R}^3)}.
	\end{equation*}
\end{Lemma}
\begin{proof}
By using paraproduct decomposition, we have
\begin{equation*}
\begin{split}
 \Delta_j [\mathbf{R}, v \cdot \nabla]f &= \textstyle{\sum}_{|k-j|\leq 2} \Delta_j \left[\mathbf{R} (\Delta_k v \cdot \nabla S_{k-1}f)-\Delta_k v \cdot \nabla \mathbf{R} S_{k-1}f \right]
 \\
 & \quad + \textstyle{\sum}_{|k-j|\leq 2} \Delta_j \left[\mathbf{R} (S_{k-1} v \cdot \nabla \Delta_{k}f)-S_{k-1} v \cdot \nabla \mathbf{R} \Delta_{k}f \right]
 \\
 & \quad + \textstyle{\sum}_{k\geq j-1} \Delta_j \left[\mathbf{R} (\Delta_{k}v \cdot \nabla \Delta_{k}f)-\Delta_{k}v \cdot \nabla \mathbf{R} \Delta_{k}f \right]
 \\
 & = B_1+B_2+B_3,
 \end{split}
\end{equation*}
where
\begin{equation*}
\begin{split}
  B_1&= \textstyle{\sum}_{|k-j|\leq 2} \Delta_j \left\{ \mathbf{R} (\Delta_k v \cdot \nabla S_{k-1}f)-\Delta_k v \cdot \nabla \mathbf{R} S_{k-1}f \right\},
  \\
  B_2&= \textstyle{\sum}_{|k-j|\leq 2} \Delta_j \left\{\mathbf{R} (S_{k-1} v \cdot \nabla \Delta_{k}f)-S_{k-1} v \cdot \nabla \mathbf{R} \Delta_{k}f \right\}
  \\
  B_3&=\textstyle{\sum}_{k\geq j-1} \Delta_j \left\{ \mathbf{R} (\Delta_{k}v \cdot \nabla \Delta_{k}f)-\Delta_{k}v \cdot \nabla \mathbf{R} \Delta_{k}f \right\}.
\end{split}
\end{equation*}

For $0\leq \alpha <1$, by H\"older's inequality and Bernstein's inequality, we can derive
\begin{equation}\label{B1}
\begin{split}
  \{2^{j\alpha} \| B_1 \|_{L^2_x} \}_{l^2_j} & \lesssim  \left\{ \textstyle{\sum}_{|k-j|\leq 2} (\| \nabla \mathbf{R} S_{k-1} f\|_{L^\infty_x}+\| \nabla S_{k-1} f\|_{L^\infty_x}) 2^{j\alpha} \|\Delta_j \Delta_k v\|_{L^2} \right\}_{l^2_j}
  \\
  & \lesssim \left\{ \textstyle{\sum}_{|k-j|\leq 2}  2^{j\alpha} 2^k \|\Delta_j \Delta_k v\|_{L_x^\infty} (\|\mathbf{R} S_{k-1} f \|_{L^2}+\| S_{k-1} f \|_{L^2}) \right\}_{l^2_j}
  \\
  & \lesssim \| v\|_{\dot{B}^{1+\alpha}_{\infty,\infty}} (\|\mathbf{R} f \|_{L^2_x}+\| f \|_{L^2_x}) \lesssim \| v\|_{\dot{B}^{1+\alpha}_{\infty,\infty}} \| f \|_{L^2_x}.
\end{split}
\end{equation}
By H\"older's inequality, we have
\begin{equation}\label{B3}
\begin{split}
  \{2^{j\alpha} \| B_3 \|_{L^2_x} \}_{l^2_j} &\lesssim \{2^{j\alpha}  \textstyle{\sum}_{k\geq j-1} 2^k  \| \Delta_k v\|_{L^\infty}\cdot 2^{-k}(\| \nabla \Delta_k f\|_{L^2}+\| \nabla \mathbf{R} \Delta_k f\|_{L^2})  \}_{l^2_j}
  \\
  &\lesssim \|v\|_{\dot{B}^{1}_{\infty,\infty}} \| f \|_{\dot{H}^\alpha}.
\end{split}
\end{equation}
Note
\begin{equation*}
  B_2= \textstyle{\sum}_{|k-j|\leq 2} \Delta_j [\mathbf{R} , S_{k-1} v \cdot]\Delta_{k} \nabla f.
\end{equation*}
By Lemma \ref{Miao} and H\"older's inequality, we get
\begin{equation}\label{B2}
\begin{split}
 \{2^{j\alpha} \| B_2 \|_{L_x^2} \}_{l^2_j}\lesssim & \left\{ \textstyle{\sum}_{|k-j|\leq 2} \|x\Phi\|_{L^1} \| \nabla S_{k-1} v \|_{L^\infty_x} 2^{j\alpha} \| \Delta_j \Delta_{k}f \|_{L^2_x} \right\}_{l^2_j}
 \\
 \lesssim & \left\{ \textstyle{\sum}_{|k-j|\leq 2} \| \nabla S_{k-1} v \|_{L^\infty_x} 2^{j\alpha} \| \Delta_j \Delta_{k}f \|_{L^2_x} \right\}_{l^2_j}
\\
 \lesssim & \|v\|_{\dot{B}^{1}_{\infty,\infty}}\| f \|_{\dot{H}^\alpha}.
\end{split}
\end{equation}
Here, we use the fact that ${\Phi}=\frac{x_ix_j}{|x|^2}2^{3j}\Psi(2^jx)$ and $\Psi$ is in Schwartz space. Gathering \eqref{B1}, \eqref{B3} and \eqref{B2} together, we have finished the proof of Lemma \ref{ceR}.
\end{proof}
\begin{Lemma}\label{YR}
	Let $2<s_1 \leq s_2$. Then
	\begin{equation*}
		\{ 2^{(s_1-1)j}\| [\Delta_j, \mathbf{T}]f\|_{\dot{H}^{s_2-s_1}_x(\mathbb{R}^3)} \}_{l^2_j} \lesssim  \| \partial v\|_{L^\infty_x }\|f\|_{\dot{H}^{s_2}_x(\mathbb{R}^3)}+\| v\|_{H^{s_2}}  \| f\|_{{H}^{1}}.
	\end{equation*}
\end{Lemma}
\begin{proof}
We first have
\begin{equation*}
  [\Delta_j, \mathbf{T}]f= [\Delta_j, v \cdot \nabla]f.
\end{equation*}
By using paraproduct decomposition, we have
\begin{equation*}
\begin{split}
  [\Delta_j, v \cdot \nabla]f &= \textstyle{\sum}_{|k-j|\leq 2}\big( \Delta_j  (\Delta_k v \cdot \nabla S_{k-1}f)-\Delta_k v \cdot \nabla  S_{k-1}\Delta_jf \big)
 \\
 & \quad + \textstyle{\sum}_{|k-j|\leq 2}\big( \Delta_j (S_{k-1} v \cdot \nabla \Delta_{k}f)-S_{k-1} v \cdot \nabla  \Delta_{k}\Delta_jf \big)
 \\
 & \quad + \textstyle{\sum}_{k\geq j-1} \big( \Delta_j  (\Delta_{k}v \cdot \nabla \Delta_{k}f)-\Delta_{k}v \cdot \nabla  \Delta_{k}\Delta_jf \big)
 \\
 & = A_1+A_2+A_3,
 \end{split}
\end{equation*}
where
\begin{equation*}
\begin{split}
  A_1&= \textstyle{\sum}_{|k-j|\leq 2}  \left\{  \Delta_j(\Delta_k v \cdot \nabla S_{k-1}f)-\Delta_k v \cdot \nabla  S_{k-1}\Delta_jf \right\},
  \\
  A_2&= \textstyle{\sum}_{|k-j|\leq 2}  \left\{\Delta_j (S_{k-1} v \cdot \nabla \Delta_{k}f)-S_{k-1} v \cdot \nabla \Delta_{k}\Delta_jf \right\}
  \\
  A_3&=\textstyle{\sum}_{k\geq j-1}  \left\{  \Delta_j(\Delta_{k}v \cdot \nabla \Delta_{k}f)-\Delta_{k}v \cdot \nabla  \Delta_{k}\Delta_jf \right\}.
\end{split}
\end{equation*}

Notice $\mathrm{supp} \hat{A}_1, \mathrm{supp} \hat{A}_2 \subseteq \{\xi\in \mathbb{R}^3: 2^{j-5} \leq |\xi|\leq 2^{j+5} \}$. By H\"older's inequality and commutator estimate, we can derive
\begin{equation}\label{A2}
\begin{split}
  \{2^{j(s_1-1)} \| A_2 \|_{\dot{H}^{s_2-s_1}_x} \}_{l^2_j} & \lesssim  \left\{ 2^{j(s_2-1)} \| A_2 \|_{L^2} \right\}_{l^2_j}
  \\
  & = \left\{ 2^{j(s_2-1)} \| [\Delta_j, S_{j-1} v \cdot \nabla] \Delta_j f\|_{L^2} \right\}_{l^2_j}
  \\
  & \lesssim \left\{ 2^{j(s_2-1)}  \|\nabla v\|_{L^\infty}  \|\Delta_j f\|_{L^2} \right\}_{l^2_j}\\
  & \lesssim  \|\nabla v\|_{L^\infty}  \| f\|_{\dot{H}^{s_2-1}}.
\end{split}
\end{equation}
By H\"older's inequality, we have
\begin{equation}\label{A1}
\begin{split}
  \{2^{j(s_1-1)} \| A_1 \|_{\dot{H}^{s_2-s_1}_x} \}_{l^2_j} & \lesssim  \left\{ 2^{j(s_2-1)} \| A_2 \|_{L^2} \right\}_{l^2_j}
  \\
  & = \left\{ 2^{j(s_2-1)} \|\Delta_j v \|_{L^3}  \| \nabla f\|_{L^6} \right\}_{l^2_j}
  \\
  & \lesssim  \| v \|_{\dot{B}^{s_2-1}_{3,2}}  \| \nabla f\|_{L^6}
  \\
  & \lesssim  \| v \|_{H^{s_2-\frac{1}{2}}}  \| f\|_{H^1} \lesssim  \| v \|_{H^{s_2}}  \| f\|_{H^1}.
\end{split}
\end{equation}
We rewrite $A_3$ by the form
\begin{equation*}
  A_3= \textstyle{\sum}_{k \geq j-1}  [\Delta_j , \Delta_k v \cdot \nabla]\Delta_{k}  f.
\end{equation*}
Using H\"older's inequality, we prove that
\begin{equation}\label{A3}
\begin{split}
  \{2^{j(s_1-1)} \| A_3 \|_{\dot{H}^{s_2-s_1}_x} \}_{l^2_j} & \lesssim  \|\nabla v\|_{L^\infty}  \| f\|_{\dot{H}^{s_2-1}}.
\end{split}
\end{equation}
Thus, we complete the proof of this lemma.
\end{proof}
\begin{Lemma}\label{ce}
	{Let $0 < \alpha <1 $. Then}
	\begin{equation*}
		\| [\Lambda^\alpha_x, v \cdot \nabla]f\|_{L^2_x(\mathbb{R}^3)} \lesssim \|\partial v\|_{\dot{B}^{0}_{\infty,2}} \|f\|_{\dot{H}^{\alpha}_{x}(\mathbb{R}^3)}.
	\end{equation*}
\end{Lemma}
\begin{proof}
By using paraproduct decomposition, we have
\begin{equation*}
\begin{split}
 \Delta_j [\Lambda^\alpha_x, v \cdot \nabla]f &= \textstyle{\sum}_{|k-j|\leq 2} \Delta_j \left[\Lambda^\alpha_x (\Delta_k v \cdot \nabla S_{k-1}f)-\Delta_k v \cdot \nabla \Lambda^\alpha_x S_{k-1}f \right]
 \\
 & \quad + \textstyle{\sum}_{|k-j|\leq 2} \Delta_j \left[\Lambda^\alpha_x (S_{k-1} v \cdot \nabla \Delta_{k}f)-S_{k-1} v \cdot \nabla \Lambda^\alpha_x \Delta_{k}f \right]
 \\
 & \quad + \textstyle{\sum}_{k\geq j-1} \Delta_j \left[\Lambda^\alpha_x (\Delta_{k}v \cdot \nabla \Delta_{k}f)-\Delta_{k}v \cdot \nabla \Lambda^\alpha_x \Delta_{k}f \right]
 \\
 & = V_1+V_2+V_3,
 \end{split}
\end{equation*}
where
\begin{equation*}
\begin{split}
  V_1&= \textstyle{\sum}_{|k-j|\leq 2} \Delta_j \left[\Lambda^\alpha_x (\Delta_k v \cdot \nabla S_{k-1}f)-\Delta_k v \cdot \nabla \Lambda^\alpha_x S_{k-1}f \right],
  \\
  V_2&= \textstyle{\sum}_{|k-j|\leq 2} \Delta_j \left[\Lambda^\alpha_x (S_{k-1} v \cdot \nabla \Delta_{k}f)-S_{k-1} v \cdot \nabla \Lambda^\alpha_x \Delta_{k}f \right]
  \\
  & = \textstyle{\sum}_{|k-j|\leq 2} \Delta_j [\Lambda^\alpha_x , S_{k-1} v \cdot \nabla]\Delta_{k}  f,
  \\
  V_3&=\textstyle{\sum}_{k\geq j-1} \Delta_j \left[\Lambda^\alpha_x (\Delta_{k}v \cdot \nabla \Delta_{k}f)-\Delta_{k}v \cdot \nabla \Lambda^\alpha_x \Delta_{k}f \right].
\end{split}
\end{equation*}

For $0<\alpha <1$, by H\"older's inequality and Bernstein's inequality, we can derive
\begin{equation}\label{V1}
\begin{split}
  \{ \| V_1 \|_{L^2_x} \}_{l^2_j} & \lesssim  \left\{ \textstyle{\sum}_{|k-j|\leq 2} (\| \nabla \Lambda^\alpha S_{k-1} f\|_{L^2_x}+2^{j\alpha}\| \nabla S_{k-1} f\|_{L^2_x}) \|\Delta_j \Delta_k v\|_{L^\infty} \right\}_{l^2_j}
  \\
  & \lesssim \left\{ \textstyle{\sum}_{|k-j|\leq 2}  2^k \| \|\Delta_j \Delta_k v\|_{L_x^\infty} \| S_{k-1}\Lambda^\alpha f \|_{L^2} \right\}_{l^2_j}
  \\
  & \lesssim \|\partial v\|_{\dot{B}^{0}_{\infty,2}} \| f \|_{\dot{H}^\alpha}. 
\end{split}
\end{equation}
By H\"older inequality, we have
\begin{equation}\label{V3}
  \{ \| V_3 \|_{L^2_x} \}_{l^2_j} \lesssim \| \nabla v\|_{L_x^\infty} \| f \|_{\dot{H}^\alpha}.
\end{equation}
By Lemma \ref{Miao} (by setting ${\Phi}=2^{j(\alpha+3)}\Psi(2^3x)$, $\Psi$ is in Schwartz space) and H\"older inequality, we get
\begin{equation}\label{V2}
\begin{split}
 \{ \| V_2 \|_{L_x^2} \}_{l^2_j}\lesssim & \left\{ \textstyle{\sum}_{|k-j|\leq 2} \|x\Phi\|_{L^1} \| \nabla S_{k-1} v \|_{L^\infty_x} \| \Delta_j \Delta_{k}f \|_{L^2_x} \right\}_{l^2_j}
 \\
 \lesssim & \left\{ \textstyle{\sum}_{|k-j|\leq 2} 2^{j\alpha} \| \nabla S_{k-1} v \|_{L^\infty_x} \| \Delta_j \Delta_{k}f \|_{L^2_x} \right\}_{l^2_j}
\\
 \lesssim & \|\nabla v\|_{\dot{B}^{0}_{\infty,\infty}}\| f \|_{\dot{H}^\alpha} \lesssim  \|\nabla v \|_{L_x^\infty} \| f \|_{\dot{H}^\alpha}.
\end{split}
\end{equation}
Gathering \eqref{V1}, \eqref{V3}, and \eqref{V2}, we can get Lemma \ref{ce}.
\end{proof}
\begin{Lemma}\label{yx}
Let $2<s_1\leq s_2$. Let $\mathbf{g}$ be a Lorentz metric and $\mathbf{g}^{00}=-1$. We have
\begin{equation}\label{YX}
\begin{split}
 \left\{ 2^{(s_1-1)j}\|[\square_{\mathbf{g}}, \Delta_j]f\|_{\dot{H}^{s_2-s_1}} \right\}_{l^2_j} \lesssim & \ \| d f\|_{L_x^\infty} \|d \mathbf{g}\|_{\dot{H}^{s_2-1}}+\| d \mathbf{g}\|_{L_x^\infty}\|d f\|_{\dot{H}^{s_2-1}}.
\end{split}
\end{equation}
\end{Lemma}
	\begin{proof}
Let ${R}_j=[\square_{\mathbf{g}}, \Delta_j]f$ and $S_j=\textstyle\sum_{j'\leq j} \Delta_{j'}$. Note $\mathbf{g}^{00}=-1$. Then we can decompose ${R}_j$ as
\begin{equation*}
	{R}_j={E}_j+{A}_j+{G}_j,
\end{equation*}
where
\begin{equation*}
\begin{split}
E_j=&\Delta_j(\mathbf{g}^{\alpha i}\partial_{\alpha i }f)-S_{j}(\mathbf{g}^{\alpha i})\Delta_j(\partial_{\alpha i }f),
\\
A_j=&\left( S_{j}(\mathbf{g}^{\alpha i})-\mathbf{g}^{\alpha i}\right)\Delta_j(\partial_{\alpha i}f),
\\
G_j=&\Delta_j\big( c^{-1}_s \partial_\alpha(c_s\mathbf{g}^{\alpha\beta})\big)\partial_{\beta}f.
\end{split}
\end{equation*}
Consider
\begin{equation*}
  \Delta_{k} \Delta_j f=0, \quad \mathrm{if} \ |k-j|\geq 8.
\end{equation*}
Note the support of $\hat{E}_j$ being in the set $\{\xi: 2^{j-5} \leq |\xi|\leq 2^{j+10}\}$, then
\begin{equation*}
\begin{split}
  &\| E_j \|_{\dot{H}^{s_2-s_1}} \lesssim 2^{j(s_2-s_1)}\| E_j\|_{L^2_x}.
\end{split}
\end{equation*}
For $A_j$, we first have
\begin{equation*}
   A_j = \sum_{j'>j} \Delta_{j'} (\mathbf{g}^{\alpha i})\Delta_j(\partial_{\alpha i}f).
\end{equation*}
By classical product estimates, we derive that
\begin{equation}\label{AJ}
\begin{split}
  \| A_j \|_{\dot{H}^{s_2-s_1}}
  & \lesssim  \sum_{j'>j} ( \| \Delta_{j'} (\mathbf{g}^{\alpha i}) \|_{L^\infty} \| \Delta_j(\partial_{\alpha i}f) \|_{\dot{H}^{s_2-s_1}}+ \| \Delta_{j'} (\mathbf{g}^{\alpha i}) \|_{\dot{H}^{s_2-s_1}} \| \Delta_j(\partial_{\alpha i}f) \|_{L^\infty}).
\end{split}
\end{equation}
By Bernstein's ineuality, we can update \eqref{AJ} as
\begin{equation*}
\begin{split}
 & \{  2^{j(s_1-1)}\| A_j \|_{\dot{H}^{s_2-s_1}}  \}_{l^2_j}
 \\
  \lesssim  & \{ \sum_{j'>j} ( \| \Delta_{j'} (\nabla \mathbf{g}^{\alpha i}) \|_{L^\infty} 2^{-j'+j}\| \Delta_j(\partial f) \|_{\dot{H}^{s_2-1}}
  \\
  & + 2^{j(s_1-1)}\| \Delta_{j'} (\mathbf{g}^{\alpha i}) \|_{\dot{H}^{s_2-s_1}}\cdot 2^j\| \Delta_j(\partial f) \|_{L^\infty} ) \}_{l^2_j}.
\end{split}
\end{equation*}
We set
\begin{equation*}
  \begin{split}
  N_1&=\{ \sum_{j'>j}  \| \Delta_{j'} (\nabla \mathbf{g}^{\alpha i}) \|_{L^\infty} 2^{-j'+j}\| \Delta_j(\partial f) \|_{\dot{H}^{s_2-1}} \}_{l^2_j},
  \\
  N_2&=\{ \sum_{j'>j} 2^{j(s_1-1)}\| \Delta_{j'} (\mathbf{g}^{\alpha i}) \|_{\dot{H}^{s_2-s_1}}\cdot 2^j\| \Delta_j(\partial f) \|_{L^\infty} ) \}_{l^2_j}.
  \end{split}
\end{equation*}
By direct calculation, we have
\begin{equation*}
  N_1 \lesssim \| d \mathbf{g} \|_{L^\infty_x} \|d f\|_{\dot{H}^{s_2-1}_x}, \quad N_2 \lesssim \| d f \|_{L^\infty_x} \|d \mathbf{g}\|_{\dot{H}^{s_2-1}_x}.
\end{equation*}
In a result, we get
\begin{equation}\label{A9}
 \left\{   2^{(s_1-1)j} \| A_j \|_{\dot{H}^{s_2-s_1}} \right\}_{l^2_j} \lesssim \| d \mathbf{g} \|_{L^\infty_x} \|d f\|_{\dot{H}^{s_2-1}_x}+ \| d f \|_{L^\infty_x} \|d \mathbf{g}\|_{\dot{H}^{s_2-1}_x}.
\end{equation}
For $G_j$, by phase decomposition, we note
\begin{equation*}
\begin{split}
  \Delta_k G_j=& \textstyle{\sum_{|k-m|\leq 2}} \Delta_k \left( \Delta_m\big(\Delta_j( c^{-1}_s \partial_\alpha(c_s\mathbf{g}^{\alpha\beta}))\big) S_{m-1}\partial_{\beta}f \right)
  \\
  &+\textstyle{\sum_{|k-m|\leq 2}} \Delta_k \left( S_{m-1}\big(\Delta_j( c^{-1}_s \partial_\alpha(c_s\mathbf{g}^{\alpha\beta}))\big) \Delta_m\partial_{\beta}f \right)
  \\
  &+\Delta_k \left( \tilde{\Delta}_k\big(\Delta_j( c^{-1}_s \partial_\alpha(c_s\mathbf{g}^{\alpha\beta}))\big)\tilde{\Delta}_k\partial_{\beta}f \right),
\end{split}
\end{equation*}
where $\tilde{\Delta}_k=\Delta_{k-1}+\Delta_k+\Delta_{k+1}$.
By H\"older's inequality, we can deduce
\begin{equation*}
\begin{split}
  \|\Delta_k G_j \|_{L^2_x} \lesssim & \ \|\Delta_k \Delta_j\big( c^{-1}_s \partial_\alpha(c_s\mathbf{g}^{\alpha\beta})\big)\|_{L^2_x} (\| S_k\partial_{\beta}f \|_{L^\infty_x}+\| \Delta_k \partial_{\beta}f \|_{L^\infty_x})
  \\
  & +\| S_k\big(\Delta_j( c^{-1}_s \partial_\alpha(c_s\mathbf{g}^{\alpha\beta})) \|_{L^\infty_x} \|\Delta_k\partial_{\beta}f \|_{L^2_x}.
\end{split}
\end{equation*}
By H\"older's inequality, we can deduce that
\begin{equation*}
\begin{split}
  \| G_j \|_{L^2_x} \lesssim  \| \Delta_j\big( c^{-1}_s \partial_\alpha(c_s\mathbf{g}^{\alpha\beta})\big)\|_{L^2_x} \| d f \|_{L^\infty_x}.
\end{split}
\end{equation*}
In a result, we can get
\begin{equation*}
  \| G_j \|_{\dot{H}^{s_2-s_1}} \lesssim \| \Delta_j\big( c^{-1}_s \partial_\alpha(c_s\mathbf{g}^{\alpha\beta})\big)\|_{L^2_x} \| d f \|_{L^\infty_x}+ \|  c^{-1}_s \partial_\alpha(c_s\mathbf{g}^{\alpha\beta}) \|_{L^\infty_x} \|\Delta_j \partial_{\beta}f \|_{L^2_x}.
\end{equation*}
\begin{equation}\label{G9}
 \left\{   2^{(s_1-1)j} \| G_j \|_{\dot{H}^{s_2-s_1}} \right\}_{l^2_j} \lesssim \| d f\|_{L^\infty_x} \|d\mathbf{g} \|_{\dot{H}^{s_2-1}_x}+ \| d\mathbf{g}\|_{L^\infty_x} \|df \|_{\dot{H}^{s_2-1}_x}.
\end{equation}
It remains for us to give a bound for $E_j$. By phase decomposition, we have
\begin{equation*}
  \begin{split}
  E_j= & \Delta_j(\mathbf{g}^{\alpha i}\partial_{\alpha i }f)-S_{j}(\mathbf{g}^{\alpha i})\Delta_j(\partial_{\alpha i }f)
  \\
   =& \Delta_j\left(S_j(\mathbf{g}^{\alpha i})\partial_{\alpha i }f \right)-S_{j}(\mathbf{g}^{\alpha i})\Delta_j(\partial_{\alpha i }f)
  \\
  &+ \Delta_j\left(\Delta_j(\mathbf{g}^{\alpha i})S_j(\partial_{\alpha i }f) \right) + \Delta_j \left(\tilde{\Delta}_j(\mathbf{g}^{\alpha i})\tilde{\Delta}_j(\partial_{\alpha i }f) \right),
  \end{split}
\end{equation*}
where $\tilde{\Delta}_j={\Delta}_{j-1}+{\Delta}_j+{\Delta}_{j+1}$.
By commutator estimates, we have
\begin{equation*}
\begin{split}
  \| E_j \|_{L^2_x} &= \| [{\Delta}_j, {S}_j \mathbf{g}^{\alpha i} ]\Delta_j(\partial_{\alpha i })f \|_{L^2_x}+\| \Delta_j\left(\Delta_j(\mathbf{g}^{\alpha i})S_j(\partial_{\alpha i }f) \right)\|_{L^2_x} + \| \Delta_j \left(\tilde{\Delta}_j(\mathbf{g}^{\alpha i})\tilde{\Delta}_j(\partial_{\alpha i }f) \right)\|_{L^2_x}
  \\
  & \lesssim  \|\partial \mathbf{g}\|_{L^\infty_x} \|\Delta_j(d f) \|_{L^2_x}+\|\partial f\|_{L^\infty_x} \|\Delta_j(d \mathbf{g}) \|_{L^2_x}.
\end{split}
\end{equation*}
In a result, we get
\begin{equation}\label{E9}
 \left\{   2^{(s_1-1)j} \| E_j \|_{\dot{H}^{s_2-s_1}} \right\}_{l^2_j} \lesssim \| \partial \mathbf{g}\|_{L^\infty_x} \|d f\|_{\dot{H}^{s_2-1}_x}+\| \partial f\|_{L^\infty_x} \|d \mathbf{g}\|_{\dot{H}^{s_2-1}_x}.
\end{equation}
By combining \eqref{A9}, \eqref{G9}, and \eqref{E9}, we can get \eqref{YX}.
\end{proof}
\begin{Lemma}\label{yux}
Let $\eta$ be defined in \eqref{etad}. Let $\mathcal{D}$ and $Q$ be stated in \eqref{DDi} and \eqref{fc} respectively. Let $s \in (2,\frac52)$ and $2<s_0<s$. Then the following estimates
\begin{equation}\label{YYE}
  \| \mathcal{D}, Q\|_{ H^{s-1}} \lesssim \| d\boldsymbol{\rho}, dv \|_{L^\infty} (\| \boldsymbol{\rho}, v \|_{H^{s}}+ \| \varpi \|_{H^{\frac32+}}),
\end{equation}
and
\begin{equation}\label{eta}
  \| \mathbf{T} \eta \|_{H^s} \lesssim \| \varpi\|_{H^{\frac{3}{2}+}} \|v, \boldsymbol{\rho} \|_{H^{s}},
\end{equation}
hold. Moreover, the function $\eta$ satisfies
\begin{equation}\label{eee}
  \| \eta \|_{H^{s_0+1}}  \lesssim (1+\| \boldsymbol{\rho} \|_{H^{s_0}_x}) \| \varpi \|_{H_x^{s_0}}.
\end{equation}
\end{Lemma}
\begin{proof}
We recall the expression of $\mathcal{D}$
\begin{equation*}
  \mathcal{D}=-3c_s^{-1}c'_sg^{\alpha \beta} \partial_\alpha \boldsymbol{\rho} \partial_\beta \boldsymbol{\rho}+2 \textstyle{\sum_{1 \leq a < b \leq 3} }\big\{ \partial_a v^a \partial_b v^b-\partial_a v^b \partial_b v^a \big\},
\end{equation*}
which is described in \eqref{DDi}. Using Lemma \ref{cj}, we can get
\begin{equation}\label{Ds}
  \| \mathcal{D} \|_{H^{s-1}} \lesssim  \| d\boldsymbol{\rho} \|_{L^\infty} \| d\boldsymbol{\rho} \|_{H^{s-1}}+ \| \partial v \|_{L^\infty} \| \partial v \|_{H^{s-1}}.
\end{equation}
Note
\begin{equation*}
\begin{split}
  Q^i=& \ 2e^{\boldsymbol{\rho}} \epsilon^{iab} \mathbf{T} v_a \varpi_b-\left( 1+c_s^{-1}c'_s\right)g^{\alpha \beta} \partial_\alpha \boldsymbol{\rho} \partial_\beta v^i
  \\
  =&\ -2e^{\boldsymbol{\rho}} \epsilon^{iab} c^2_s \partial^a \boldsymbol{\rho}  \varpi_b-\left( 1+c_s^{-1}c'_s\right)g^{\alpha \beta} \partial_\alpha \boldsymbol{\rho} \partial_\beta v^i.
\end{split}
\end{equation*}
By Lemma \ref{cj}, we can deduce that
\begin{equation}\label{Qi}
\begin{split}
  \|{Q} \|_{H^{s-1}}&  \lesssim (\|d\boldsymbol{\rho} \|_{L^\infty} + \|dv \|_{L^\infty}) ( \| d \boldsymbol{\rho} \|_{H^{s-1}}+ \| d v \|_{H^{s-1}}).
\end{split}
\end{equation}
By using \eqref{fc0}, then we have
\begin{equation}\label{fa}
  \| d v\|_{H^{s-1}}+\| d \boldsymbol{\rho}\|_{H^{s-1}} \lesssim \| \partial v\|_{H^{s-1}}+\| \partial \boldsymbol{\rho}\|_{H^{s-1}}\lesssim \|  v\|_{H^{s}}+\| \boldsymbol{\rho}\|_{H^{s}}.
\end{equation}
Combining \eqref{fa}, \eqref{Qi}, and \eqref{Ds}, we have proved \eqref{YYE}. It remains for us to prove \eqref{eta} and \eqref{eee}. By \eqref{etad} and Sobolev imbedding, and elliptic estimates, we can derive that
\begin{equation}\label{ETAs}
\begin{split}
  \| \eta \|_{H^{s}} \leq  \| \mathrm{e}^{\boldsymbol{\rho}} \mathrm{curl} \varpi \|_{H_x^{s-2}} &\leq  (1+\|\boldsymbol{\rho}\|_{H^{\frac{3}{2}+}} )\| \varpi \|_{H_x^{s-1}},
\end{split}
\end{equation}
and
\begin{equation}\label{ETAs0}
\begin{split}
  \| \eta \|_{H^{s_0+1}} \lesssim (1+\| \boldsymbol{\rho} \|_{H^{2}_x})\| \varpi \|_{H_x^{s_0}}.
\end{split}
\end{equation}
Let us give a bound for $\mathbf{T}\eta$. By using \eqref{etad}, we get
\begin{equation*}
  -\Delta \eta=\mathrm{e}^{\boldsymbol{\rho}}\mathrm{curl}\varpi=\mathrm{e}^{2\boldsymbol{\rho}}\Omega.
\end{equation*}
 Then we have
\begin{equation}\label{bs}
  -\Delta(\mathbf{T}\eta^i)  = \mathbf{T} \Omega^i  - \Delta v^m \partial_m \eta^i-2 \partial_j v^m \partial^j(\partial_m \eta).
\end{equation}
By \eqref{bs}, \eqref{W1}, and Lemma \ref{ps}, we have
\begin{equation*}
\begin{split}
 \| \mathbf{T} \eta \|_{H_x^s} & \lesssim \| \partial v \cdot \partial \varpi \|_{H_x^{s-2}}+ \| \partial^2 v \cdot \partial \eta \|_{H_x^{s-2}}+\| \partial^2 \eta \cdot \partial v \|_{H_x^{s-2}}
 \\
 & \lesssim \| \partial v \|_{H^{s-1}} \| \partial \varpi \|_{H^{\frac{1}{2}+}} +\| \partial^2 v \|_{H^{s-2}} \| \partial \eta \|_{H^{\frac{3}{2}+}}+\| \partial v \|_{H^{s-1}} \| \partial^2 \eta \|_{H^{\frac{1}{2}+}}
 \\
  & \lesssim \| \varpi \|_{H^{\frac{3}{2}+}} \|v, \boldsymbol{\rho} \|_{H^{s}}.
\end{split}
\end{equation*}
At this stage, we complete the proof of the lemma.
\end{proof}
\section{Basic energy estimates and stability theorem}
In this part, we will prove the energy estimates and stability theorem.
Firstly, we use the hyperbolic system to give the basic energy estimate for density and velocity.
\begin{theorem}\label{dv}{(Basic Energy estimates for velocity and density)}
	Let $v$ and $\boldsymbol{\rho}$ be a solution of \eqref{fc0}. Let $\varpi$ be defines in \eqref{pw1}. Then for any $s\geq 0$, we have
\begin{equation}\label{E2}
 \| \boldsymbol{\rho}\|_{H^s}+ \|v\|_{H^s} \lesssim  \left( \|\boldsymbol{\rho}_0\|_{H^s}+ \|v_0\|_{H^s} \right) \exp( {\int^t_0} \|dv, d\boldsymbol{\rho}\|_{L^\infty_x}d\tau), \quad t \in [0,T].
\end{equation}
\end{theorem}
\begin{proof}
Using Lemma \ref{sh}, we have
\begin{equation*}
 \partial_t U+ \sum_{i=1}^3 A_i(U) \partial_{i}U=0,
\end{equation*}
where $U=(\boldsymbol{\rho}, v^1, v^2, v^3)^{\mathrm{T}}$. Operating $\Lambda^a$ on the above equality, we have
\begin{equation}\label{Uy}
  \begin{split}
  \partial_t (\Lambda^a U)+ \sum_{i=1}^3 A_i(U) \partial_{i}(\Lambda^a U)=-\sum_{i=1}^3 [\Lambda^a, A_i(U)] \partial_{i} U.
  \end{split}
\end{equation}
Multiplying $\Lambda^a U$, integrating it by parts on \eqref{Uy}, and using Lemma \ref{jh}, it yields to
\begin{equation*}
  \frac{d}{dt} \| \Lambda^a U \|^2_{L^2} \lesssim  \int_{\mathbb{R}^3}\|dU\|_{L^\infty_x}\| \Lambda^a U \|^2_{L^2} d\tau.
\end{equation*}
Integrating it on $[0,t]$ and summing $a=0$ with $a=s$, we can prove
	\begin{equation*}
	\| U(t)\|^2_{H^{s}}  \lesssim \| U(0)\|^2_{H^{s}}\exp \big({\int^t_0} \|dU\|_{L^\infty_x}d\tau \big).
	\end{equation*}
That is to say,
\begin{equation*}
 \| \boldsymbol{\rho}\|_{H^s}+ \|v\|_{H^s} \lesssim  \left( \|\boldsymbol{\rho}_0\|_{H^s}+ \|v_0\|_{H^s} \right) \exp \big( {\int^t_0} \|dv, d\boldsymbol{\rho}\|_{L^\infty_x}d\tau \big), \quad t \in [0,T].
\end{equation*}
\end{proof}
Secondly, by utilizing Lemma \ref{PW}, we can give basic energy estimates for vorticity.
\begin{theorem}\label{ve}{(Basic Energy estimates for vorticity)}
Let $v$ and $\boldsymbol{\rho}$ be a solution of \eqref{fc0}. Let $\varpi$ be defined in \eqref{pw1} and $\varpi$ satisfy \eqref{W0}. For $s_0>2$, the following energy estimates hold:
\begin{equation}\label{WH}
\|\varpi(t)\|_{H^{s_0}_x}  \lesssim  \| \varpi_0\|_{H^{s_0}_x}\exp \big( {\int^t_0} (\|d v, d\boldsymbol{\rho}\|_{L^\infty_x}+\|\partial v\|_{\dot{B}^{s_0-2}_{\infty,2}})  d\tau \big),
\end{equation}
and
\begin{equation}\label{WL}
  \|\varpi\|_{H^2} \lesssim \|(v_0, \boldsymbol{\rho}_0, \varpi_0)\|_{H^2} \exp(5\int^t_0 \| \partial v, \partial \boldsymbol{\rho} \|_{L^\infty_x}d\tau).
\end{equation}
\end{theorem}
\begin{proof}
We divide the proof into several steps.

\textbf{Step 1: lower-order energy estimate}. Recall that $\varpi$ satisfies the transport equation \eqref{W0}. Multiplying $\varpi$ and integrating it on $[0,t]\times \mathbb{R}^3$, we can get
\begin{equation}\label{w0}
  \|\varpi(t)\|^2_{L^2_x} \lesssim \|\varpi_0\|^2_{L^2_x}+\int^t_0\| \partial v\|_{L^\infty_x}\|\varpi\|^2_{L^2_x}d\tau.
\end{equation}
Using Gronwall's inequality, we have
\begin{equation}\label{WOE}
  \|\varpi\|^2_{L^2} \lesssim \|\varpi_0\|^2_{L^2} \exp \big( {\int^t_0} \|dv\|_{L^\infty_x}d\tau \big).
\end{equation}
By elliptic estimates, we obtain that
\begin{equation}\label{w1}
  \| \varpi \|_{\dot{H}^2_x} \leq\| \mathrm{curl} \varpi \|_{\dot{H}^1_x}+\| \mathrm{div} \varpi\|_{\dot{H}^1_x}.
\end{equation}
By Lemma \ref{PW}, and using H\"older inequality, we can prove
\begin{equation*}
  \|\mathrm{div} \varpi\|_{\dot{H}^1_x} = \|\varpi \cdot \partial \boldsymbol{\rho}  \|_{\dot{H}^1_x}  \leq  C\|\varpi \|_{H_x^{\frac{7}{4}}}\|\partial \boldsymbol{\rho}\|_{H^1_x}.
\end{equation*}
By interpolation formula and Young's inequality, we can update it by
\begin{equation}\label{dw1}
\begin{split}
  \|\mathrm{div} \varpi\|_{\dot{H}^1_x}  \leq & C(\|\varpi \|^{\frac{1}{4}}_{L_x^{2}}\|\varpi \|^{\frac{3}{4}}_{{H}_x^{2}}) \|\partial \boldsymbol{\rho}\|_{H^1_x} \\
  \leq & C\|\varpi \|_{L_x^{2}} \|\partial \boldsymbol{\rho}\|^4_{H^1_x}+ \frac{c_1}{100} \|\varpi \|_{{H}_x^{2}}.
\end{split}
\end{equation}
Using $\mathrm{curl} \varpi=\mathrm{e}^{\boldsymbol{\rho}} \Omega$, we have
\begin{equation*}
  \|\mathrm{curl} \varpi\|_{\dot{H}_x^1} = \|\mathrm{curl} \mathrm{curl} \varpi \|_{L_x^2}  = \|\mathrm{curl} (\mathrm{e}^{\boldsymbol{\rho}} \Omega) \|_{L_x^2}.
\end{equation*}
By H\"older's inequality, we get
\begin{equation}\label{yxe}
  \| \partial \boldsymbol{\rho} \partial \varpi \|_{L_x^2}\leq C\|\varpi \|_{H_x^{\frac{7}{4}}}\|\partial \boldsymbol{\rho}\|_{H^1_x}\leq C\|\varpi \|_{L_x^{2}} \|\partial \boldsymbol{\rho}\|^4_{H^1_x}+ \frac{1}{100} \|\varpi \|_{{H}_x^{2}}.
\end{equation}
Thus, we have
\begin{equation}\label{cw1}
\begin{split}
  \|\mathrm{curl} \varpi\|_{\dot{H}_x^1} \leq & C\| \partial \boldsymbol{\rho} \partial \varpi \|_{L_x^2}+ C\|\mathrm{curl} \Omega \|_{L_x^2}
  \\
  \leq & C\|\varpi \|_{L_x^{2}} \|\partial \boldsymbol{\rho}\|^4_{H^1_x}+ \frac{1}{100} \|\varpi \|_{{H}_x^{2}}+ C\|\mathrm{curl} \Omega \|_{L_x^2},
  \end{split}
\end{equation}
and
\begin{equation}\label{cw2}
\begin{split}
  \|\mathrm{curl} \varpi\|_{\dot{H}_x^1} \geq &  \|\mathrm{curl} \Omega \|_{L_x^2}-C\| \partial \boldsymbol{\rho} \partial \varpi \|_{L_x^2}
  \\
  \geq & \|\mathrm{curl} \Omega \|_{L_x^2}-C\|\varpi \|_{L_x^{2}} \|\partial \boldsymbol{\rho}\|^4_{H^1_x}- \frac{1}{100} \|\varpi \|_{{H}_x^{2}}.
  \end{split}
\end{equation}
Note \eqref{W2}. That is,
\begin{equation*}
\begin{split}
& \mathbf{T} \big( \mathrm{curl} \Omega^i -2\mathrm{e}^{-\boldsymbol{\rho}} \partial_a \boldsymbol{\rho}  \partial^i \varpi^a \big)
= \ \partial^i \big( 2 \mathrm{e}^{-\boldsymbol{\rho}}  \partial_n v_a \partial^n \varpi^b \big) + \sum^6_{j=1}R^{i}_j.
\end{split}
\end{equation*}
Multiplying $\mathrm{curl} \Omega_i -2\mathrm{e}^{-\boldsymbol{\rho}} \partial_a \boldsymbol{\rho}  \partial_i \varpi^a$ on the above equality and integrating it on $\mathbb{R}^3$, we can derive that
\begin{equation}\label{oneW}
\begin{split}
  \frac{d}{dt} \|\mathrm{curl} \Omega -2\mathrm{e}^{-\boldsymbol{\rho}} \partial_a \boldsymbol{\rho}  \partial \varpi^a\|^2_{L^2}=& \int_{\mathbb{R}^3} \partial^i \big( 2 \mathrm{e}^{-\boldsymbol{\rho}}  \partial_n v_a \partial^n \varpi^b \big) \cdot \big( \mathrm{curl} \Omega_i -2\mathrm{e}^{-\boldsymbol{\rho}} \partial_a \boldsymbol{\rho}  \partial_i \varpi^a \big) dx
  \\
  & + \sum^6_{j=1} \int_{\mathbb{R}^3} R^{i}_j \cdot \big( \mathrm{curl} \Omega_i -2\mathrm{e}^{-\boldsymbol{\rho}} \partial_a \boldsymbol{\rho}  \partial_i \varpi^a \big) dx.
\end{split}
\end{equation}
Set
\begin{equation}\label{K1}
  K_1=\int_{\mathbb{R}^3} \partial^i \big( 2 \mathrm{e}^{-\boldsymbol{\rho}}  \partial_n v_a \partial^n \varpi^b \big) \cdot \big( \mathrm{curl} \Omega_i -2\mathrm{e}^{-\boldsymbol{\rho}} \partial_a \boldsymbol{\rho}  \partial_i \varpi^a \big) dx,
\end{equation}
and
\begin{equation}\label{K2}
  K_2=\sum^6_{j=1} \int_{\mathbb{R}^3} R^{i}_j \cdot \big( \mathrm{curl} \Omega_i -2\mathrm{e}^{-\boldsymbol{\rho}} \partial_a \boldsymbol{\rho}  \partial_i \varpi^a \big) dx.
\end{equation}
Integrating by parts, we can update $K_1$ as
\begin{equation}\label{w3}
\begin{split}
K_1
 =&\int_{\mathbb{R}^3}  \big( 2 \mathrm{e}^{-\boldsymbol{\rho}}  \partial_n v_a \partial^n \varpi^b \big) \cdot \partial^i \big( \mathrm{curl} \Omega_i -2\mathrm{e}^{-\boldsymbol{\rho}} \partial_a \boldsymbol{\rho}  \partial_i \varpi^a \big) dx
 \\
 =& \int_{\mathbb{R}^3}  \big( 2 \mathrm{e}^{-\boldsymbol{\rho}}  \partial_n v_a \partial^n \varpi^b \big) \cdot \partial^i \big( -2\mathrm{e}^{-\boldsymbol{\rho}} \partial_a \boldsymbol{\rho}  \partial_i \varpi^a \big) dx,
\end{split}
\end{equation}
where we use the fact $ \partial^i  \mathrm{curl} \Omega_i=0$. By Plancherel formula and H\"older's inequality, we can prove
\begin{equation}\label{K1e}
\begin{split}
  |K_1| = &\left| \int_{\mathbb{R}^3}  \Lambda^{\frac{1}{2}}\big( 2 \mathrm{e}^{-\boldsymbol{\rho}}  \partial_n v_a \partial^n \varpi^b \big) \cdot \Lambda^{-\frac{1}{2}}\partial^i \big( -2\mathrm{e}^{-\boldsymbol{\rho}} \partial_a \boldsymbol{\rho}  \partial_i \varpi^a \big) dx \right|
  \\
  \leq &C\|\partial v \partial \varpi\|_{H^{\frac{1}{2}}_x}+ C\|\partial \boldsymbol{\rho} \partial \varpi\|_{H^{\frac{1}{2}}_x}
  \\
 \leq &C( \|\partial v \|_{H^1_x} + \|\partial \boldsymbol{\rho} \|_{H^1_x})\|\partial \varpi\|_{H^1_x}
\\
 \leq &C( \| v \|^2_{H^2_x} + \|\boldsymbol{\rho} \|^2_{H^2_x})+ \frac{1}{100}\|\varpi\|^2_{H^2_x}.
\end{split}
\end{equation}
On the other hand, using \eqref{rF}, we can get
\begin{equation}\label{w4}
\begin{split}
  |K_2|  \lesssim & \| \partial v\|_{L^\infty_x} \| \partial^2 \varpi\|_{L^2} \|\mathrm{curl} \Omega -2\mathrm{e}^{-\boldsymbol{\rho}} \partial_a \boldsymbol{\rho}  \partial \varpi^a \|_{L^2_x}
  \\
   & +(\| \partial v\|_{L^\infty_x}\| \partial \boldsymbol{\rho}\|_{L^\infty_x} \| \partial \varpi\|_{L^2_x}+ \| \partial \varpi\|^2_{L^2_x})\|\mathrm{curl} \Omega -2\mathrm{e}^{-\boldsymbol{\rho}} \partial_a \boldsymbol{\rho}  \partial \varpi^a \|_{L^2_x}
  \\
  & + \| \partial v\|_{L^\infty_x}\| \partial \boldsymbol{\rho}\|_{L^\infty_x} \| \partial^2 v\|_{L^2}\|\mathrm{curl} \Omega -2\mathrm{e}^{-\boldsymbol{\rho}} \partial_a \boldsymbol{\rho}  \partial \varpi^a \|_{L^2_x}.
\end{split}
\end{equation}
Combining \eqref{WOE} to \eqref{w4}, we therefore have
\begin{equation*}
\begin{split}
  \frac{d}{dt} \|\mathrm{curl} \Omega -2\mathrm{e}^{-\boldsymbol{\rho}} \partial_a \boldsymbol{\rho}  \partial \varpi^a\|^2_{L^2} \leq& C( \| v \|^2_{H^2_x} + \|\boldsymbol{\rho} \|^2_{H^2_x})+ \frac{1}{100}\|\varpi\|^2_{H^2_x}
  \\
  & +C\| \partial v\|_{L^\infty_x} \| \partial^2 \varpi\|_{L^2} \|\mathrm{curl} \Omega -2\mathrm{e}^{-\boldsymbol{\rho}} \partial_a \boldsymbol{\rho}  \partial \varpi^a \|_{L^2_x}
  \\
   & +C(\| \partial v\|_{L^\infty_x}\| \partial \boldsymbol{\rho}\|_{L^\infty_x} \| \partial \varpi\|_{L^2_x}+ \| \partial \varpi\|^2_{L^2_x})\|\mathrm{curl} \Omega -2\mathrm{e}^{-\boldsymbol{\rho}} \partial_a \boldsymbol{\rho}  \partial \varpi^a \|_{L^2_x}
  \\
  & + C\| \partial v\|_{L^\infty_x}\| \partial \boldsymbol{\rho}\|_{L^\infty_x} \| \partial^2 v\|_{L^2}\|\mathrm{curl} \Omega -2\mathrm{e}^{-\boldsymbol{\rho}} \partial_a \boldsymbol{\rho}  \partial \varpi^a \|_{L^2_x}.
\end{split}
\end{equation*}
Integrating it from $[0,t]$, we can derive that
\begin{equation}\label{w6}
\begin{split}
  &\|\mathrm{curl} \Omega -2\mathrm{e}^{-\boldsymbol{\rho}} \partial_a \boldsymbol{\rho}  \partial \varpi^a\|^2_{L^2}(t)-\|\mathrm{curl} \Omega -2\mathrm{e}^{-\boldsymbol{\rho}} \partial_a \boldsymbol{\rho}  \partial \varpi^a\|^2_{L^2}(0)
  \\
  \leq& C\int^t_0( \| v \|^2_{H^2_x} + \|\boldsymbol{\rho} \|^2_{H^2_x}+ \frac{1}{100}\|\varpi\|^2_{H^2_x}) d\tau
  \\
  & +C \int^t_0 \| \partial v\|_{L^\infty_x} \| \partial^2 \varpi\|_{L^2} \|\mathrm{curl} \Omega -2\mathrm{e}^{-\boldsymbol{\rho}} \partial_a \boldsymbol{\rho}  \partial \varpi^a \|_{L^2_x} d\tau
  \\
   & +C \int^t_0 (\| \partial v\|_{L^\infty_x}\| \partial \boldsymbol{\rho}\|_{L^\infty_x} \| \partial \varpi\|_{L^2_x}+ \| \partial \varpi\|^2_{L^2_x})\|\mathrm{curl} \Omega -2\mathrm{e}^{-\boldsymbol{\rho}} \partial_a \boldsymbol{\rho}  \partial \varpi^a \|_{L^2_x} d\tau
  \\
  & + C \int^t_0 \| \partial v\|_{L^\infty_x}\| \partial \boldsymbol{\rho}\|_{L^\infty_x} \| \partial^2 v\|_{L^2}\|\mathrm{curl} \Omega -2\mathrm{e}^{-\boldsymbol{\rho}} \partial_a \boldsymbol{\rho}  \partial \varpi^a \|_{L^2_x} d\tau.
\end{split}
\end{equation}
Consider
\begin{equation}\label{w5}
  \|\mathrm{curl} \Omega\|_{L^2}-C\| \partial \boldsymbol{\rho}  \partial \varpi\|_{L^2} \leq \|\mathrm{curl} \Omega -2\mathrm{e}^{-\boldsymbol{\rho}} \partial_a \boldsymbol{\rho}  \partial \varpi^a \|_{L^2_x} \leq \|\mathrm{curl} \Omega\|_{L^2}+C\| \partial \boldsymbol{\rho}  \partial \varpi\|_{L^2}.
\end{equation}
Inserting \eqref{yxe} into \eqref{w5}, we can have
\begin{equation}\label{w7}
  \|\mathrm{curl} \Omega -2\mathrm{e}^{-\boldsymbol{\rho}} \partial_a \boldsymbol{\rho}  \partial \varpi^a \|_{L^2_x} \geq \|\mathrm{curl} \Omega\|_{L^2}-(C\|\varpi \|_{L_x^{2}} \|\partial \boldsymbol{\rho}\|^4_{H^1_x}+ \frac{1}{100} \|\varpi \|_{{H}_x^{2}}),
\end{equation}
and
\begin{equation}\label{w8}
  \|\mathrm{curl} \Omega -2\mathrm{e}^{-\boldsymbol{\rho}} \partial_a \boldsymbol{\rho}  \partial \varpi^a \|_{L^2_x}
  \leq  \|\mathrm{curl} \Omega\|_{L^2}+(C\|\varpi \|_{L_x^{2}} \|\partial \boldsymbol{\rho}\|^4_{H^1_x}+ \frac{1}{100} \|\varpi \|_{{H}_x^{2}}).
\end{equation}
Adding \eqref{w6} and \eqref{w0}, and combining with \eqref{cw1}, \eqref{cw2}, \eqref{w7}, and \eqref{w8}, we can get
\begin{equation}\label{w2ee}
  \|\varpi\|_{H^2} \lesssim \|(v_0, \boldsymbol{\rho}_0, \varpi_0)\|_{H^2} \exp(5\int^t_0 \| \partial v, \partial \boldsymbol{\rho} \|_{L^\infty_x}d\tau).
\end{equation}
\textbf{Step 2: higher-order energy estimate of $\mathrm{div} \varpi$}. Note \eqref{W01}. Using product estimates in Lemma \ref{ps}, we can derive that
\begin{equation}\label{DW}
\begin{split}
  \| \mathrm{div} \varpi \|_{{H}^{s_0-1}} &= \|  \varpi \cdot \partial \boldsymbol{\rho}  \|_{{H}^{s_0-1}} \lesssim \|\varpi\|_{H^2}\|\partial \boldsymbol{\rho}\|_{{H}^{s_0-1}}.
\end{split}
\end{equation}
By interpolation formula, we can show that
\begin{equation}\label{DW1}
  \|  \varpi \|_{H_x^2} \lesssim \|  \varpi \|^{\frac{s_0-2}{s_0}}_{L_x^2}\| \varpi \|^{\frac{2}{s_0}}_{H^{s_0}_x}.
\end{equation}
Substituting \eqref{DW1} to \eqref{DW}, we then obtain
\begin{equation*}
\begin{split}
  &\| \mathrm{div} \varpi \|_{{H}^{s_0-1}}
  \\
   \lesssim \ & \| \varpi \|^{\frac{s_0-2}{s_0}}_{L_x^2}\| \varpi \|^{\frac{2}{s_0}}_{H^{s_0}_x}\|\partial \boldsymbol{\rho}\|_{{H}^{s_0}}
  \\
   \lesssim \ & \big( \| \varpi_0 \|_{L^2}\exp(\textstyle{\int^t_0}\|dv \|_{L^\infty_x}d\tau) \big)^{\frac{s_0-2}{s_0}}_{L_x^2}\| \varpi \|^{\frac{2}{s_0}}_{H^{s_0}_x}\big( \| \boldsymbol{\rho}_0 \|_{H^{s_0}}\exp(\int^t_0\|dv, d\boldsymbol{\rho} \|_{L^\infty_x}d\tau) \big).
\end{split}
\end{equation*}
For $\frac{s_0-2}{s_0} \in (0,1)$, we then get
\begin{equation*}
\begin{split}
  \| \mathrm{div} \varpi \|_{{H}^{s_0-1}}
   \lesssim  \| \varpi_0 \|^{\frac{s_0-2}{s_0}}_{L^2}\| \boldsymbol{\rho}_0 \|_{H^{s_0}} \| \varpi \|^{\frac{2}{s_0}}_{H^{s_0}_x}\exp(2{\int^t_0}\|dv, d\boldsymbol{\rho} \|_{L^\infty_x}d\tau).
\end{split}
\end{equation*}
By Young's inequality, we can derive that
\begin{equation}\label{DWE}
\begin{split}
  \| \mathrm{div} W \|^2_{{H}^{s_0-1}}
   \lesssim  (\| \varpi_0 \|^{2}_{L^2}+\| \boldsymbol{\rho}_0 \|^2_{H^{s_0}})\exp(4{\int^t_0}\|dv, d\boldsymbol{\rho} \|_{L^\infty_x}d\tau)+ \frac{1}{16}\| \varpi \|^{2}_{H^{s_0}_x}.
\end{split}
\end{equation}

\textbf{Step 3: higher-order energy estimate of $\mathrm{curl} \varpi$}. To get a estimate of $\mathrm{curl} \varpi$, we need to see the equation \eqref{W1} and \eqref{W2}. Note
\begin{equation*}
 \|\mathrm{curl} \varpi\|_{L_x^2} = \|e^{\boldsymbol{\rho}} \Omega \|_{L_x^2}  \lesssim  \|\Omega \|_{L_x^2},
\end{equation*}
From \eqref{W1}, we known
\begin{equation*}
  \frac{d}{dt} \|\Omega\|^2_{L_x^2} \lesssim \| \partial v\|_{L_x^\infty}\|\Omega\|_{L_x^2} \| \partial \varpi\|_{L_x^2}+ \| \partial v\|_{L_x^\infty} \|\Omega\|^2_{L_x^2}.
\end{equation*}
Thus,
\begin{equation}\label{263}
\begin{split}
  \|\Omega\|_{L_x^2} &\lesssim \|\Omega_0\|_{L_x^2}+\| \partial v\|_{L^1_tL_x^\infty} \| \partial W\|_{L^\infty_tL_x^2}.
\end{split}
\end{equation}
For $\Omega_0=\mathrm{e}^{-\boldsymbol{\rho}_0}\mathrm{curl}\varpi_0$, we then have
\begin{equation*}
  \|\Omega_0\|_{L^2}
   \lesssim \|\mathrm{curl} \varpi_0\|_{L_x^2} \lesssim \|\varpi_0\|_{H_x^1}.
\end{equation*}
Substituting the above inequality to \eqref{263}, we obtain
\begin{equation}\label{264}
\begin{split}
  \|\mathrm{curl} \varpi\|_{L_x^2} \lesssim \|\Omega\|_{L_x^2}  & \lesssim \|\varpi_0\|_{H_x^1}+\| \partial v\|_{L^1_tL_x^\infty} \| \varpi\|_{L^\infty_tH_x^1}.
\end{split}
\end{equation}
By elliptic estimates, we first note
\begin{equation}\label{cer}
\begin{split}
\|\mathrm{curl} \varpi\|_{\dot{H}_x^{s_0-1}}
= & \|\mathrm{curl} \mathrm{curl} \varpi\|_{\dot{H}_x^{s_0-2}}= \|\mathrm{curl} (\mathrm{e}^{\boldsymbol{\rho}}\Omega )\|_{\dot{H}_x^{s_0-2}}
\\
 \leq & \| \partial \boldsymbol{\rho} \cdot \Omega\|_{\dot{H}_x^{s_0-2}} + \|\mathrm{e}^{\boldsymbol{\rho}} \mathrm{curl}\Omega\|_{\dot{H}_x^{s_0-2}}
\\
 \leq & \| \partial \boldsymbol{\rho} \cdot \Omega\|_{\dot{H}_x^{s_0-2}} + \|\mathrm{curl}\Omega\|_{\dot{H}_x^{s_0-2}}.
 \end{split}
\end{equation}
Similarly, by using $\Omega=\mathrm{e}^{-\boldsymbol{\rho}}\mathrm{curl}\varpi$, we then get
\begin{equation}\label{ceh}
\begin{split}
\|\mathrm{curl}\Omega\|_{\dot{H}_x^{s_0-2}}& = \|\mathrm{curl} (\mathrm{e}^{-\boldsymbol{\rho}}\mathrm{curl}\varpi)\|_{\dot{H}_x^{s_0-2}}
\\
& \lesssim \| \partial \boldsymbol{\rho} \cdot \mathrm{curl}\varpi \|_{\dot{H}_x^{s_0-2}} + \|\mathrm{curl}\varpi\|_{\dot{H}_x^{s_0-1}}
\\
& \lesssim \| \partial \boldsymbol{\rho} \cdot \Omega \|_{\dot{H}_x^{s_0-2}} + \|\mathrm{curl}\varpi\|_{\dot{H}_x^{s_0-1}}.
 \end{split}
\end{equation}
Seeing from \eqref{cer} and \eqref{ceh}, we known the difference between $\|\mathrm{curl}\Omega\|_{\dot{H}_x^{s_0-2}}$ and $\|\mathrm{curl}\varpi\|_{\dot{H}_x^{s_0-1}}$ is only a lower order term $\| \partial \boldsymbol{\rho} \cdot \Omega \|_{\dot{H}_x^{s_0-2}}$. Let us see what's the bound of this lower order term. On one hand, we use Lemma \ref{ps}(taking $s_1 =\frac{1}{2}+s_0-2, s_2=1$) to obtain
\begin{equation}\label{266}
\begin{split}
  \| \partial \boldsymbol{\rho} \cdot \Omega\|_{\dot{H}_x^{s_0-2}}  &\lesssim \| \partial \boldsymbol{\rho} \|_{{H}_x^{\frac{1}{2}+s_0-2}} \|\Omega\|_{{H}_x^{1}}
  \\
  & \lesssim \| \partial \boldsymbol{\rho} \|_{{H}_x^{s-1}}\|\Omega\|_{{H}_x^{1}}
  \\
  & \lesssim  (1+\| \boldsymbol{\rho} \|^2_{{H}_x^{s}})\|\mathrm{curl} \varpi\|_{{H}_x^{1}}
  \\
  & \lesssim  (1+\| \boldsymbol{\rho} \|^2_{{H}_x^{s}})\|\varpi\|^{\frac{s_0-2}{s_0}}_{L^2_x} \| \varpi\|^{\frac{2}{s_0}}_{{H}_x^{s_0}}.
\end{split}
\end{equation}
where we use the fact
\begin{equation*}
  \begin{split}
  \|\Omega\|_{{H}_x^{1}} & \lesssim \| \mathrm{e}^{-\boldsymbol{\rho}}\mathrm{curl} \varpi \|_{{H}_x^{1}}
   \lesssim (1+\|\boldsymbol{\rho} \|_{{H}_x^{s}})\|\mathrm{curl} \varpi\|_{{H}_x^{1}}.
  \end{split}
\end{equation*}
By using Young's inequality, we can update \eqref{266} by
\begin{equation}\label{26e}
\begin{split}
  \| \partial \boldsymbol{\rho} \cdot \Omega\|_{\dot{H}_x^{s_0-2}}
  & \lesssim  (1+\| \boldsymbol{\rho} \|^2_{{H}_x^{s}})\|\varpi\|_{L^2_x} +\frac{1}{16}\| \varpi\|_{{H}_x^{s_0}}.
\end{split}
\end{equation}
We are now in a position to consider $\|\mathrm{curl}\Omega\|_{\dot{H}_x^{s_0-2}}$. Operating $\Lambda_x^{s_0-2}$ on \eqref{W2} give rise to
\begin{equation*}
\begin{split}
&\mathbf{T}  \left( \Lambda_x^{s_0-2} \left( \text{curl}\Omega^i-2\mathrm{e}^{-\boldsymbol{\rho}}  \partial_a \boldsymbol{\rho}  \partial^i \varpi^a\right) \right)
\\
=&-[\Lambda_x^{s_0-2}, \mathbf{T}]( \mathrm{curl} \Omega^i -2 \mathrm{e}^{-\boldsymbol{\rho}} \partial_a \boldsymbol{\rho}  \partial^i \varpi^a)
\\
&+ \partial^i \big(\Lambda_x^{s_0-2}  \big( 2 \mathrm{e}^{-\boldsymbol{\rho}}  \partial_n v^a \partial^n \varpi_a \big) \big) + \Lambda_x^{s_0-2} R^i_1
\\
&+ \Lambda_x^{s_0-2} R^i_2+ \Lambda_x^{s_0-2} R^i_3+\Lambda_x^{s_0-2} R^i_4+\Lambda_x^{s_0-2} R^i_5+\Lambda_x^{s_0-2} R^i_6.
\end{split}
\end{equation*}
Multiplying $\Lambda_x^{s_0-2} \left( \text{curl}\Omega_i-2\mathrm{e}^{-\boldsymbol{\rho}}  \partial_a \boldsymbol{\rho}  \partial_i \varpi^a\right)$ on the above equality and integrating it on $\mathbb{R}^3$, we have
\begin{equation}\label{W21}
\begin{split}
& \frac{d}{dt}\left( \| \text{curl}\Omega-2  \mathrm{e}^{-\boldsymbol{\rho}}  \partial_a \boldsymbol{\rho}  \partial \varpi^a  \|^2_{\dot{H}^{s_0-2}}  \right)= \sum^{8}_{l=0}I_l,
\end{split}
\end{equation}
where
\begin{equation*}
\begin{split}
 & I_0= { \int_{\mathbb{R}^3}}  \partial_k v^k | \Lambda_x^{s_0-2}\left(\text{curl}\Omega-2  \mathrm{e}^{-\boldsymbol{\rho}}  \partial_a \boldsymbol{\rho}  \partial \varpi^a \right)|^2  dx,
 \\
 & I_1  = \displaystyle{ \int_{\mathbb{R}^3}} \Lambda_x^{s_0-2} R^i_1 \cdot \Lambda_x^{s_0-2} \left( \text{curl}\Omega_i-2\mathrm{e}^{-\boldsymbol{\rho}}  \partial_a \boldsymbol{\rho}  \partial_i \varpi^a\right) dx,
\\
& I_2 =\displaystyle{ \int_{\mathbb{R}^3}} \Lambda_x^{s_0-2} R^i_2 \cdot \Lambda_x^{s_0-2} \left( \text{curl}\Omega_i-2\mathrm{e}^{-\boldsymbol{\rho}}  \partial_a \boldsymbol{\rho}  \partial_i \varpi^a\right)dx,
\\
 &I_3 = \displaystyle{ \int_{\mathbb{R}^3}} \Lambda_x^{s_0-2} R^i_3 \cdot \Lambda_x^{s_0-2} \left( \text{curl}\Omega_i-2\mathrm{e}^{-\boldsymbol{\rho}}  \partial_a \boldsymbol{\rho}  \partial_i \varpi^a\right)dx,
 \\
 & I_4  = \displaystyle{ \int_{\mathbb{R}^3}} \Lambda_x^{s_0-2} R^i_4 \cdot \Lambda_x^{s_0-2} \left( \text{curl}\Omega_i-2\mathrm{e}^{-\boldsymbol{\rho}}  \partial_a \boldsymbol{\rho}  \partial_i \varpi^a\right)dx,
 \\
 &I_5 = \textstyle{\sum_{i=1}^3}\displaystyle{ \int_{\mathbb{R}^3}} \Lambda_x^{s_0-2} R_5 \cdot \Lambda_x^{s_0-2} \left( \text{curl}\Omega_i-2\mathrm{e}^{-\boldsymbol{\rho}}  \partial_a \boldsymbol{\rho}  \partial_i \varpi^a\right)dx,
\\
 &I_6 = \displaystyle{ \int_{\mathbb{R}^3}} \Lambda_x^{s_0-2} R_6 \cdot \Lambda_x^{s_0-2} \left( \text{curl}\Omega_i-2\mathrm{e}^{-\boldsymbol{\rho}}  \partial_a \boldsymbol{\rho}  \partial_i \varpi^a\right)dx,
 \\ &I_7=\displaystyle{ \int_{\mathbb{R}^3}}  \partial^i \big( \Lambda_x^{s_0-2} \big( 2 \mathrm{e}^{-\boldsymbol{\rho}}  \partial_n v^a \partial^n \varpi_a \big) \big)\cdot\Lambda_x^{s_0-2} \left( \text{curl}\Omega_i-2\mathrm{e}^{-\boldsymbol{\rho}}  \partial_a \boldsymbol{\rho}  \partial_i \varpi^a\right)dx,
\\
&I_8  =\displaystyle{ \int_{\mathbb{R}^3}}[\Lambda_x^{s_0-2}, v \cdot \nabla]  \big( \text{curl}\Omega^i-2\mathrm{e}^{-\boldsymbol{\rho}} \partial_a \boldsymbol{\rho}  \partial^i \varpi^a \big)\cdot \Lambda_x^{s_0-2} \left( \text{curl}\Omega_i-2\mathrm{e}^{-\boldsymbol{\rho}}  \partial_a \boldsymbol{\rho}  \partial_i \varpi^a\right)dx.
\end{split}
\end{equation*}
We will estimate these terms one by one. For $I_0$, by H\"older's inequality, we can derive
\begin{equation}\label{I0}
\begin{split}
  |I_0| \lesssim  & \| \partial v\|_{L^\infty_x} \| \text{curl}\Omega-2  \mathrm{e}^{-\boldsymbol{\rho}}  \partial_a \boldsymbol{\rho}  \partial \varpi^a \|^2_{\dot{H}^{s_0-2}}
  \\
  \lesssim & \| \partial v\|_{\dot{B}^{s_0-2}_{\infty,2}} \| \text{curl}\Omega-2  \mathrm{e}^{-\boldsymbol{\rho}}  \partial_a \boldsymbol{\rho}  \partial \varpi^a \|^2_{\dot{H}^{s_0-2}}.
  \end{split}
\end{equation}
By using $R_1$ in \eqref{rF}, H\"older's inequality and Lemma \ref{lpe}, we have
\begin{equation}\label{I1}
\begin{split}
  |I_1| \lesssim & \| \partial v \cdot \partial^2\varpi\|_{\dot{H}^{s_0-2}} \| \text{curl}\Omega-2  \mathrm{e}^{-\boldsymbol{\rho}}  \partial_a \boldsymbol{\rho}  \partial \varpi^a \|_{\dot{H}^{s_0-2}}
  \\
   \lesssim &\big( \| \partial v\|_{\dot{B}^{s_0-2}_{\infty,2}}\|\partial^2\varpi\|_{L^2}+\| \partial v\|_{L^\infty_x}\|\partial^2\varpi\|_{\dot{H}^{s_0-2}} \big)\| \text{curl}\Omega-2  \mathrm{e}^{-\boldsymbol{\rho}}  \partial_a \boldsymbol{\rho}  \partial \varpi^a \|_{\dot{H}^{s_0-2}}
  \\
  \lesssim &  \| \partial v\|_{\dot{B}^{s_0-2}_{\infty,2}}  \|\partial^2 \varpi\|_{H^{s_0-2}} \| \text{curl}\Omega-2  \mathrm{e}^{-\boldsymbol{\rho}}  \partial_a \boldsymbol{\rho}  \partial \varpi^a \|_{\dot{H}^{s_0-2}}.
  \end{split}
\end{equation}
Using $R_6$ in \eqref{rF}, H\"older's inequality, Lemma \ref{lpe}, and the product estimate in Besov spaces, we then get
\begin{equation}\label{I6}
\begin{split}
  |I_6| \lesssim \ & \| \partial v \cdot \partial \boldsymbol{\rho} \cdot \partial^2 v\|_{\dot{H}^{s_0-2}} \| \text{curl}\Omega-2  \mathrm{e}^{-\boldsymbol{\rho}}  \partial_a \boldsymbol{\rho}  \partial \varpi^a \|_{\dot{H}^{s_0-2}}
  \\
   \lesssim \ & \big( \|\partial v \cdot \partial \boldsymbol{\rho} \|_{\dot{B}^{s_0-2}_{\infty,2}} \|\partial^2 v\|_{L^2}+\| \partial v \cdot \partial \boldsymbol{\rho}\|_{L^\infty_x}\|\partial^2 v\|_{\dot{H}^{s_0-2}} \big)\| \text{curl}\Omega-2  \mathrm{e}^{-\boldsymbol{\rho}}  \partial_a \boldsymbol{\rho}  \partial \varpi^a \|_{\dot{H}^{s_0-2}}
  \\
   \lesssim \ & \big( \|\partial v \|_{L^\infty_x} \| \partial \boldsymbol{\rho} \|_{\dot{B}^{s_0-2}_{\infty,2}}+ \|\partial \boldsymbol{\rho} \|_{L^\infty_x} \| \partial v \|_{\dot{B}^{s_0-2}_{\infty,2}}\big)  \|\partial^2 v\|_{L^2}\| \text{curl}\Omega-2  \mathrm{e}^{-\boldsymbol{\rho}}  \partial_a \boldsymbol{\rho}  \partial \varpi^a \|_{\dot{H}^{s_0-2}}
  \\
   & \ +\| \partial v\|_{L^\infty_x} \|\partial \boldsymbol{\rho}\|_{L^\infty_x}\|\partial^2 v\|_{\dot{H}^{s_0-2}}\| \text{curl}\Omega-2  \mathrm{e}^{-\boldsymbol{\rho}}  \partial_a \boldsymbol{\rho}  \partial \varpi^a \|_{\dot{H}^{s_0-2}}
   \\
   \lesssim & (\| \partial \boldsymbol{\rho} \|^2_{\dot{B}^{s_0-2}_{\infty,2}}+\| \partial v \|^2_{\dot{B}^{s_0-2}_{\infty,2}})\|\partial^2 v\|_{{H}^{s_0-2}}\| \text{curl}\Omega-2  \mathrm{e}^{-\boldsymbol{\rho}}  \partial_a \boldsymbol{\rho}  \partial \varpi^a \|_{\dot{H}^{s_0-2}}.
  \end{split}
\end{equation}
{By using $R_2$ in \eqref{rF}, H\"older's inequality and Lemma \ref{ps}, we get}
\begin{equation}\label{I2}
\begin{split}
  |I_2| \lesssim &\| (\partial v, \partial \boldsymbol{\rho}) \cdot \partial \boldsymbol{\rho} \cdot \partial \Omega\|_{\dot{H}^{s_0-2}} \| \text{curl}\Omega-2  \mathrm{e}^{-\boldsymbol{\rho}}  \partial_a \boldsymbol{\rho}  \partial \varpi^a \|_{\dot{H}^{s_0-2}}
  \\
  \lesssim &\|(\partial v, \partial \boldsymbol{\rho}) \cdot \partial \boldsymbol{\rho} \|_{\dot{B}^{s_0-2}_{\infty,2}} \|\partial \Omega\|_{L^2}\| \text{curl}\Omega-2  \mathrm{e}^{-\boldsymbol{\rho}}  \partial_a \boldsymbol{\rho}  \partial \varpi^a \|_{\dot{H}^{s_0-2}}
  \\
  & + \|(\partial v, \partial \boldsymbol{\rho}) \cdot \partial \boldsymbol{\rho}\|_{L^\infty_x}\|\partial \Omega\|_{\dot{H}^{s_0-2}} \| \text{curl}\Omega-2  \mathrm{e}^{-\boldsymbol{\rho}}  \partial_a \boldsymbol{\rho}  \partial \varpi^a \|_{\dot{H}^{s_0-2}}
  \\
   \lesssim & (\| \partial v\|_{L^\infty_x}\| \partial \boldsymbol{\rho}\|_{\dot{B}^{s_0-2}_{\infty,2}}+ \| \partial  \boldsymbol{\rho}\|_{L^\infty_x}\| \partial v\|_{\dot{B}^{s_0-2}_{\infty,2}})  \|\partial \Omega\|_{L^2} \| \text{curl}\Omega-2  \mathrm{e}^{-\boldsymbol{\rho}}  \partial_a \boldsymbol{\rho}  \partial \varpi^a \|_{\dot{H}^{s_0-2}}
  \\
  & + (\| \partial \boldsymbol{\rho}\|_{L^\infty_x}\| \partial \boldsymbol{\rho}\|_{\dot{B}^{s_0-2}_{\infty,2}}+ \| \partial  \boldsymbol{\rho}\|_{L^\infty_x}\| \partial \boldsymbol{\rho}\|_{\dot{B}^{s_0-2}_{\infty,2}})  \|\partial \Omega\|_{L^2}\| \text{curl}\Omega-2  \mathrm{e}^{-\boldsymbol{\rho}}  \partial_a \boldsymbol{\rho}  \partial \varpi^a \|_{\dot{H}^{s_0-2}}
  \\
  &  + (\| \partial v\|_{L^\infty_x}+ \| \partial \boldsymbol{\rho}\|_{L^\infty_x})\| \partial \boldsymbol{\rho}\|_{L^\infty_x} \|\partial \Omega\|_{\dot{H}^{s_0-2}} \| \text{curl}\Omega-2  \mathrm{e}^{-\boldsymbol{\rho}}  \partial_a \boldsymbol{\rho}  \partial \varpi^a \|_{\dot{H}^{s_0-2}}
  \\
  \lesssim &  (\| \partial \boldsymbol{\rho}\|^2_{\dot{B}^{s_0-2}_{\infty,2}}+ \| \partial v\|^2_{\dot{B}^{s_0-2}_{\infty,2}})  \|\partial \Omega\|_{H^{s_0-2}} \| \text{curl}\Omega-2  \mathrm{e}^{-\boldsymbol{\rho}}  \partial_a \boldsymbol{\rho}  \partial \varpi^a \|_{\dot{H}^{s_0-2}}.
  \end{split}
\end{equation}
Noting $R_3$ in \eqref{rF} and using Lemma \ref{wql}, we also have
\begin{equation}\label{I3}
\begin{split}
  |I_3| & \lesssim \| \partial v \partial \boldsymbol{\rho} \partial \varpi \|_{\dot{H}_x^{s_0-2}} \| \text{curl}\Omega-2  \mathrm{e}^{-\boldsymbol{\rho}}  \partial_a \boldsymbol{\rho}  \partial \varpi^a \|_{\dot{H}^{s_0-2}}
  \\
  & \lesssim \| \partial v\|_{L^\infty_x}\| \partial \boldsymbol{\rho}\|_{L^\infty_x} \| \partial \varpi\|_{H^{s_0-2}_x} \| \text{curl}\Omega-2  \mathrm{e}^{-\boldsymbol{\rho}}  \partial_a \boldsymbol{\rho}  \partial \varpi^a \|_{\dot{H}^{s_0-2}}
  \\
  & \quad + \| \partial \boldsymbol{\rho}\|_{L^\infty_x}\| \partial v\|_{H^{{s_0-1}}_x} \| \partial \varpi\|_{H^1_x} \| \text{curl}\Omega-2  \mathrm{e}^{-\boldsymbol{\rho}}  \partial_a \boldsymbol{\rho}  \partial \varpi^a \|_{\dot{H}^{s_0-2}}
  \\
  & \quad + \| \partial v||_{L^\infty_x}\| \partial \boldsymbol{\rho}\|_{H^{{s_0-1}}_x} \| \partial \varpi\|_{H^1_x} \| \text{curl}\Omega-2  \mathrm{e}^{-\boldsymbol{\rho}}  \partial_a \boldsymbol{\rho}  \partial \varpi^a \|_{\dot{H}^{s_0-2}}
  \\
  & \lesssim (\| \partial \boldsymbol{\rho}\|^2_{\dot{B}^{s_0-2}_{\infty,2}}+ \| \partial v\|^2_{\dot{B}^{s_0-2}_{\infty,2}})  \|\partial \Omega\|_{H^{s_0-2}} \| \text{curl}\Omega-2  \mathrm{e}^{-\boldsymbol{\rho}}  \partial_a \boldsymbol{\rho}  \partial \varpi^a \|_{\dot{H}^{s_0-2}}..
  \end{split}
\end{equation}
By using $R_4$ in \eqref{rF}, H\"older's inequality and Lemma \ref{ps}, we can show
\begin{equation}\label{I4}
\begin{split}
  |I_4|
   \lesssim & \| \partial \varpi \cdot \partial \varpi\|_{\dot{H}_x^{s_0-2}} \| \text{curl}\Omega-2  \mathrm{e}^{-\boldsymbol{\rho}}  \partial_a \boldsymbol{\rho}  \partial \varpi^a \|_{\dot{H}^{s_0-2}}
  \\
   \lesssim & \| \partial \varpi\|^2_{H^{\frac{2s_0-1}{4}}_x} \| \text{curl}\Omega-2  \mathrm{e}^{-\boldsymbol{\rho}}  \partial_a \boldsymbol{\rho}  \partial \varpi^a \|_{\dot{H}^{s_0-2}}.
  \end{split}
\end{equation}
Recalling $R_5$ in \eqref{rF} and using H\"older's inequality and Lemma \ref{ps}(taking $s_1=1, s_2=1$), we have
\begin{equation}\label{I5}
\begin{split}
  |I_5|
  & \lesssim \| \partial v \cdot \partial v \cdot \partial \boldsymbol{\rho}\|_{\dot{H}_x^{s_0-2}} \| \text{curl}\Omega-2  \mathrm{e}^{-\boldsymbol{\rho}}  \partial_a \boldsymbol{\rho}  \partial \varpi^a \|_{\dot{H}^{s_0-2}}
  \\
  & \lesssim \big( \| \partial v\|_{L^\infty_x}+ \| \partial \boldsymbol{\rho}\|_{L^\infty_x} \big) \big(\| \partial v\|_{H^{{s_0-1}}_x} +\| \partial \boldsymbol{\rho}\|_{H^{{s_0-1}}_x} \big)^2\| \text{curl}\Omega-2  \mathrm{e}^{-\boldsymbol{\rho}}  \partial_a \boldsymbol{\rho}  \partial \varpi^a \|_{\dot{H}^{s_0-2}}.
  \end{split}
\end{equation}
It remains for us to handle the most difficult term $I_7$. Integrating $I_7$ by parts, we get
\begin{equation}\label{I7B}
  \begin{split}
  I_7=& \int_{\mathbb{R}^3} \partial^i \big( \Lambda_x^{s_0-2} \big( 2 \mathrm{e}^{-\boldsymbol{\rho}}  \partial_n v^a \partial^n \varpi_a \big) \big)\Lambda_x^{s_0-2}(\text{curl}\Omega^i)dx
  \\
  &+ \int_{\mathbb{R}^3} \partial^i \big( \Lambda_x^{s_0-2} \big( 2 \mathrm{e}^{-\boldsymbol{\rho}}  \partial_n v^a \partial^n \varpi_a \big) \big)\Lambda_x^{s_0-2}(-2  \mathrm{e}^{-\boldsymbol{\rho}}  \partial_a \boldsymbol{\rho}  \partial_i \varpi^a)dx
  \\
   =&  \int_{\mathbb{R}^3}  \partial^i \left\{ \Lambda_x^{s_0-2} \big( 2 \mathrm{e}^{-\boldsymbol{\rho}}  \partial_n v^a \partial^n \varpi_a \big) \Lambda_x^{s_0-2}(\text{curl}\Omega^i) \right\}dx
  \\
   &-   \int_{\mathbb{R}^3}  \big( \Lambda_x^{s_0-2} \big( 2 \mathrm{e}^{-\boldsymbol{\rho}}  \partial_n v^a \partial^n \varpi_a \big)  \partial^i \Lambda_x^{s_0-2}(\text{curl}\Omega^i\big)dx
   \\
   &+ \int_{\mathbb{R}^3} \partial^i \big( \Lambda_x^{s_0-2} \big( 2 \mathrm{e}^{-\boldsymbol{\rho}}  \partial_n v^a \partial^n \varpi_a \big) \big)\Lambda_x^{s_0-2}(-2  \mathrm{e}^{-\boldsymbol{\rho}}  \partial_a \boldsymbol{\rho}  \partial_i \varpi^a)dx
   \\
   =& -  \int_{\mathbb{R}^3}  \big( \Lambda_x^{s_0-2} \big( 2 \mathrm{e}^{-\boldsymbol{\rho}}  \partial_n v^a \partial^n \varpi_a \big) \partial^i \Lambda_x^{s_0-2}(\text{curl}\Omega^i\big)dx
   \\
   &+ \int_{\mathbb{R}^3} \partial^i \big( \Lambda_x^{s_0-2} \big( 2 \mathrm{e}^{-\boldsymbol{\rho}}  \partial_n v^a \partial^n \varpi_a \big) \big)\Lambda_x^{s_0-2}(-2  \mathrm{e}^{-\boldsymbol{\rho}}  \partial_a \boldsymbol{\rho}  \partial_i \varpi^a)dx.
  \end{split}
\end{equation}
We note
\begin{equation}\label{IB0}
 \partial^i \Lambda_x^{s_0-2}(\text{curl}\Omega^i)=\Lambda_x^{s_0-2} \mathrm{div}(\text{curl}\Omega)=0.
\end{equation}
Substituting \eqref{IB0} to \eqref{I7B}, we can update it as
\begin{equation*}
\begin{split}
  I_7=&\int_{\mathbb{R}^3} \partial^i \big( \Lambda_x^{s_0-2} \big( 2 \mathrm{e}^{-\boldsymbol{\rho}}  \partial_n v^a \partial^n \varpi_a \big) \big)\Lambda_x^{s_0-2}(-2  \mathrm{e}^{-\boldsymbol{\rho}}  \partial_a \boldsymbol{\rho}  \partial_i \varpi^a)dx.
  \end{split}
\end{equation*}
By Plancherel formula, we can obtain
\begin{equation*}
\begin{split}
  I_7=&\int_{\mathbb{R}^3} \xi^i |\xi|^{s_0-2} \widehat{\big( 2 \mathrm{e}^{-\boldsymbol{\rho}}  \partial_n v^a \partial^n \varpi_a \big)} \cdot |\xi|^{s_0-2} \widehat{\big(-2  \mathrm{e}^{-\boldsymbol{\rho}}  \partial_a \boldsymbol{\rho}  \partial_i \varpi^a \big)}d\xi
  \\
  =&\int_{\mathbb{R}^3} \xi^i |\xi|^{s_0-\frac{5}{2}} \widehat{\big( 2 \mathrm{e}^{-\boldsymbol{\rho}}  \partial_n v^a \partial^n \varpi_a \big)} \cdot |\xi|^{s_0-\frac{3}{2}} \widehat{\big(-2  \mathrm{e}^{-\boldsymbol{\rho}}  \partial_a \boldsymbol{\rho}  \partial_i \varpi^a \big)}d\xi
  \\
  =&\int_{\mathbb{R}^3}  \Lambda^{s_0-\frac{5}{2}} \partial^i \big( 2 \mathrm{e}^{-\boldsymbol{\rho}}  \partial_n v^a \partial^n \varpi_a \big) \cdot \Lambda^{s_0-\frac{3}{2}} \big(-2  \mathrm{e}^{-\boldsymbol{\rho}}  \partial_a \boldsymbol{\rho}  \partial_i \varpi^a \big)dx.
  \end{split}
\end{equation*}
By H\"older inequality and Lemma \ref{ps}, we can obtain
\begin{equation}\label{I7}
\begin{split}
  | I_{7}| \lesssim & \|  \partial v \partial \varpi \|_{H^{s_0-\frac{3}{2}}}+\|  \partial \boldsymbol{\rho} \partial\varpi \|_{H^{s_0-\frac{3}{2}}}
  \\
  \lesssim & (\|  \partial v\|_{H^{s_0-1}} + \| \partial \boldsymbol{\rho}\|_{H^{s_0-1}}) \|\partial \varpi \|_{H^{1}}.
  \end{split}
\end{equation}
It remains for us to give a bound for $I_8$. By using Lemma \ref{ce}, we can prove
\begin{equation}\label{I8}
\begin{split}
  |I_8|  \lesssim & \|[\Lambda_x^{s_0-2}, v \cdot \nabla]  \big( \text{curl}\Omega-2\mathrm{e}^{-\boldsymbol{\rho}}  \partial_a \boldsymbol{\rho}  \partial \varpi^a \big)\|_{L^2_x}
  \\
  & \ \cdot \| \text{curl}\Omega-2  \mathrm{e}^{-\boldsymbol{\rho}}  \partial_a \boldsymbol{\rho}  \partial \varpi^a \|_{\dot{H}^{s_0-2}}
  \\
   \lesssim & \|\partial v \|_{\dot{B}^{0}_{\infty,2}}  \| \text{curl}\Omega-2  \mathrm{e}^{-\boldsymbol{\rho}}  \partial_a \boldsymbol{\rho}  \partial \varpi^a \|^2_{\dot{H}^{s_0-2}}
   \\
   \lesssim &\|\partial v \|_{\dot{B}^{s_0-2}_{\infty,2}}\| \text{curl}\Omega-2  \mathrm{e}^{-\boldsymbol{\rho}}  \partial_a \boldsymbol{\rho}  \partial \varpi^a \|^2_{\dot{H}^{s_0-2}}.
\end{split}
\end{equation}
Combining \eqref{I0}, \eqref{I1},... to \eqref{I8}, we can derive that
\begin{equation}\label{WF}
\begin{split}
& \frac{d}{dt}\left( \| \text{curl}\Omega-2  \mathrm{e}^{-\boldsymbol{\rho}}  \partial_a \boldsymbol{\rho}  \partial \varpi^a  \|^2_{\dot{H}^{s_0-2}}  \right)
\\
\leq  & C\| \partial v\|_{\dot{B}^{s_0-2}_{\infty,2}}  \|\partial^2 \varpi\|_{H^{s_0-2}} \| \text{curl}\Omega-2  \mathrm{e}^{-\boldsymbol{\rho}}  \partial_a \boldsymbol{\rho}  \partial \varpi^a \|_{\dot{H}^{s_0-2}}
\\
&+C(\| \partial \boldsymbol{\rho}\|^2_{\dot{B}^{s_0-2}_{\infty,2}}+ \| \partial v\|^2_{\dot{B}^{s_0-2}_{\infty,2}})  \|\partial \Omega\|_{H^{s_0-2}} \| \text{curl}\Omega-2  \mathrm{e}^{-\boldsymbol{\rho}}  \partial_a \boldsymbol{\rho}  \partial \varpi^a \|_{\dot{H}^{s_0-2}}
\\
& +C \| \partial v, \partial \boldsymbol{\rho}\|_{L^\infty_x}\| \partial \boldsymbol{\rho}, \partial v\|_{H^{{s_0-1}}}\|\partial \varpi \|_{{H}^{1}} \| \text{curl}\Omega-2  \mathrm{e}^{-\boldsymbol{\rho}}  \partial_a \boldsymbol{\rho}  \partial \varpi^a \|_{\dot{H}^{s_0-2}}
\\
&+C\| \partial \varpi\|^2_{H^{\frac{2s_0-1}{4}}} \| \text{curl}\Omega-2  \mathrm{e}^{-\boldsymbol{\rho}}  \partial_a \boldsymbol{\rho}  \partial \varpi^a \|_{\dot{H}^{s_0-2}}
\\
&  +C \| \partial v, \partial \boldsymbol{\rho}\|_{L^\infty_x} \| \partial v, \partial \boldsymbol{\rho}\|^2_{H^{{s_0-1}}} \| \text{curl}\Omega-2  \mathrm{e}^{-\boldsymbol{\rho}}  \partial_a \boldsymbol{\rho}  \partial \varpi^a \|_{\dot{H}^{s_0-2}}
\\
& +C\big( \|\partial v\|^2_{\dot{B}^{s_0-2}_{\infty,2}} + \|\partial \boldsymbol{\rho} \|^2_{\dot{B}^{s_0-2}_{\infty,2}}\big)  \|\partial^2 v\|_{H^{s_0-2}}\| \text{curl}\Omega-2  \mathrm{e}^{-\boldsymbol{\rho}}  \partial_a \boldsymbol{\rho}  \partial \varpi^a \|_{\dot{H}^{s_0-2}}
\\
&+ C \|  \partial v, \partial \boldsymbol{\rho}\|_{H^{s_0-1}} \|\partial \varpi \|_{H^{1}}+C\|\partial v \|_{\dot{B}^{s_0-2}_{\infty,2}}  \| \text{curl}\Omega-2  \mathrm{e}^{-\boldsymbol{\rho}}  \partial_a \boldsymbol{\rho}  \partial \varpi^a \|^2_{\dot{H}^{s_0-2}}
\end{split}
\end{equation}
Denote
\begin{equation*}
  K(t)=\| \text{curl}\Omega-2  \mathrm{e}^{-\boldsymbol{\rho}}  \partial_a \boldsymbol{\rho}  \partial \varpi^a  \|_{\dot{H}^{s_0-2}}.
\end{equation*}
Integrating \eqref{WF} from $[0,t]$, we have
\begin{equation}\label{WF}
\begin{split}
 K^2(t)-K^2(0)    \leq & C\int^t_0 \| \partial v\|_{ \dot{B}^{s_0-2}_{\infty,2}}  \|\partial^2 \varpi\|_{H^{s_0-2}} K(\tau)d\tau+ C \int^t_0 \|  \partial v, \partial \boldsymbol{\rho}\|_{H^{s_0-1}} \|\partial \varpi \|_{H^{1}} d\tau
\\
& + C\int^t_0 (\| \partial \boldsymbol{\rho}\|^2_{\dot{B}^{s_0-2}_{\infty,2}}+ \| \partial v\|^2_{\dot{B}^{s_0-2}_{\infty,2}})  \|\partial \Omega\|_{H^{s_0-2}} K(\tau)d\tau
\\
& +C  \int^t_0\| \partial v, \partial \boldsymbol{\rho}\|_{L^\infty_x}\| \partial \boldsymbol{\rho}, \partial v\|_{H^{{s_0-1}}}\|\partial \varpi \|_{{H}^{1}} K(\tau)d\tau
\\
&+C\int^t_0 \| \partial \varpi\|^2_{H^{\frac{2s_0-1}{4}}} K(\tau)d\tau+C\int^t_0\|\partial v \|_{\dot{B}^{s_0-2}_{\infty,2}}  K^2(\tau)d\tau
\\
& +C\int^t_0 \big( \|\partial v\|^2_{\dot{B}^{s_0-2}_{\infty,2}} + \|\partial \boldsymbol{\rho} \|^2_{\dot{B}^{s_0-2}_{\infty,2}}\big)  \|\partial^2 v\|_{H^{s_0-2}}K(\tau)d\tau
\\
&  +C \int^t_0 \| \partial v, \partial \boldsymbol{\rho}\|_{L^\infty_x} \| \partial v, \partial \boldsymbol{\rho}\|^2_{H^{{s_0-1}}} K(\tau)d\tau.
\end{split}
\end{equation}
By Young's inequality, we give the upper bound of $K(t)$
\begin{equation}\label{Up}
\begin{split}
  K^2(t)
    \geq \ & \|\text{curl}\Omega\|^2_{\dot{H}^{s_0-2}}-  C\| \mathrm{e}^{-\boldsymbol{\rho}}\partial \boldsymbol{\rho} \cdot \partial \varpi \|^2_{\dot{H}^{s_0-2}}
    \\
    \geq \ & \|\text{curl}\Omega\|^2_{\dot{H}^{s_0-2}}-  C\| \partial \boldsymbol{\rho}\|^2_{H^{s_0-1}_x}\|\partial \varpi\|^2_{{H}^\frac{1}{2}}
    \\
    \geq \ & \|\text{curl}\Omega\|^2_{\dot{H}^{s_0-2}}-  C\| \partial \boldsymbol{\rho}\|^2_{H^{s_0-1}_x}(\| \varpi\|^{\frac{1}{4}}_{L^2} \|\partial \varpi\|^{\frac{3}{4}}_{{H}^1})^2
    \\
    \geq \ & \|\text{curl}\Omega\|^2_{\dot{H}^{s_0-2}}-  C\| \partial \boldsymbol{\rho}\|^8_{H^{s_0-1}_x}\| \varpi\|^{2}_{L^2}-\frac{1}{100}\| \varpi\|^{2}_{H^{2}_x} .
\end{split}
\end{equation}
By H\"older inequality, we can give the supper bound of $K(t)$
\begin{equation}\label{Sp}
\begin{split}
  K^2(t)
    \leq \ & \|\text{curl}\Omega\|^2_{\dot{H}^{s_0-2}}+  C\| \mathrm{e}^{-\boldsymbol{\rho}}\partial \boldsymbol{\rho} \cdot \partial \varpi \|^2_{\dot{H}^{s_0-2}}
    \\
    \leq \ & \|\text{curl}\Omega\|^2_{\dot{H}^{s_0-2}}+  C\| \partial \boldsymbol{\rho}\|^2_{H^{s_0-1}_x}\|\partial \varpi\|^2_{{H}^\frac{1}{2}}.
\end{split}
\end{equation}
Adding \eqref{WOE} and combining \eqref{Sp}, \eqref{Up}, and \eqref{WF}, we therefore get
\begin{equation*}
\begin{split}
  \| \varpi \|^2_{H^{s_0}} \lesssim & \ \| (v_0, \boldsymbol{\rho}_0, \varpi_0) \|^2_{H^{s_0}}\left(1+\| (v_0, \boldsymbol{\rho}_0, \varpi_0) \|_{H^{s_0}}\right)\exp \big( {\int^t_0} \|d\boldsymbol{\rho}, dv\|_{L^\infty_x}d\tau \big)
  \\
  & + \int^t_0 (\|d v\|_{L^\infty_x}+\| d\boldsymbol{\rho}\|_{L^\infty_x})(1+\|d v, d\boldsymbol{\rho}\|_{L^\infty_x})\| \varpi \|^2_{H^{s_0}} d\tau
  \\
  & +\int^t_0 (\|d v, d\boldsymbol{\rho}\|_{\dot{B}_{\infty,2}^{s_0-2}})(1+\|d v, d\boldsymbol{\rho}\|_{L^\infty_x})\| \varpi \|^2_{H^{s_0}} d\tau.
\end{split}
\end{equation*}
Using Gronwall inequality, we can get \eqref{WL}. Hence, we complete the proof of Theorem \ref{ve}.
\end{proof}
Using Theorem \ref{dv} and \ref{ve}, we can derive the following theorem
\begin{theorem}\label{be}{(Total energy estimates)}
Let $v$ and ${\rho}$ be a solution of \eqref{CEE}. Let $\boldsymbol{\rho}$ and $\varpi$ be defined as \eqref{pw1}. Set the energy
\begin{equation*}
  E(t)=\| \boldsymbol{\rho}(t)\|_{H^s}+\|v(t)\|_{H^s}+\|\varpi(t)\|_{H^{s_0}},
\end{equation*}
and
\begin{equation*}
   E_l(t)=\| \boldsymbol{\rho}(t)\|_{H^2}+\|v(t)\|_{H^2}+\|\varpi(t)\|_{H^{2}}.
\end{equation*}
Then the following estimates
\begin{equation}\label{E7}
\begin{split}
 & E(t)
 \lesssim E(0)  \exp \big ( {\int^t_0} (\|dv, d\boldsymbol{\rho}\|_{L^\infty_x}+\|\partial v\|_{\dot{B}^{s_0-2}_{\infty,2}})d\tau \big),
\end{split}
\end{equation}
and
\begin{equation}\label{W2e}
\begin{split}
 & E_l(t)
 \lesssim E_l(0)  \exp \big ( {\int^t_0} (\|dv\|_{L^\infty_x}+ \| d\boldsymbol{\rho}\|_{L^\infty_x})d\tau \big),
\end{split}
\end{equation}
holds.
\end{theorem}
Based on these basic energy estimates, we are ready to give a stability theorem as follows.
\subsection{Stability theorem}
\begin{theorem}\label{St}{(Stability theorem)}
	Let $s>s_0>2$. Suppose that $(v,\boldsymbol{\rho},\varpi) $ is a solution of \eqref{fc0} with initial data $v_0, \boldsymbol{\rho}_0, \varpi_0$ and there is a universal $C$ such that
\begin{equation*}
  \|v,\boldsymbol{\rho}\|_{L^\infty_t H^s}+\|\varpi\|_{L^\infty_t H^{s_0}}+\|dv,d\boldsymbol{\rho} \|_{L^2_t L^\infty_x}+ \| dv,d\boldsymbol{\rho}, \partial v_{+} \|_{L^2_t\dot{B}^{s_0-2}_{\infty,2}} \leq C.
\end{equation*}
	Let $(\varphi, \psi, V )$ be another solution to \eqref{fc0} with the initial data $(\varphi_0, \psi_0) \in H^s, V_0 \in H^{s_0}$ such that
$\varphi, \psi, V \in H^s$, $d\varphi,d\psi \in {L^2_t L^\infty_x}$ and $d\varphi_{+}, d\psi \in L^2_t \dot{B}^{s_0-2}_{\infty,2}$. Then the following estimate
\begin{equation}\label{s}
\begin{split}
  \|(v-\varphi, \boldsymbol{\rho}-\psi)(t,\cdot)\|_{H^{s-1}}+\|(\varpi-V)(t,\cdot)\|_{H^{s_0-2}} &\lesssim \|(v_0-\varphi_0, \boldsymbol{\rho}_0-\psi_0)\|_{H^{s-1}}
  \end{split}
\end{equation}
holds.
\end{theorem}
\begin{proof}
Let $U=(\rho, v)^T$ and $B=(\psi,\varphi)^T$. Then
\begin{equation*}
\begin{split}
  &\partial_t U+ \sum^3_{i=1 } A_i(U)\partial_{x_i}U=0,
  \\
  & \partial_t B+ \sum^3_{i=1 } A_i(B)\partial_{x_i}B=0.
  \end{split}
\end{equation*}
Then $U-B$ satisfies
\begin{equation*}
\begin{split}
  &\partial_t (U-B)+ \sum^3_{i=1 }A_i(U)\partial_{x_i}(U-B)=F,
  \end{split}
\end{equation*}
where
$$F=-\sum^3_{i=1}(A_i(U)-A_i(B))\partial_{x_i}B.$$
The standard energy estimates implies that
\begin{equation*}
  \frac{d}{dt}\|U-B\|_{H^{s-1}} \leq C_{U, B} \left( \| dU, dB\|_{L^\infty_x} \|U-B\|_{H^{s-1}}+ \|U-B\|_{L^\infty_x}\| \partial B\|_{H^{s-1}} \right),
\end{equation*}
where $C_{U, B}$ depends on the $L^\infty_x$ norm of $U, B$. By using $v, \boldsymbol{\rho}, \varpi, \varphi,\psi, V \in L^\infty_tH^s$, $dv, d\boldsymbol{\rho}, d\varphi,d\psi \in {L^4_t L^\infty_x}$ and $\partial V, \partial \varpi \in L^\infty$, we can derive that
\begin{equation}\label{UB}
\begin{split}
  \|(U-B)(t,\cdot)\|_{H^{s-1}} &\leq C_{U, B} \|(U-B)(0,\cdot)\|_{H^{s}}
  \\
  &
  \leq C_{U, B} \|(v_0-\varphi_0, \boldsymbol{\rho}_0-\psi_0) \|_{H^{s}}.
  \end{split}
\end{equation}
Since
\begin{equation*}
  \varpi=\bar{\rho} \mathrm{e}^{-\boldsymbol{\rho}}\mathrm{curl}v, \quad V=\bar{\rho} \mathrm{e}^{-\psi}\mathrm{curl}\varphi,
\end{equation*}
then
\begin{equation*}
  \begin{split}
  \varpi-V=\bar{\rho} \mathrm{e}^{-\boldsymbol{\rho}}(\mathrm{curl}v-\mathrm{curl}\varphi)+(\bar{\rho} \mathrm{e}^{-\boldsymbol{\rho}}-\bar{\rho} \mathrm{e}^{-\psi})\mathrm{curl}\varphi,
  \end{split}
\end{equation*}
we derive the following estimate by Lemma \ref{ps}:
\begin{equation}\label{20}
\begin{split}
  & \|(\varpi-V)(t,\cdot)\|_{H^{s_0-2}}
  \\
  \leq \ & \|\boldsymbol{\rho}\|_{H^{s_0}} \|\mathrm{curl}v-\mathrm{curl}\varphi\|_{H^{s_0-2}}+ \|\boldsymbol{\rho}-\psi \|_{H^{s_0-2}} \|\mathrm{curl}\varphi\|_{H^{s_0}}.
\end{split}
\end{equation}
For $s> s_0>2$, it follows that
\begin{equation*}
  \|\mathrm{curl}\varphi\|_{H^{s_0}}=\|V \cdot \psi\|_{H^{s_0}} \lesssim \|V\|_{H^{s_0}} \cdot \|\psi\|_{H^{s_0}}.
\end{equation*}
By using elliptic estimate, we have
\begin{equation*}
  \|\mathrm{curl}v-\mathrm{curl}\varphi\|_{H^{s_0-2}} \lesssim \|v-\varphi\|_{H^{s_0-1}}.
\end{equation*}
The above estimates combining with \eqref{20} tell us that
\begin{equation}\label{WV}
\begin{split}
  \|(\varpi-V)(t,\cdot)\|_{H^{s_0-2}} &\leq C_{\boldsymbol{\rho},\varphi} (\| v-\phi\|_{H^{s_0-1}}+\|\boldsymbol{\rho}-\psi \|_{H^{s_0-2}})
  \\
  & \leq C_{\boldsymbol{\rho},\varphi} \|(v_0-\varphi_0, \boldsymbol{\rho}_0-\psi_0) \|_{H^{s-1}},
\end{split}
\end{equation}
where $C_{\boldsymbol{\rho},\varphi}$ is a constant depending with $ \| \boldsymbol{\rho} \|_{H^{s_0}}$ and $ \| \varphi \|_{H^{s_0}}$. From \eqref{UB} and \eqref{WV}, we complete the proof of Theorem \ref{St}.
\end{proof}
As a direct result of Theorem \ref{St}, we can deduce:
\begin{corollary}\label{cor}{(Uniqueness of solution)}
Let $s>s_0>2$. Supppose $(v,\boldsymbol{\rho},\varpi)$ and $(\varphi, \psi, V )$ be solutions of \eqref{fc0} with initial data $v_0, \boldsymbol{\rho}_0$ satisfying $v_0, \boldsymbol{\rho}_0 \in H^s, \varpi_0=\bar{\rho} \mathrm{e}^{-\boldsymbol{\rho}_0}\mathrm{curl}v_0 \in H^{s_0}$. Then we have
\begin{equation*}
  v=\varphi, \quad \boldsymbol{\rho}=\psi, \quad \varpi=V .
\end{equation*}
\end{corollary}
\section{Reduction to the case of smooth initial data}
In this part, we reduce Theorem \ref{dingli} to the case of smooth initial data by compactness arguments.
\begin{proposition}\label{p3}
Let $s>s_0>2$. For each $M_0>0$, there exists $T, M, C>0$ such that, for each smooth initial data $(v_0, \boldsymbol{\rho}_0, \varpi_0)$ which satisfies
\begin{equation}\label{2000}
	\begin{split}
	&\|(v_0, \boldsymbol{\rho}_0) \|_{H^s} + \| \varpi_0\|_{H^{s_0}}  \leq M_0,
	\end{split}
	\end{equation}
there exists a smooth solution $(v, \boldsymbol{\rho}, \varpi)$ to \eqref{fc1} on $[-T,T] \times \mathbb{R}^3$ satisfying
\begin{equation}\label{e9}
 \|(v, \boldsymbol{\rho})\|_{H^s}+\|\varpi\|_{H^{s_0}} \leq M.
\end{equation}
Furthermore, the solution satisfies the conditions

$\mathrm{(1)}$ dispersive estimate for $v$, $\boldsymbol{\rho}$ and $v_+$
	\begin{equation}\label{303}
	\|d v, d \boldsymbol{\rho}\|_{L^2_t C^\delta_x}+\|\partial v_+, d \boldsymbol{\rho}, d v\|_{L^2_t \dot{B}^{s_0-2}_{\infty,2}} \leq M,
	\end{equation}

$\mathrm{(2)}$ Let $f$ satisfy equation \eqref{linear}. For each $1 \leq r \leq s_0+1$, the Cauchy problem \eqref{linear} is well-posed in $H^r \times H^{r-1}$, and the following estimate holds:
	\begin{equation}\label{304}
	\|\left< \partial \right>^k f\|_{L^2_t L^\infty_x} \lesssim  \| f_0\|_{H^r}+ \| f_1\|_{H^{r-1}}+\| G\|_{L^\infty_tH^{r-1}\cap L^1_t H^r}+ \| B \|_{L^1_t H^{r-1}},\ \ k<r-1.
	\end{equation}
\end{proposition}
In the following, we will use Proposition \ref{p3} to prove Theorem \ref{dingli}.
\begin{proof}[proof of Theorem \ref{dingli}]
Consider arbitrary initial data $(v_0, \boldsymbol{\rho}_0) \in H^s$, $\varpi_0 \in H^{s_0}$ satisfying
\begin{equation*}
	\|v_0\|_{H^s} + \|\boldsymbol{\rho}_0 \|_{H^s} + \|\varpi_0\|_{H^{s_0}} \leq R.
\end{equation*}
Let $(v_0^k, \boldsymbol{\rho}_0^k, \varpi_0^k)$ be a sequence of smooth data converging to $(v_0, \boldsymbol{\rho}_0, \varpi_0)$, which also satisfy the same bound. By Proposition \ref{p3}, there exists the corresponding solutions $(v^k, \boldsymbol{\rho}^k, \varpi^k)$ satisfying \eqref{wte}.

Note that the solutions of \eqref{CEE} also satisfy the symmetric hyperbolic system \eqref{sh}. Set $U^k=(\boldsymbol{\rho}^k, v^k,), k \in \mathbb{Z}^+$. For $j, l \in \mathbb{Z}^+$, we then have
\begin{equation*}
\begin{split}
  &\partial_t U^j+ \sum^3_{i=1}A_i(U^k)\partial_{i}U^j=0,
  \\
  & \partial_t U^l+ \sum^3_{i=1}A_i(U^l)\partial_{i}U^l=0.
  \end{split}
\end{equation*}
The standard energy estimates implies that
\begin{equation*}
  \frac{d}{dt}\|U^j-U^l\|_{H^{s-1}} \leq C_{U^j, U^l} \left(\| dU^j, dU^l\|_{L^\infty_x}\|U^j-U^l\|_{H^{s-1}}+ \|U^j-U^l\|_{L^\infty_x}\| \partial U^l\|_{H^{s-1}} \right),
\end{equation*}
where $C_{U^j, U^l}$ depends on the $L^\infty_x$ norm of $U^j, U^l$. By Strichartz estimates of $dv^k, d\boldsymbol{\rho}^k, k \in \mathbb{Z}^+$ in Proposition \ref{p3}, we could derive that
\begin{equation*}
\begin{split}
  \|(U^j-U^l)(t,\cdot)\|_{H^{s-1}} &\lesssim \|(U^j-U^l)(0,\cdot)\|_{H^{s-1}}
  \\
  &\lesssim\|(v_0^j-v_0^l, \boldsymbol{\rho}^j_0-\boldsymbol{\rho}^l_0) \|_{H^{s-1}}.
  \end{split}
\end{equation*}
As a result, $\{(v^k,\boldsymbol{\rho}^k)\}_{k=1}^{\infty}$ is a Cauchy sequence in $C([-T,T];H^{s-1})$. Denote
\begin{equation*}
  \lim_{k\rightarrow \infty}(v^k,\boldsymbol{\rho}^k)=(v, \boldsymbol{\rho}).
\end{equation*}
Thus, $v, \boldsymbol{\rho}$ is in $C([-T,T];H^{s-1})$. Since
\begin{equation*}
 \varpi^k=\bar{\rho}^{-1}\mathrm{e}^{-\boldsymbol{\rho}^k}\mathrm{curl}v^k, \quad k \in \mathbb{Z}^+.
\end{equation*}
Due to product estimates and elliptic estimates, we have
\begin{equation*}
\begin{split}
  \|\varpi^k-\varpi^l\|_{H^{s_0-2}} &\lesssim (\| v^k-v^l\|_{H^{s_0-1}}+\|\boldsymbol{\rho}^k-\boldsymbol{\rho}^l\|_{H^{s_0-1}})
  \\
  &\lesssim (\| v_0^k-v_0^l\|_{H^{s}}+\|\boldsymbol{\rho}^k_0-\boldsymbol{\rho}^l_0 \|_{H^{s}}).
\end{split}
\end{equation*}
Above, in the last line we use Lemma \ref{jh0}. This implies that $\varpi^k$ is a Cauchy sequence in $C([-T,T];H^{s_0-2})$. We denote the limit
\begin{equation*}
  \lim_{k\rightarrow \infty} \varpi^k = \varpi,
\end{equation*}
and $W$ is in $C([-T,T];H^{s_0-2})$.

Since $v^k,\boldsymbol{\rho}^k$ is uniformly bounded in $C([-T,T];H^{s})$ and $\varpi^k$ is bounded in $C([-T,T];H^{s_0})$ respectively. Thus, $(v, \boldsymbol{\rho}) \in C([-T,T];H^{s}), \varpi \in C([-T,T];H^{s_0})$. On the other hand, using Proposition \ref{p3}, $dv^k$ and $d\boldsymbol{\rho}^k$ are uniformly bounded in $L^2([-T,T];C^\delta)$. Consequently, $\{(dv^k, d\boldsymbol{\rho}^k)\}_{k \geq 1}$ converges to $(dv, d\boldsymbol{\rho})$ in $L^2([-T,T];L^\infty)$.

We also have $d v^k$ and $d\boldsymbol{\rho}^k$ bounded in $L^2([-T,T];\dot{B}^{s_0-2}_{\infty,2})$. This combines with $\{(dv^k, d\boldsymbol{\rho}^k)\}_{k \geq 1}$ converging to $(dv, d\boldsymbol{\rho})$ in $L^2([-T,T];L^\infty)$. We have $d v, d\boldsymbol{\rho} \in L^2([-T,T];\dot{B}^{s_0-2}_{\infty,2})$. Set
$$v=v_{+}+\eta, \quad \eta=(-\Delta)^{-1} (\mathrm{e}^{\boldsymbol{\rho}} \mathrm{curl} \varpi). $$
We can derive
\begin{equation}\label{40}
\begin{split}
  \| \partial v_{+}\|_{\dot{B}^{s_0-2}_{\infty,2}} \leq \| \partial v\|_{\dot{B}^{s_0-2}_{\infty,2}}+\| \partial \eta\|_{\dot{B}^{s_0-2}_{\infty,2}}.
  \end{split}
\end{equation}
By Sobolev inequality, elliptic estimate, and \eqref{W2e}, we have
\begin{equation*}
\begin{split}
  \| \partial \eta\|_{\dot{B}^{s_0-2}_{\infty, 2}} & \lesssim \| \partial(-\Delta)^{-1}(\mathrm{e}^{\boldsymbol{\rho}} \mathrm{curl} \varpi) \|_{\dot{B}^{s_0-\frac{1}{2}}_{2,2}}
   \\
  & \lesssim \|\mathrm{e}^{\boldsymbol{\rho}} \mathrm{curl} \varpi\|_{\dot{H}^{s_0-\frac{3}{2}}}
 \\
 & \lesssim (1+ \|\boldsymbol{\rho}\|_{H^{2}}) \|\varpi \|_{H^{s_0-\frac{1}{2}}}
 \\
 & \lesssim (1+\|\boldsymbol{\rho}_0\|_{H^{2}})\|\varpi_0 \|_{H^{2}}\exp{(2\int^t_0 \|dv,d\boldsymbol{\rho}\|_{L^\infty_x}d\tau)}.
 \end{split}
\end{equation*}
As a result, we have $\partial \eta \in L^2([-T,T];\dot{B}^{s_0-2}_{\infty,2})$. So we can obtain $\partial v_{+} \in L^2([-T,T];\dot{B}^{s_0-2}_{\infty,2})$. At this stage, we complete of proof of Theorem \ref{dingli}.
\end{proof}

\section{Reduction to existence for small, smooth, compactly supported data}
In this section, our goal is to give a reduction of Proposition \ref{p3} to the existence for small, smooth, compactly supported data using physical localization arguments.
\begin{proposition}\label{p1}
	Let $s>s_0>2$. Assume \eqref{a0} and \eqref{a1} hold. Suppose the initial data $(v_0, \boldsymbol{\rho}_0, \varpi_0)$ be smooth, supported in $B(0,c+2)$ and satisfying
	\begin{equation}\label{300}
	\begin{split}
	&\|v_0\|_{H^s} + \|\boldsymbol{\rho}_0 \|_{H^s} + \|\varpi_0\|_{H^{s_0}} \leq \epsilon_3.
	\end{split}
	\end{equation}
	Then the Cauchy problem \eqref{fc1} admits a smooth solution $v,\boldsymbol{\rho},\varpi$ on $[-1,1] \times \mathbb{R}^3$, which have the following properties:
	
	$\mathrm{(1)}$ energy estimate
	\begin{equation}\label{402}
	\begin{split}
	&\|v\|_{L^\infty_t H^{s}}+\| \boldsymbol{\rho}\|_{L^\infty_t H^{s}} + \| \varpi\|_{H^{s_0}} \leq \epsilon_2.
	\end{split}
	\end{equation}

	$\mathrm{(2)}$ dispersive estimate for $v$ and $\boldsymbol{\rho}$
	\begin{equation}\label{s403}
	\|d v, d \boldsymbol{\rho}\|_{L^2_t C^\delta_x}+\| d \boldsymbol{\rho}, \partial v_{+}, d v\|_{L^2_t \dot{B}^{s_0-2}_{\infty,2}} \leq \epsilon_2,
	\end{equation}

	$\mathrm{(3)}$ dispersive estimate for the linear equation

Let $f$ satisfy
 the equation \eqref{linear}.
For each $1 \leq r \leq s_0+1$, the Cauchy problem \eqref{linear} is well-posed in $H^r \times H^{r-1}$, and the following estimate holds:
	\begin{equation}\label{304}
	\|\left< \partial \right>^k f\|_{L^2_t L^\infty_x} \lesssim  \| f_0\|_{H^r}+ \| f_1\|_{H^{r-1}}+\| G\|_{L^\infty_tH^{r-1}\cap L^1_t H^r}+ \| B \|_{L^1_t H^{r-1}},\ \ k<r-1.
	\end{equation}

\end{proposition}
\begin{proof}[proof of Proposition \ref{p1} by Proposition \ref{p3}] In Proposition \ref{p3}, the corresponding statement is considering the small, supported data. While, the initial data is large in Proposition \ref{p1}. Firstly, by using a scaling method, we can reduce the initial data in Proposition \ref{p1} to be small.

$\textbf{Step 1: Scaling}$. Assume
\begin{equation}\label{a4}
\begin{split}
&\|v_0\|_{H^s}+ \| \boldsymbol{\rho}_0\|_{H^s} + \|\varpi_0\|_{H^{s_0}}  \leq M_0.
\end{split}
\end{equation}
If taking the scaling
\begin{equation*}
\begin{split}
&\underline{{v}}(t,x)=v(Tt,Tx),\quad \underline{\boldsymbol{\rho}}(t,x)=\boldsymbol{\rho}(Tt,Tx), \quad \underline{\varpi}(t,x)={\varpi}(Tt,Tx)
\end{split}
\end{equation*}
we then obtain
\begin{equation*}
\begin{split}
&\|\underline{v}(0)\|_{\dot{H}^s}+\|\underline{\boldsymbol{\rho}}(0)\|_{\dot{H}^s} \leq M_0 T^{s-\frac{3}{2}}, \qquad \|\underline{\varpi}(0)\|_{\dot{H}^{s_0}} \leq M_0 T^{s_0-\frac{3}{2}}.
\end{split}
\end{equation*}
Let $\epsilon_3$ be stated in \eqref{a0}. Choose sufficiently small $T$ such that
\begin{equation*}
M_0 T^{s_0-\frac{3}{2}} \ll \epsilon_3.
\end{equation*}
Therefore, we get
\begin{equation*}
\begin{split}
  &\|\underline{v}(0)\|_{\dot{H}^s} \leq \epsilon_3, \quad \|\underline{\boldsymbol{\rho}}(0)\|_{\dot{H}^s} \leq \epsilon_3, \quad \|\underline{\varpi}(0)\|_{\dot{H}^{s_0}} \leq \epsilon_3.
\end{split}
\end{equation*}
Secondly, to reduce the initial data with support set, we need a physical localization technique.

$\textbf{Step 2: Localization}$.

The propagation speed of \eqref{fc1} is also finite. We then let $c$ be the largest speed of \eqref{fc1}. Set $\chi$ be a smooth function supported in $B(0,c+2)$, and which equals $1$ in $B(0,c+1)$. For given $y \in \mathbb{R}^3$, we define the localized initial data for the velocity and density near $y$:
\begin{equation*}
\begin{split}
\widetilde{v}_0(x)=&\chi(x-y)\left( v_0(x)-v_0(y)\right),
\\
\widetilde{ \boldsymbol{\rho}}_0(x)=&\chi(x-y)\left( \boldsymbol{\rho}_0(x)-\boldsymbol{\rho}_0(y)\right).
\end{split}
\end{equation*}
Based on these quantities $v^y_0$ and $\boldsymbol{\rho}_0^y$. Followed by \eqref{pw1}, we set
\begin{equation}\label{wde}
  \widetilde{\varpi}_0=\bar{\rho} \mathrm{e}^{ -\widetilde{ \boldsymbol{\rho}}_0}\mathrm{curl}\widetilde{v}_0=\bar{\rho} \mathrm{e}^{ -\widetilde{ \boldsymbol{\rho}}_0}\widetilde{\omega}_0,
\end{equation}
then $\widetilde{\varpi}_0$ is the specific vorticity.
Since $s,s_0\in (2, \frac{5}{2})$, we can verify
\begin{equation}\label{245}
\| \widetilde{v}_0, \widetilde{\boldsymbol{\rho}}_0\|_{H^s} \lesssim \|v_0, \boldsymbol{\rho}_0 \|_{\dot{H}^s} \lesssim \epsilon_3.
\end{equation}
By calculation on \eqref{wde}, we have
\begin{equation}\label{WY1}
   \widetilde{\varpi}_0= \bar{\rho}^{-1} \mathrm{e}^{ -\widetilde{ \boldsymbol{\rho}}_0} \nabla \chi(x-y)\cdot (v_0(x)-v_0(y))+\chi(x-y) \bar{\rho}^{-1} \mathrm{e}^{ -\widetilde{ \boldsymbol{\rho}}_0}\mathrm{curl}{v}_0.
\end{equation}
Hence, we derive that
\begin{equation}\label{W0e}
  \| \widetilde{\varpi}_0 \|_{L^2} \lesssim \| \partial v_0 \|_{L^2} \lesssim \epsilon_3.
\end{equation}
Substituting $\varpi_0=\bar{\rho}^{-1} \mathrm{e}^{ -{ \boldsymbol{\rho}}_0}\mathrm{curl}v_0$ into \eqref{WY1}, we can update \eqref{WY1} as
\begin{equation*}
   \widetilde{\varpi}_0= \bar{\rho}^{-1} \mathrm{e}^{- \widetilde{ \boldsymbol{\rho}_0}} \nabla \chi(x-y)\cdot (v_0(x)-v_0(y))+\chi(x-y) \mathrm{e}^{ -\widetilde{ \boldsymbol{\rho}_0}+\boldsymbol{\rho}_0}\varpi_0.
\end{equation*}
Therefore, we can obtain
\begin{equation}\label{Whe}
\begin{split}
\| \widetilde{\varpi}_0 \|_{\dot{H}^{s_0}} &= \| \bar{\rho}^{-1} \mathrm{e}^{ - \widetilde{ \boldsymbol{\rho}_0}} \nabla \chi(x-y)\cdot (v_0(x)-v_0(y))+\chi(x-y) \mathrm{e}^{ -\widetilde{ \boldsymbol{\rho}_0}+\boldsymbol{\rho}_0}\varpi_0 \|_{\dot{H}^{s_0}}
\\
&\lesssim \|\varpi_0 \|_{\dot{H}^{s_0}}+ \|v_0 \|_{\dot{H}^{s_0}}+\| \boldsymbol{\rho}_0\|_{\dot{H}^{s_0}} \lesssim \epsilon_3.
\end{split}
\end{equation}
Adding \eqref{W0e} and \eqref{Whe}, we can get
\begin{equation}\label{246}
\| \widetilde{\varpi}_0 \|_{H^{s_0}} \lesssim \epsilon_3.
\end{equation}
By Proposition \ref{p1}, there is a smooth solution $(\widetilde{v}, \widetilde{\boldsymbol{\rho}}, \widetilde{\varpi})$ on $[-1,1]\times \mathbb{R}^3$ satisfying the following equation
\begin{equation}\label{p}
\begin{cases}
&\square_{\widetilde{{g}}} \widetilde{v}=-\mathrm{e}^{\widetilde{\boldsymbol{\rho}}}\tilde{c}^2_s\mathrm{curl} \widetilde{\varpi}+\widetilde{Q},
\\
&\square_{\widetilde{{g}}} \widetilde{\boldsymbol{\rho}}=\widetilde{\mathcal{D}},
\\
& \mathbf{T} \widetilde{\varpi}=(\widetilde{\varpi} \cdot \nabla) \widetilde{v}.
\end{cases}
\end{equation}
Above, $\tilde{c}^2_s, \widetilde{\mathcal{D}}, \widetilde{Q}$, and $\tilde{{g}}$ are defined by
\begin{equation}\label{DDE}
\begin{split}
\tilde{c}^2_s:&= \frac{dp(\widetilde{\boldsymbol{\rho}})}{d\widetilde{\boldsymbol{\rho}}}, \quad \tilde{c}'_s:= \frac{d\tilde{c}_s}{d\widetilde{\boldsymbol{\rho}}}
\\
 \widetilde{Q}:&=2e^{\widetilde{\boldsymbol{\rho}}}\epsilon^i_{ab}\mathbf{T}\widetilde{v}^a(\widetilde{\varpi})^b-\left( 1+\tilde{c}_s^{-1}\tilde{c}'_s\right)\tilde{{g}}^{\alpha \beta} \partial_\alpha \widetilde{\boldsymbol{\rho}} \partial_\beta \widetilde{v},
\\
\widetilde{\mathcal{D}}:&=-3\tilde{c}_s^{-1}\tilde{c}'_s\tilde{{g}}^{\alpha \beta} \partial_\alpha \widetilde{\boldsymbol{\rho}} \partial_\beta \widetilde{\boldsymbol{\rho}}+2\textstyle{\sum_{1 \leq a < b \leq 3}} \big\{ \partial_a \widetilde{v}^a \partial_b \widetilde{v}^b-\partial_a (\widetilde{v})^b \partial_b (\widetilde{v})^a \big\}
\\
\widetilde{g}:&=-dt\otimes dt+\tilde{c}_s^{-2}(\boldsymbol{\rho}^y)\textstyle{\sum_{a=1}^{3}}\left( dx^a-\widetilde{v}^adt\right)\otimes\left( dx^a-\widetilde{v}^adt\right),
\end{split}
\end{equation}
and the specific vorticity $\widetilde{\varpi}$ satisfies
\begin{equation}\label{Wy}
  \widetilde{\varpi}=\bar{\rho}^{-1} \mathrm{e}^{ -\widetilde{\boldsymbol{\rho}}}\mathrm{curl}\widetilde{v}.
\end{equation}
Set
\begin{equation}\label{dvp}
  \widetilde{v}_{+}:=\widetilde{v}-\widetilde{\eta}, \quad \Delta \widetilde{\eta}:=-\mathrm{e}^{\widetilde{\boldsymbol{\rho}}}\mathrm{curl}\widetilde{\varpi}.
\end{equation}
By Proposition \ref{p1} again, we can find the solution of \eqref{DDE} satisfying
\begin{equation}\label{see0}
  \|\widetilde{v}\|_{H^s}+ \|\widetilde{\boldsymbol{\rho}}\|_{H^s}+ \| \widetilde{\varpi} \|_{H^{s_0}} \leq \epsilon_2,
\end{equation}
and
\begin{equation}\label{see1}
  \|d\widetilde{v}, d\widetilde{\boldsymbol{\rho}}\|_{L^2_tC^\delta_x}+ \| d\widetilde{v}, \partial \widetilde{v}_{+}, d\widetilde{\boldsymbol{\rho}} \|_{L^2_t \dot{B}^{s_0-2}_{2,\infty}} \leq \epsilon_2.
\end{equation}
Furthermore, the linear equation
\begin{equation}\label{312}
\begin{cases}
&\square_{ \widetilde{{g}}} \tilde{f}=\mathbf{T}\widetilde{G}+\widetilde{B},
\\
&\tilde{f}(t_0,\cdot)=\widetilde{f}_0, \ \mathbf{T}\tilde{f}(t_0,)=\widetilde{f}_1,
\end{cases}
\end{equation}
admits a solution $\widetilde{f} \in C([0,T],H^r)\times C^1([0,T],H^{r-1})$, and the following estimate holds:
	\begin{equation}\label{s31}
	\|\left< \partial \right>^k \widetilde{f}\|_{L^2_t L^\infty_x} \lesssim  \| \widetilde{f}_0\|_{H^r}+ \| \widetilde{f}_1\|_{H^{r-1}}+\| \widetilde{G}\|_{L^\infty_tH^{r-1}\cap L^1_t H^r}+ \| \widetilde{B} \|_{L^1_t H^{r-1}},\ \ k<r-1.
	\end{equation}
Consequently, $(\widetilde{v}+v_0(y), \widetilde{\boldsymbol{\rho}}+\boldsymbol{ \rho}_0(y), \widetilde{\varpi})$ is also a solution of \eqref{p}, and its initial data coincides with $(v_0, \boldsymbol{\rho}_0, \varpi_0)$ in $B(y,c+1)$. Giving the restrictions, for $y\in \mathbb{R}^3$,
\begin{equation}\label{RS}
  \left( \widetilde{v}+v_0(y) \right)|_{\mathrm{K}^y},
  \quad \left(\widetilde{\boldsymbol{\rho}}+\boldsymbol{\rho}_0(y) \right)|_{\mathrm{K}^y},
  \quad \widetilde{\varpi}|_{\mathrm{K}^y},
\end{equation}
where $\mathrm{K}^y=\left\{ (t,x): ct+|x-y| \leq c+1, |t| <1 \right\}$, then the restrictions \eqref{RS} solve \eqref{CEE}-\eqref{id} on $\mathrm{K}^y$. By finite speed of propagation, a smooth solution $(v, \boldsymbol{ \rho}, \varpi)$ solves \eqref{fc1} in $[-1,1] \times \mathbb{R}^3$, where $v, \boldsymbol{ \rho}$ and $\varpi$ is denoted by
\begin{equation}\label{vw}
\begin{split}
  v(t,x)  &= \widetilde{v}+v_0(y), \ \ \ (t,x) \in \mathrm{K}^y,
  \\
   \boldsymbol{ \rho}(t,x)  &=\widetilde{\boldsymbol{ \rho}}+\boldsymbol{ \rho}_0(y), \ \ \ (t,x) \in \mathrm{K}^y,
  \\
   \varpi(t,x)  &=\widetilde{\varpi}, \ \ \qquad \ \quad (t,x) \in \mathrm{K}^y.
\end{split}
\end{equation}
We also can obtain the initial information
\begin{equation*}
 (v ,\boldsymbol{ \rho}, \varpi)|_{t=0}=(v_0, \boldsymbol{ \rho}_0, \varpi_0).
\end{equation*}
Therefore, the function $(v, \boldsymbol{ \rho}, \varpi)$ defined in \eqref{vw} is the solution of \eqref{fc1} according to the uniqueness of solutions, i.e. Corollary \ref{cor}. By \eqref{see0} and \eqref{see1}, we can obtain that $(v, \boldsymbol{ \rho}, \varpi)$ satisfies \eqref{e9} and \eqref{303}, which is stated in Proposition \ref{p3}. It remains for us to prove \eqref{304} in Proposition \ref{p3}. Consider the solution ${f}$ for
\begin{equation}\label{312}
\begin{cases}
&\square_{ {{g}}} {f}=\mathbf{T}G+B,
\\
&{f}(0)=f_0, \ \mathbf{T}{f}(0)=f_1.
\end{cases}
\end{equation}
Take
\begin{equation*}
\begin{split}
  &\tilde{f}_0=\chi(x-y)f_0, \quad  \tilde{f}_1=\chi(x-y)f_1,
  \\
  & \tilde{G}=\chi(x-y)G, \quad  \tilde{B}=\chi(x-y)B.
\end{split}
\end{equation*}
By finite speed of propagation, we can conclude that $\tilde{f}=f$ in $\mathrm{K}_y$.
We write ${f}$ as
\begin{equation*}
{f}(t,x)=\textstyle{\sum_{y \in n^{-\frac{1}{2}}\mathbb{Z}^n}} \psi(x-y)\widetilde{f}(t,x).
\end{equation*}
For $1\leq r \leq 1+s_0$ and $k<r-1$, using \eqref{304}, it follows that
\begin{equation*}
\begin{split}
\| \left< \partial \right>^k {f} \|^2_{L^2_t L^\infty_x}  \leq & \textstyle{\sum_{y \in n^{-\frac{1}{2}}\mathbb{Z}^n}} \|\psi(x-y)  \left< \partial \right>^k \tilde{f}(x,t)  \|^2_{L^2_t L^\infty_x}
\\
 \lesssim  & \textstyle{\sum_{y \in n^{-\frac{1}{2}}\mathbb{Z}^n}} \big(  \| \chi(x-y) f_0\|^2_{H^r \times H^{r}}+ \|\chi(x-y) f_1\|^2_{H^r \times H^{r-1}}
\\
& +\|\chi(x-y) G\|_{L^\infty_t H^{r-1}\cap L^1_t H^r}+\|\chi(x-y) B\|_{L^1_t H^{r-1}} \big),
\\
\lesssim  &   \| f_0\|^2_{H^r \times H^{r}}+ \| f_1\|^2_{H^r \times H^{r-1}}
 +\| G\|_{L^\infty_t H^{r-1}\cap L^1_t H^r}+\| B\|_{L^1_t H^{r-1}} .
\end{split}
\end{equation*}
At this stage, we have finished the proof.
\end{proof}

\section{A bootstrap argument}
Let $\mathbf{m}$ be a standard Minkowski metric satisfying
\begin{equation*}
  \mathbf{m}^{00}=-1, \quad \mathbf{m}^{ij}=\delta^{ij}, \quad i, j=1,2,3.
\end{equation*}
Taking $v=0, \boldsymbol{\rho}=0$ in $g$, the inverse matrix of the metric $g$ is
\begin{equation*}
g^{-1}(0)=
\left(
\begin{array}{cccc}
-1 & 0 & 0 & 0\\
0 & c^2_s(0) & 0& 0\\
0 & 0 & c^2_s(0) & 0
\\
0 & 0 & 0 &   c^2_s(0)
\end{array}
\right ).
\end{equation*}
By a linear change of coordinates which preserves $dt$, we may assume that $g^{\alpha \beta}(0)=\mathbf{m}^{\alpha \beta}$. Let $\chi$ be a smooth cut-off function supported in the region $B(0,3+2c) \times [-\frac{3}{2}, \frac{3}{2}]$, which equals to $1$ in the region $B(0,2+2c) \times [-1, 1]$. Set
\begin{equation*}
  \mathbf{g}=\chi(t,x)(g-g(0))+g(0).
\end{equation*}
Here
\begin{equation*}
  {g}=-dt\otimes dt+c_s^{-2}(\boldsymbol{\rho})\sum_{a=1}^{3}\left( dx^a-v^adt\right)\otimes\left( dx^a-v^adt\right).
\end{equation*}
Consider the following system
\begin{equation}\label{CS}
	\begin{cases}
	\square_{\mathbf{g}}v^i= -\mathrm{e}^{\boldsymbol{\rho}}c_s^2 \mathrm{curl} \varpi^i+Q^i,
\\
 \square_{\mathbf{g}} \boldsymbol{\rho} =  \mathcal{D},
    \\
    \mathbf{T}\varpi= (\varpi \cdot \nabla) v,
	\end{cases}
	\end{equation}
where  $Q^i, E^i$, and $\mathcal{D}$ are null forms relative to $\mathbf{g}$, which are defined by
\begin{equation}\label{CSN}
\begin{split}
 Q^i&:=2e^{\boldsymbol{\rho}} \epsilon^i_{ab} \mathbf{T} v^a W^b-\left( 1+c_s^{-1}c'_s\right)\mathbf{g}^{\alpha \beta} \partial_\alpha \boldsymbol{\rho} \partial_\beta v^i,
\\
\mathcal{D}&:=-3c_s^{-1}c'_s\mathbf{g}^{\alpha \beta} \partial_\alpha \boldsymbol{\rho} \partial_\beta \boldsymbol{\rho}+2 \textstyle{\sum_{1 \leq a < b \leq 3} }\big\{ \partial_a v^a \partial_b v^b-\partial_a v^b \partial_b v^a \big\}.
\end{split}
\end{equation}
We denote by $\mathcal{H}$ the family of smooth solutions $(v, \boldsymbol{\rho}, \varpi)$ to Equation \eqref{CS} for $t \in [-2,2]$, with the initial data $(v_0, \boldsymbol{\rho}_0, \varpi_0)$ supported in $B(0,2+c)$, and for which 
\begin{equation}\label{401}
\|v_0\|_{H^s} + \|\boldsymbol{\rho}_0\|_{H^s}+\| \varpi_0\|_{H^{s_0}}  \leq \epsilon_3,
\end{equation}
\begin{equation}\label{402}
 \| v\|_{L^\infty_t H^{s}}+\| \boldsymbol{\rho}\|_{L^\infty_t H^{s}}+\| \varpi\|_{L^\infty_t H^{s_0}} \leq 2 \epsilon_2,
\end{equation}
\begin{equation}\label{403}
  \| d v, d \boldsymbol{\rho}, \partial v_{+}\|_{L^2_t C_x^\delta}+ \|d \boldsymbol{\rho}, \partial v_{+}, dv\|_{L^2_t \dot{B}^{s_0-2}_{\infty,2}} \leq 2 \epsilon_2.
\end{equation}
Therefore, the bootstrap argument can be stated as follows:
\begin{proposition}\label{p4}
Assume that \eqref{a0} holds. Then there is a continuous functional $G: \mathcal{H} \rightarrow \mathbb{R}^{+}$, satisfying $G(0)=0$, so that for each $(v, \boldsymbol{\rho}, \varpi) \in \mathcal{H}$ satisfying $G(v, \boldsymbol{\rho}) \leq 2 \epsilon_1$ the following hold:

$\mathrm{(1)}$ The function $v, \boldsymbol{\rho}$, and $\varpi$ satisfies $G(v, \boldsymbol{\rho}) \leq \epsilon_1$.

$\mathrm{(2)}$ The following estimate holds,
\begin{equation}\label{404}
\|v\|_{L^\infty_t H_x^{s}}+ \|\boldsymbol{\rho}\|_{L^\infty_t H_x^{s}}+ \|\varpi\|_{L^\infty_t H_x^{s_0}} \leq \epsilon_2,
\end{equation}
\begin{equation}\label{405}
\|d v, d \boldsymbol{\rho}\|_{L^2_t C^\delta_x}+\|\partial v_{+}, d \boldsymbol{\rho}, dv\|_{L^2_t \dot{B}^{s_0-2}_{\infty,2}} \leq \epsilon_2.
\end{equation}

$\mathrm{(3)}$ For $1 \leq k \leq s_0+1$, the equation \eqref{linear} endowed with the metric $\mathbf{g}$ is well-posed in $H^r \times H^{r-1}$, and the Strichartz estimates \eqref{304} hold.
\end{proposition}
\begin{proof}[{proof of Proposition \ref{p1} by Proposition \ref{p4}}]
By using the standard continuity method in \cite{ST}, we can prove Proposition \ref{p1} by Proposition \ref{p4}.
\end{proof}
\section{Regularity of the characteristic hypersurface}
In this part, we will prove Proposition \ref{p4} by analysing the regularity of characteristic hypersurface and Strichartz estimates of linear wave equations, which is stated in Proposition \ref{r2}, Proposition \ref{r3}, and Proposition \ref{r4}. This well-known strategy is proposed by Smith-Tataru \cite{ST}. The functional $G(v, \boldsymbol{\rho}) \leq \epsilon_1$ will be defined later, and one can see \eqref{500} for details. Let $(v,\boldsymbol{\rho},\varpi) \in \mathcal{H}$, and the corresponding metric $ \mathbf{g}$ which equals the Minkowski metric for $t \in [-2, -\frac{3}{2}]$. Let $\Gamma_{\theta}$ be the flowout of this section under the Hamiltonian flow of $ \mathbf{g}$.  For each $\theta$, the null Lagrangian manifold $\Gamma_{\theta}$ is the graph of a null covector field given by $dr_{\theta}$, where $r_{\theta}$ is a smooth extension of $\theta \cdot x -t$, and that the level sets of $r_{\theta}$ are small perturbations of the level sets of the function $\theta \cdot x -t$ in a certain norm captured by $G$. We also let $\Sigma_{\theta,r}$ for $r \in \mathbb{R}$ denote the level sets of $r_{\theta}$. The characteristic hypersurface $\Sigma_{\theta,r}$ is thus the flowout of the set $\theta \cdot x=r-2$ along the null geodesic flow in the direction $\theta$ at $t=-2$.

Let us introduce an orthonormal sets of coordinates on $\mathbb{R}^3$ by setting $x_{\theta}=\theta \cdot x$. Let $x'_{\theta}$ be given orthonormal coordinates on the hyperplane prependicular to $\theta$, which then define coordinates on $\mathbb{R}^3$ by projection along $\theta$. Then $(t,x'_{\theta})$ induce the coordinates on $\Sigma_{\theta,r}$, and $\Sigma_{\theta,r}$ is given by
\begin{equation*}
  \Sigma_{\theta,r}=\left\{ (t,x): x_{\theta}-\phi_{\theta, r}=0  \right\}
\end{equation*}
for a smooth function $\phi_{\theta, r}(t,x'_{\theta})$. We now introduce two norms for functions defined on $[-2,2] \times \mathbb{R}^3$,
\begin{equation}\label{d0}
  \begin{split}
  &\vert\kern-0.25ex\vert\kern-0.25ex\vert u\vert\kern-0.25ex\vert\kern-0.25ex\vert_{s_0, \infty} = \sup_{-2 \leq t \leq 2} \sup_{0 \leq j \leq 1} \| \partial_t^j u(t,\cdot)\|_{H^{s_0-j}(\mathbb{R}^3)},
  \\
  & \vert\kern-0.25ex\vert\kern-0.25ex\vert  u\vert\kern-0.25ex\vert\kern-0.25ex\vert_{s_0,2} = \big( \sup_{0 \leq j \leq 1} \int^{2}_{-2} \| \partial_t^j u(t,\cdot)\|^2_{H^{s_0-j}(\mathbb{R}^3)} dt \big)^{\frac{1}{2}}.
  \end{split}
\end{equation}
The same notation applies for functions in $[-2,2] \times \mathbb{R}^3$. We denote
\begin{equation*}
  \vert\kern-0.25ex\vert\kern-0.25ex\vert f\vert\kern-0.25ex\vert\kern-0.25ex\vert_{s_0,2,\Sigma_{\theta,r}}=\vert\kern-0.25ex\vert\kern-0.25ex\vert f|_{\Sigma_{\theta,r}} \vert\kern-0.25ex\vert\kern-0.25ex\vert_{s_0,2},
\end{equation*}
where the right hand side is the norm of the restriction of $f$ to ${\Sigma_{\theta,r}}$, taken over the $(t,x'_{\theta})$ variablles used to parametrise ${\Sigma_{\theta,r}}$. Similarly, the notation
\begin{equation*}
  \|f\|_{H^{s_0}(\Sigma_{\theta,r})}
\end{equation*}
denotes the $H^{s_0}(\mathbb{R}^2)$ norm of $f$ restricted to the time $t$ slice of ${\Sigma_{\theta,r}}$ using the $x'_{\theta}$ coordinates on ${\Sigma^t_{\theta,r}}$.

We now set
\begin{equation}\label{500}
  G(v, \boldsymbol{\rho})= \sup_{\theta, r} \vert\kern-0.25ex\vert\kern-0.25ex\vert d \phi_{\theta,r}-dt\vert\kern-0.25ex\vert\kern-0.25ex\vert_{s_0,2,{\Sigma_{\theta,r}}}.
\end{equation}
\begin{proposition}\label{r1}
Let $(v, \boldsymbol{\rho}, \varpi) \in \mathcal{H}$ so that $G(v, \boldsymbol{\rho}) \leq 2 \epsilon_1$. Then
\begin{equation}\label{501}
  \vert\kern-0.25ex\vert\kern-0.25ex\vert  {\mathbf{g}}^{\alpha \beta}-\mathbf{m}^{\alpha \beta} \vert\kern-0.25ex\vert\kern-0.25ex\vert_{s_0,2,{\Sigma_{\theta,r}}} + \vert\kern-0.25ex\vert\kern-0.25ex\vert 2^{j}({\mathbf{g}}^{\alpha \beta}-S_j {\mathbf{g}}^{\alpha \beta}), d \Delta_j {\mathbf{g}}^{\alpha \beta}, 2^{-j} \partial_x \Delta_j d {\mathbf{g}}^{\alpha \beta}  \vert\kern-0.25ex\vert\kern-0.25ex\vert_{s_0-1,2,{\Sigma_{\theta,r}}} \lesssim \epsilon_2.
\end{equation}
\end{proposition}
\begin{proposition}\label{r2}
Let $(v, \boldsymbol{\rho}, \varpi) \in \mathcal{H}$ so that $G(v, \boldsymbol{\rho}) \leq 2 \epsilon_1$. Then

\begin{equation}\label{G}
G(v, \boldsymbol{\rho}) \lesssim \epsilon_2.
\end{equation}
Furthermore, for each $t$ it holds that
\begin{equation}\label{502}
  \|d \phi_{\theta,r}(t,\cdot)-dt \|_{C^{1,\delta}_{x'}} \lesssim \epsilon_2+  \| d {\mathbf{g}}(t,\cdot) \|_{C^\delta_x(\mathbb{R}^3)}.
\end{equation}
\end{proposition}
\subsection{Energy estimates on the characteristic hypersurface}
Let $(v, \boldsymbol{\rho}, \varpi) \in \mathcal{H}$. Then the following estimates hold:
\begin{equation}\label{5021}
   \vert\kern-0.25ex\vert\kern-0.25ex\vert v,\boldsymbol{\rho}\vert\kern-0.25ex\vert\kern-0.25ex\vert_{s,\infty}+\vert\kern-0.25ex\vert\kern-0.25ex\vert \varpi\vert\kern-0.25ex\vert\kern-0.25ex\vert_{s_0,\infty} +\|dv,d\boldsymbol{\rho}, \partial v_{+}\|_{L^2_tC^\delta_x}+ \|d\boldsymbol{\rho}, \partial v_{+}, d v\|_{L^2_t \dot{B}^{s_0-2}_{\infty,2}} \lesssim \epsilon_2.
\end{equation}
It suffices to prove Proposition \ref{r1} and Proposition \ref{r2} for $\theta=(0,0,1)$ and $r=0$. We fix this choice, and suppress $\theta$ and $r$ in our notation. We use $(x_3, x')$ instead of $(x_{\theta}, x'_{\theta})$. Then $\Sigma$ is defined by
\begin{equation*}
  \Sigma=\left\{ x_3- \phi(t,x')=0 \right\}.
\end{equation*}
The hypothesis $G \leq 2 \epsilon_1$ implies that
\begin{equation}\label{503}
 \vert\kern-0.25ex\vert\kern-0.25ex\vert d \phi_{\theta,r}(t,\cdot)-dt \vert\kern-0.25ex\vert\kern-0.25ex\vert_{s_0,2, \Sigma} \leq 2 \epsilon_1.
\end{equation}
By using Sobolev imbedding, we have
\begin{equation}\label{504}
  \|d \phi(t,x')-dt \|_{L^2_t C^{1,\delta}_{x'}} + \| \partial_t d \phi(t,x')\|_{L^2_t C^{\delta}_{x'}} \lesssim \epsilon_1.
\end{equation}
Let us now introduce two lemmas in Smith-Tataru's paper \cite{ST}.
\begin{Lemma}\label{te0}\cite{ST}
Let $\tilde{h}(t,x)=h(t,x',x_3+\phi(t,x'))$. Then we have
\begin{equation*}
  \vert\kern-0.25ex\vert\kern-0.25ex\vert \tilde{h}\vert\kern-0.25ex\vert\kern-0.25ex\vert_{s_0,\infty}\lesssim \vert\kern-0.25ex\vert\kern-0.25ex\vert h\vert\kern-0.25ex\vert\kern-0.25ex\vert_{s_0,\infty}, \quad \|d\tilde{h}\|_{L^2_tL^\infty}\lesssim \|d h\|_{{L^2_tL^\infty}}, \quad  \|\tilde{h}\|_{H^{s_0}_{x}}\lesssim \|h\|_{H^{s_0}_{x}}.
\end{equation*}
\end{Lemma}
\begin{Lemma}\label{te2}\cite{ST}
For $r\geq 1$, we have
\begin{equation*}
\begin{split}
  \sup_{t\in[-2,2]} \| f\|_{H^{r-\frac{1}{2}}(\mathbb{R}^n)} & \lesssim \vert\kern-0.25ex\vert\kern-0.25ex\vert f \vert\kern-0.25ex\vert\kern-0.25ex\vert_{r,2},
  \\
  \sup_{t\in[-2,2]} \| f\|_{H^{r-\frac{1}{2}}(\Sigma^t)} & \lesssim \vert\kern-0.25ex\vert\kern-0.25ex\vert f \vert\kern-0.25ex\vert\kern-0.25ex\vert_{r,2,\Sigma}.
\end{split}
\end{equation*}
If $r> \frac{n+1}{2}$, then
\begin{equation*}
  \vert\kern-0.25ex\vert\kern-0.25ex\vert hf\vert\kern-0.25ex\vert\kern-0.25ex\vert_{r,2}\lesssim  \vert\kern-0.25ex\vert\kern-0.25ex\vert h\vert\kern-0.25ex\vert\kern-0.25ex\vert_{r,2} \vert\kern-0.25ex\vert\kern-0.25ex\vert f\vert\kern-0.25ex\vert\kern-0.25ex\vert_{r,2}.
\end{equation*}
Similarly, if $r>\frac{n}{2}$, then
\begin{equation*}
  \vert\kern-0.25ex\vert\kern-0.25ex\vert hf\vert\kern-0.25ex\vert\kern-0.25ex\vert_{r,2,\Sigma}\lesssim  \vert\kern-0.25ex\vert\kern-0.25ex\vert h\vert\kern-0.25ex\vert\kern-0.25ex\vert_{r,2,\Sigma} \vert\kern-0.25ex\vert\kern-0.25ex\vert f\vert\kern-0.25ex\vert\kern-0.25ex\vert_{r,2,\Sigma}.
\end{equation*}
\end{Lemma}
\begin{Lemma}\label{te1}
{Suppose $U$ satisfy the hyperbolic system}
\begin{equation}\label{505}
  U_t+ \sum^{3}_{i=1}A_i(U)U_{x_i}= F.
\end{equation}
Then
\begin{equation}\label{te10}
\begin{split}
 \vert\kern-0.25ex\vert\kern-0.25ex\vert  U\vert\kern-0.25ex\vert\kern-0.25ex\vert_{s_0,2,\Sigma} & \lesssim \|d U \|_{L^2_t L^{\infty}_x}+ \| U\|_{L^{\infty}_tH_x^{s_0}}+\| F\|_{L^{2}_tH_x^{s_0}}.
   \end{split}
\end{equation}
\end{Lemma}

\begin{proof}
Choosing the change of coordinates $x_3 \rightarrow x_3-\phi(t,x')$ and setting $\tilde{U}(t,x)=U(t,x',x_3+\phi(t,x'))$, the system \eqref{te10} is transformed to
\begin{equation*}
  \partial_t \tilde{U}+ \sum_{i=1}^3 A_i(\tilde{U}) \partial_{x_i} \tilde{U}= - \partial_t \phi  \partial_3 \tilde{U} - \sum_{i=1}^3 A_i(\tilde{U}) \partial_{x_i}\phi \partial_i \tilde{U}+\tilde{F}.
\end{equation*}
For $\phi$ is independent of $x_3$, we then update it by
\begin{equation}\label{U}
  \partial_t \tilde{U}+ \sum_{i=1}^3 A_i(\tilde{U}) \partial_{x_i} \tilde{U}= - \partial_t \phi  \partial_3 \tilde{U} - \sum_{i=1}^2 A_i(\tilde{U}) \partial_{x_i}\phi \partial_i \tilde{U}+\tilde{F}.
\end{equation}
To prove \eqref{te10}, we first establish the $0$-order estimate. A direct calculation on$[-2,2]\times \mathbb{R}^3$ shows that
\begin{equation*}
\begin{split}
   \vert\kern-0.25ex\vert\kern-0.25ex\vert \tilde{U}\vert\kern-0.25ex\vert\kern-0.25ex\vert^2_{0,2,\Sigma} & \lesssim \| d\tilde{U} \|_{L^1_t L^\infty}\|\tilde{U}\|_{L^2} + \|\partial d\phi\|_{L^1_t L^\infty}\|\tilde{U}\|_{L^2}+\|\tilde{U}\|_{L^2}\|\tilde{F}\|_{L^1_tL^2}
   \\
   & \lesssim \| d\tilde{U} \|_{L^2_t L^\infty}\|\tilde{U}\|_{L^2} + \|\partial d\phi\|_{L^2_t L^\infty}\|\tilde{U}\|_{L^2}+\|\tilde{U}\|_{L^2}\|\tilde{F}\|_{L^2_tL^2}.
\end{split}
\end{equation*}
By using Lemma \ref{te0}, \eqref{5021} and \eqref{504}, we can prove that
\begin{equation}\label{U0}
 \vert\kern-0.25ex\vert\kern-0.25ex\vert U\vert\kern-0.25ex\vert\kern-0.25ex\vert_{0,2,\Sigma} \lesssim \|d U \|_{L^2_t L^{\infty}_x}+ \| U\|_{L^{\infty}_tL^2}+\|F\|_{L^2_tL^2}.
\end{equation}
We now establish the $s_0$-order estimate. Taking the derivative of $\partial^{\beta}_{x'}$($1 \leq |\beta| \leq s_0$) on \eqref{U} and integrating it on $[-2,2]\times \mathbb{R}^3$, we get
\begin{equation}\label{U1}
  \begin{split}
 \| \partial^{\beta}_{x'} \tilde{U}\|^2_{L^2_{\Sigma}} & \lesssim \| d \tilde{U} \|_{L^1_t L^{\infty}_x} \| \partial^{\beta}_{x} \tilde{U}\|_{L^{\infty}_tL_x^2}+\| \partial^{\beta}_{x} \tilde{U}\|_{L^{\infty}_tL_x^2}\| \partial^{\beta}_{x} \tilde{F}\|_{L^{1}_tL_x^2} +I_1+I_2,
  \end{split}
\end{equation}
where
\begin{equation*}
\begin{split}
  &I_1= -\int_{-2}^2\int_{\mathbb{R}^3} \partial^{\beta}_{x'}  \big( \partial_t \phi  \partial_3 \tilde{U} \big) \cdot \Lambda^{\beta}_{x'} \tilde{U}  dxd\tau,
\\
& I_2= -\sum^2_{i=1}\int_{-2}^2\int_{\mathbb{R}^3} \partial^{\beta}_{x'} \big( A_i(\tilde{U}) \partial_{x_i}\phi \partial_i \tilde{U} \big) \cdot \partial^{\beta}_{x'} \tilde{U}   dxd\tau.
\end{split}
\end{equation*}
We can write $I_1$ as
\begin{equation*}
\begin{split}
  I_1 =& -\int_{-2}^2\int_{\mathbb{R}^3} \big( \partial^{\beta}_{x'}(\partial_t \phi  \partial_3 \tilde{U})-\partial_t \phi \partial_3 \partial^{\beta}_{x'} \tilde{U} \big)   \partial^{\beta}_{x'} \tilde{U}dxd\tau
  \\
  & + \int_{-2}^2\int_{\mathbb{R}^3} \partial_t \phi \cdot \partial_3 \partial^{\beta}_{x'} \tilde{U} \cdot \partial^{\beta}_{x'} \tilde{U}  dx d\tau,
 \\
 &= -\int_{-2}^2\int_{\mathbb{R}^3} [ \partial^{\beta}_{x'}, \partial_t \phi  \partial_3] \tilde{U} \cdot \partial^{\beta}_{x'} \tilde{U}dxd\tau
\end{split}
\end{equation*}
We also write
\begin{equation*}
\begin{split}
  I_2 = & -\int_{-2}^2\int_{\mathbb{R}^3} \big( \partial^{\beta}_{x'} \big( A_i(\tilde{U}) \partial_{x_i}\phi \partial_i \tilde{U}) -  A_i(\tilde{U}) \partial_{x_i}\phi \partial_i \partial^{\beta}_{x'} \tilde{U} \big) \cdot \partial^{\beta}_{x'} \tilde{U}  dxd\tau
  \\
  & +\sum^2_{i=1} \int_{-2}^2\int_{\mathbb{R}^3} \big( A_i(\tilde{U}) \partial_{x_i}\phi \big) \cdot \partial_i(\partial^{\beta}_{x'} \tilde{U}) \cdot \partial^{\beta}_{x'} \tilde{U}dxd\tau.
\end{split}
\end{equation*}
By commutator estimates in Lemma \ref{jh}, we can get
\begin{equation}\label{U2}
\begin{split}
  |I_1|
  & \lesssim  \big( \|\partial^\beta \tilde U\|_{L^{\infty}_tL^2} \| d \partial_t \phi \|_{L^1_tL_{x}^\infty}+ \sup_{\theta, r} \|\partial^{\beta}_{x'} \partial_t \phi\|_{L^2(\Sigma_{\theta,r})} \| d \tilde{U}\|_{L^1_tL^\infty} \big)\cdot\|\partial^\beta \tilde U\|_{L^\infty_tL^2}
\end{split}
\end{equation}
and
\begin{equation}\label{U3}
\begin{split}
  |I_2| \lesssim  & \big(\| \partial^{\beta}_{x'} \tilde{U} \|_{L^2_tL^2} \|d \partial \phi\|_{L^1_t L^\infty_x} + \|d\tilde{U}\|_{L^1_tL^\infty} \sup_{\theta,r}\|\partial^{\beta}_{x'}d \phi\|_{L^2(\Sigma_{\theta,r})}  \big) \cdot \|\partial^{\beta}_{x'} \tilde{U} \|_{L^\infty_tL^2}
  \\
  & + \big( \| d \tilde{U} \|_{L^2_t L_x^\infty} \| \partial \phi\| _{L^2_tL_x^\infty}+ \|\tilde{U}\|_{L^2_t L^\infty} \|\partial^2\phi\|_{L^2_tL_{x}^\infty} \big)  \cdot \|\partial^\beta \tilde{U} \|^2_{L^\infty_t L^2}.
\end{split}
\end{equation}
Taking sum of $1\leq \beta \leq s_0$ on \eqref{U1}, due to Lemma \ref{te0}, \eqref{U2}, \eqref{U3}, \eqref{5021}, and \eqref{504}, we obtain
\begin{equation}\label{U4}
\begin{split}
 \vert\kern-0.25ex\vert\kern-0.25ex\vert \partial_{x'} U\vert\kern-0.25ex\vert\kern-0.25ex\vert_{s_0-1,2,\Sigma} & \lesssim \|d U \|_{L^2_t L^{\infty}_x}+\|d U\|_{L^{\infty}_tH_x^{s_0-1}}+\|d F\|_{L^{2}_tH_x^{s_0-1}}.
   \end{split}
\end{equation}
Taking derivatives on \eqref{505}, we have
\begin{equation*}
  \begin{split}
  (\partial U)_t + \sum^3_{i=1} A_i(U)(\partial U)_{x_i}=-\sum^3_{i=1}\partial(A_i(U))U_{x_i}+\partial F,
  \end{split}
\end{equation*}
In a similar process, we can obtain
\begin{equation}\label{U40}
\begin{split}
 \vert\kern-0.25ex\vert\kern-0.25ex\vert \partial U\vert\kern-0.25ex\vert\kern-0.25ex\vert_{s_0-1,2,\Sigma} & \lesssim \|d U \|_{L^2_t L^{\infty}_x}+\|d U\|_{L^{\infty}_tH_x^{s_0-1}}+\|d F\|_{L^{2}_tH_x^{s_0-1}}.
   \end{split}
\end{equation}
Using $\partial_t U=- \sum^3_{i=1}A_i(U)\partial_iU$ and Lemma \ref{te2}, we can carry out
\begin{equation}\label{U5}
\begin{split}
\vert\kern-0.25ex\vert\kern-0.25ex\vert \partial_t U\vert\kern-0.25ex\vert\kern-0.25ex\vert_{s_0-1,2,\Sigma}
& \lesssim \vert\kern-0.25ex\vert\kern-0.25ex\vert U\vert\kern-0.25ex\vert\kern-0.25ex\vert_{s_0-1,2,\Sigma} \vert\kern-0.25ex\vert\kern-0.25ex\vert\partial_t U\vert\kern-0.25ex\vert\kern-0.25ex\vert_{s_0-1,2,\Sigma}
\\
 & \lesssim \|d U \|_{L^2_t L^{\infty}_x}+ \|d U\|_{L^{\infty}_tH_x^{s_0-1}}+\|d F\|_{L^{2}_tH_x^{s_0-1}}.
\end{split}
\end{equation}
Combining \eqref{U0}, \eqref{U4}, \eqref{U5}, and \eqref{U5}, we obtain \eqref{te10}. Thus, the proof is finished.
\end{proof}
Using Lemma \ref{te1}, and combining with \eqref{402} and \eqref{403}, we can obtain the following corollary.
\begin{corollary}\label{vte}
Suppose that $(v, \boldsymbol{\rho}, \varpi) \in \mathcal{H}$. Then the following estimate
\begin{equation}
\vert\kern-0.25ex\vert\kern-0.25ex\vert v\vert\kern-0.25ex\vert\kern-0.25ex\vert_{s_0,2,\Sigma}+ \vert\kern-0.25ex\vert\kern-0.25ex\vert\boldsymbol{\rho}\vert\kern-0.25ex\vert\kern-0.25ex\vert_{s_0,2,\Sigma} \lesssim \epsilon_2
\end{equation}
holds.
\end{corollary}
The next goal is to establish the characteristic energy estimates for $\varpi$. Compared with the energy estimates for $\varpi$, the characteristic energy is along the hypersurface, not the Cauchy slices. So, it's not trivial. To prove the energy estimates of $\varpi$ along the null hypersurface, Let us first give a lemma.
\begin{Lemma}\label{te3}
Let $f$ satisfy the following transport equation
\begin{equation}\label{333}
  \mathbf{T} f=F.
\end{equation}
Set $L= \partial (-\Delta)^{-1}\mathrm{curl}$. Then
\begin{equation}\label{teE}
\begin{split}
  \vert\kern-0.25ex\vert\kern-0.25ex\vert  Lf\vert\kern-0.25ex\vert\kern-0.25ex\vert^2_{s_0-2,2,\Sigma} \lesssim & \ \big( \| \partial v\|_{L^2_t\dot{B}^{s_0-2}_{\infty,2}}+\vert\kern-0.25ex\vert\kern-0.25ex\vert d\phi-dt \vert\kern-0.25ex\vert\kern-0.25ex\vert_{s_0,2,\Sigma} \big) \|f\|^2_{{H}^{s_0-2}_x}(1+\vert\kern-0.25ex\vert\kern-0.25ex\vert d\phi-dt \vert\kern-0.25ex\vert\kern-0.25ex\vert_{s_0,2,\Sigma})
  \\
  & + \sum_{\sigma\in\{0,s_0-2 \}}\big|\int^t_0 \int_{\mathbb{R}^3} \Lambda^{\sigma}_{x'}{LF} \cdot \Lambda^{\sigma}_{x'}Lfdxd\tau \big|.
\end{split}
\end{equation}
\begin{proof}
Taking the operator $L$ on \eqref{333}, we derive that
\begin{equation*}
  \mathbf{T} Lf=LF+[L,\mathbf{T}]f.
\end{equation*}
Choosing the change of coordinates $x_3 \rightarrow x_3-\phi(t,x')$ and setting $\widetilde{f}=f(x_1,x_2,x_3-\phi(t,x'))$, then the above equation transforms to
\begin{equation*}
\begin{split}
  (\partial_t+ \partial_t \phi \partial_{x_3}) \widetilde{Lf}+ \tilde{v}^i \cdot (\partial_{x_i}+\partial_{x_i} \phi \partial_{x_3} ) \widetilde{Lf}= & \widetilde{LF}+\widetilde{[L,\mathbf{T}]f}.
  \end{split}
\end{equation*}
Rewrite it as
\begin{equation}\label{Q}
\begin{split}
  \partial_t\widetilde{Lf}+ \tilde{v}^i \cdot \partial_{x_i}\widetilde{Lf}= & \widetilde{LF}+\widetilde{[L,\mathbf{T}]f}
- \partial_t \phi \partial_{x_3}\widetilde{Lf}-\tilde{v}^i \cdot \partial_{x_i} \phi \partial_{x_3} \widetilde{Lf} .
  \end{split}
\end{equation}
Multiplying $\widetilde{Lf}$ and integrating it on $\mathbb{R}^{+} \times \mathbb{R}^3$, we can show that
\begin{equation}\label{LL}
  \begin{split}
  \vert\kern-0.25ex\vert\kern-0.25ex\vert Lf\vert\kern-0.25ex\vert\kern-0.25ex\vert^2_{0,2,\Sigma} \lesssim \ &\big| \int^2_{-2} \int_{\mathbb{R}^3} LF \cdot Lfdx d\tau \big|+ \| \partial v\|_{L^2_t L^\infty_x}(1+ \|\partial \phi\|_{L^\infty_x})\|Lf\|^2_{L^2_x}
  \\
  & \ +\|[L,\mathbf{T}]f\|_{L^2_t L^2_x}\|Lf\|^2_{L^2_x}
  \end{split}
\end{equation}
Considering $L= \partial (-\Delta)^{-1}\mathrm{curl}$, by commutator estimates in Lemma \ref{ceR}, we get
\begin{equation}\label{LL0}
\begin{split}
  \|[L,\mathbf{T}]f\|_{L^2_x} \lesssim & \ \|\partial v\|_{L^\infty}\|f \|_{L^2_x}.
\end{split}
\end{equation}
By elliptic estimates, we also have
\begin{equation}\label{LL1}
  \|Lf\|^2_{L^2_x} \lesssim \|f\|^2_{L^2_x}.
\end{equation}
By using \eqref{LL}, \eqref{LL0}, and \eqref{LL1}, it can give us
\begin{equation}\label{te30}
  \vert\kern-0.25ex\vert\kern-0.25ex\vert Lf\vert\kern-0.25ex\vert\kern-0.25ex\vert^2_{0,2,\Sigma} \lesssim  \ \| \partial v\|_{L^2_t L^\infty_x}(1+ \|\partial \phi\|_{L^2_tL^\infty_x})\|f\|^2_{L^2_x} + \big| \int^2_{-2} \int_{\mathbb{R}^3} LF \cdot Lfdx d\tau \big|.
\end{equation}
It remains for us to estimate the high order term. Taking derivatives $\Lambda^{s_0-2}_{x'}$ on \eqref{Q}, we have
\begin{equation}\label{Q0}
\begin{split}
  \partial_t \Lambda^{s_0-2}_{x'}\widetilde{Lf}+ \tilde{v}^i \cdot \partial_{x_i}\Lambda^{s_0-2}_{x'}\widetilde{Lf}=
  & \Lambda^{s_0-2}_{x'}\left( \widetilde{[L,\mathbf{T}]f} \right)-\Lambda^{s_0-2}_{x'}( \partial_t \phi \partial_{x_3}\widetilde{Lf})
  \\
  & +\Lambda^{s_0-2}_{x'}\widetilde{LF} -\Lambda^{s_0-2}_{x'}(\tilde{v}^i \partial_{x_i} \phi \partial_{x_3} \widetilde{Lf})
  \\
  & -[\Lambda^{s_0-2}_{x'}, \tilde{v}^i \partial_{x_i}]\widetilde{Lf} .
  \end{split}
\end{equation}
Multiplying $\Lambda^{s_0-2}_{x'}\widetilde{Lf}$ on \eqref{Q0} and integrating it on $[-2,2]\times \mathbb{R}^3$, we derive that
\begin{equation*}
  \begin{split}
  \|\Lambda^{s_0-2}_{x'}\widetilde{Lf}\|_{L^2_{\Sigma}} \lesssim & \ \|dv\|_{L^1_t L^\infty_x}\|\Lambda^{s_0-2}_{x'}Lf\|^2_{L^{2}_x}+
   \|\Lambda^{s_0-2}_{x'} ( {[L,\mathbf{T}]f} )\|_{L^1_t L^2_x}\|\Lambda^{s_0-2}_{x'}Lf\|_{L^{2}_x}
  \\
  & \ +\big|\int^t_0 \Lambda^{s_0-2}_{x'}{LF} \cdot \Lambda^{s_0-2}_{x'}Lfdxd\tau \big|
  + \| [\Lambda^{s_0-2}_{x'}, \tilde{v}^i \partial_{x_i}]{Lf}\|_{L^1_t L^2_x}\|\Lambda^{s_0-2}_{x'}Lf\|_{L^{2}_x}
  \\
  & \ +\big| \int^2_{-2}\int_{\mathbb{R}^3} \big( \Lambda^{s_0-2}_{x'}( \partial_t \phi \partial_{x_3}{Lf})-\Lambda^{s_0-2}_{x'}(\tilde{v}^i \partial_{x_i} \phi \partial_{x_3} {Lf}) \big)\Lambda^{s_0-2}_{x'}{Lf}dxd\tau \big|.
  \end{split}
\end{equation*}
We will estimate the right terms one by one. By using elliptic estimates, we can prove
\begin{equation}\label{Q1}
  \|dv\|_{L^1_t L^\infty_x}\|\Lambda^{s_0-2}_{x'}Lf\|^2_{L^{2}_x} \leq \|dv\|_{L^2_t L^2_x}\|f\|^2_{\dot{H}^{s_0-2}_x},
\end{equation}
and
\begin{equation}\label{Q40}
\begin{split}
   \|\Lambda^{s_0-2}_{x'} ( {[L,\mathbf{T}]f} )\|_{L^1_t L^2_x} \|\Lambda^{s_0-2}_{x'}Lf\|_{L^{2}_x}
  \lesssim  \ \|[L,\mathbf{T}]f\|_{L^1_t \dot{H}^{s_0-2}_x} \|f\|_{\dot{H}^{s_0-2}_x}.
\end{split}
\end{equation}
For the right term $\|[L,\mathbf{T}]f\|_{L^1_t \dot{H}^{s_0-2}_x}$ in \eqref{Q40}, by using Lemma \ref{ceR}, we have
\begin{equation*}
\begin{split}
  \|[L,\mathbf{T}]f\|_{L^1_t \dot{H}^{s_0-2}_x} & \lesssim \|v\|_{L^1_t\dot{B}^1_{\infty,\infty}} \cdot  \| f \|_{\dot{H}^{s_0-2}_x}+ \| v\|_{L^1_t\dot{B}^{s_0-1}_{\infty,\infty}} \cdot  \| f \|_{L^2_x}
  \\
 & \lesssim \| \partial v\|_{L^2_t L^{\infty}_x}\|f \|_{\dot{H}^{s_0-2}_x}+ \| \partial v\|_{L^2_t \dot{B}^{s_0-2}_{\infty,2}} \|f \|_{L^2_x}
 \\
 & \lesssim \big( \| \partial v\|_{L^2_t\dot{B}^{s_0-2}_{\infty,2}}+\| \partial v\|_{L^2_tL^{\infty}_x} \big) \|f\|_{{H}^{s_0-2}_x}.
\end{split}
\end{equation*}
Substituting it to \eqref{Q40}, we can obtain
\begin{equation}\label{Q4}
\begin{split}
   \|\Lambda^{s_0-2}_{x'} ( {[L,\mathbf{T}]f} )\|_{L^1_t L^2_x} \|\Lambda^{s_0-2}_{x'}Lf\|_{L^{2}_x}
  \lesssim  \big( \| \partial v\|_{L^2_t\dot{B}^{s_0-2}_{\infty,2}}+\| \partial v\|_{L^2_tL^{\infty}_x} \big) \|f\|^2_{{H}^{s_0-2}_x}.
\end{split}
\end{equation}
Using Lemma \ref{ce}, we also have
\begin{equation}\label{Q5}
\begin{split}
 \| [\Lambda^{s_0-2}_{x'}, \tilde{v}^i \partial_{x_i}]{Lf}\|_{L^1_t L^2_x}\|\Lambda^{s_0-2}_{x'}Lf\|_{L^{2}_x}
\leq & \ \| [\Lambda^{s_0-2}_{x}, \tilde{v}^i \partial_{x_i}]{Lf}\|_{L^1_t L^2_x}\|\Lambda^{s_0-2}_{x}Lf\|_{L^{2}_x}
\\
\lesssim & \ \| [\Lambda^{s_0-2}_{x}, \tilde{v}^i \partial_{x_i}]{Lf}\|_{L^1_t L^2_x}\|f\|_{\dot{H}^{s_0-2}_x}
\\
\lesssim & \ \| \partial v\|_{L^1_t \dot{B}^0_{\infty,2}} \|Lf\|_{\dot{H}^{s_0-2}_x}\|f\|_{\dot{H}^{s_0-2}_x}
\\
\lesssim & \ \| \partial v\|_{L^2_t \dot{B}^0_{\infty,2}} \|f\|^2_{\dot{H}^{s_0-2}_x}.
\end{split}
\end{equation}
For $\phi$ is independent with $x_3$, we have
\begin{equation}\label{jhz}
\begin{split}
  & \int^2_{-2}\int_{\mathbb{R}^3} \left( \Lambda^{s_0-2}_{x'}( \partial_t \phi \partial_{x_3}{Lf})-\Lambda^{s_0-2}_{x'}(\tilde{v}^i \partial_{x_i} \phi \partial_{x_3} {Lf}) \right)\Lambda^{s_0-2}_{x'}{Lf}dxd\tau
  \\
  = \ & \int^2_{-2}\int_{\mathbb{R}^3} \left( [\Lambda^{s_0-2}_{x'}, \partial_t \phi \partial_{x_3}]{Lf})-[\Lambda^{s_0-2}_{x'}, \tilde{v}^i \partial_{x_i} \phi \partial_{x_3}]{Lf} \right)\Lambda^{s_0-2}_{x'}{Lf}dxd\tau
  \\
  & + \int^2_{-2}\int_{\mathbb{R}^3} \left( \partial_t \phi \partial_{x_3}\Lambda^{s_0-2}_{x'}{Lf})-\tilde{v}^i \partial_{x_i} \phi \partial_{x_3}\Lambda^{s_0-2}_{x'}{Lf} \right)\Lambda^{s_0-2}_{x'}{Lf}dxd\tau
  \\
  = \ & \int^2_{-2}\int_{\mathbb{R}^3} \left( [\Lambda^{s_0-2}_{x'}, \partial_t \phi \partial_{x_3}]{Lf})-[\Lambda^{s_0-2}_{x'}, \tilde{v}^i \partial_{x_i} \phi \partial_{x_3}]{Lf} \right)\Lambda^{s_0-2}_{x'}{Lf}dxd\tau
  \\
  & + \int^2_{-2}\int_{\mathbb{R}^3} \partial_{x_3}\tilde{v}^i \partial_{x_i} \phi \Lambda^{s_0-2}_{x'}{Lf} \Lambda^{s_0-2}_{x'}{Lf}dxd\tau.
\end{split}
\end{equation}
We will estimate the right terms on \eqref{jhz}. For the first one and the second one, we use Lemma \ref{ce} and Lemma \ref{LPE} to bound
\begin{equation}\label{r1E}
\begin{split}
  & \big( \|[\Lambda^{s_0-2}_{x'}, \partial_t \phi \partial_{x_3}]{Lf}\|_{L^1_tL^2_x}+\|[\Lambda^{s_0-2}_{x'}, \tilde{v}^i \partial_{x_i} \phi \partial_{x_3}]{Lf}\|_{L^1_tL^2_x} \big) \|Lf\|_{\dot{H}^{s_0-2}_x}
  \\
  \lesssim & \big( \|\partial(\partial_t \phi)\|_{L^1_t \dot{H}^{s_0-2}_x} + \partial(v \partial \phi)\|_{L^1_t \dot{H}^{s_0-2}_x} \big)\|Lf\|^2_{\dot{H}^{s_0-2}_x}
  \\
  \lesssim & \big(\| \partial(d \phi)\|_{L^2_t \dot{B}^{s_0-2}_{\infty,2}} + \|\partial v\|_{L^2_t \dot{B}^{s_0-2}_x} \|\partial \phi\|_{L^2_t L^\infty_x} + \|\partial \phi\|_{L^2_t C^{\beta}_x}\|\partial v\|_{L^2_t L^\infty_x} \big)\|Lf\|^2_{\dot{H}^{s_0-2}_x},
\end{split}
\end{equation}
where we take $\beta$=$s_0-\frac32-\epsilon_0>s_0-2$. By Sobolev imbedding, we can get
\begin{equation*}
  \| \partial(d \phi)\|_{L^2_t \dot{B}^{s_0-2}_{\infty,2}} \lesssim \| \partial(d \phi)\|_{L^2_t C^{\beta}} \lesssim  \| \partial(d \phi-dt)\|_{L^2_t H_{x'}^{s_0-1}(\Sigma)} \lesssim \vert\kern-0.25ex\vert\kern-0.25ex\vert d\phi-dt \vert\kern-0.25ex\vert\kern-0.25ex\vert_{s_0,2,\Sigma},
\end{equation*}
and
\begin{equation*}
  \|\partial \phi\|_{L^2_t L^\infty_x} + \|\partial \phi\|_{L^2_t C^{\beta}_x} \leq 1+\vert\kern-0.25ex\vert\kern-0.25ex\vert d \phi-dt\vert\kern-0.25ex\vert\kern-0.25ex\vert_{L^2_t C^{\beta}_x}\lesssim 1+ \vert\kern-0.25ex\vert\kern-0.25ex\vert d\phi-dt\vert\kern-0.25ex\vert\kern-0.25ex\vert_{s_0,2,\Sigma}.
\end{equation*}
Substituting them to \eqref{r1E}, we can update \eqref{r1E} as
\begin{equation}\label{r2E}
\begin{split}
  & \big( \|[\Lambda^{s_0-2}_{x'}, \partial_t \phi \partial_{x_3}]{Lf}\|_{L^1_tL^2_x}+\|[\Lambda^{s_0-2}_{x'}, \tilde{v}^i \partial_{x_i} \phi \partial_{x_3}]{Lf}\|_{L^1_tL^2_x} \big) \|Lf\|_{\dot{H}^{s_0-2}_x}
  \\
  \lesssim & \|f\|^2_{\dot{H}^{s_0-2}_x} \big( \vert\kern-0.25ex\vert\kern-0.25ex\vert d\phi-dt\vert\kern-0.25ex\vert\kern-0.25ex\vert_{s_0,2,\Sigma}+ (\|\partial v\|_{L^2_t \dot{B}^{s_0-2}_x}+\|\partial v\|_{L^2_t L^\infty_x})(1+\vert\kern-0.25ex\vert\kern-0.25ex\vert d\phi-dt\vert\kern-0.25ex\vert\kern-0.25ex\vert_{s_0,2,\Sigma}) \big).
\end{split}
\end{equation}
For the third term on the right hand of \eqref{jhz}, we can bound it by
\begin{equation}\label{r3E}
\begin{split}
  & \left|\int^2_{-2}\int_{\mathbb{R}^3} \partial_{x_3}\tilde{v}^i \partial_{x_i} \phi \Lambda^{s_0-2}_{x'}{Lf} \Lambda^{s_0-2}_{x'}{Lf}dxd\tau \right|
  \\
  \lesssim & \| \partial v\|_{L^2_t L^\infty_x} \|\partial \phi\|_{L^2_t L^\infty_x} \|Lf\|^2_{\dot{H}^{s_0-2}_x}
  \\
  \lesssim & (1+ \vert\kern-0.25ex\vert\kern-0.25ex\vert d\phi-dt\vert\kern-0.25ex\vert\kern-0.25ex\vert_{s_0,2,\Sigma})\| \partial v\|_{L^2_t L^\infty_x}\|f\|^2_{H^{s_0-2}_x}.
\end{split}
\end{equation}
Sustituting \eqref{r2E} and \eqref{r3E}, we can estimate \eqref{jhz} by
\begin{equation}\label{Q6}
 \begin{split}
 & \left|\int^2_{-2}\int_{\mathbb{R}^3} \big( \Lambda^{s_0-2}_{x'}( \partial_t \phi \partial_{x_3}{Lf})-\Lambda^{s_0-2}_{x'}(\tilde{v}^i \partial_{x_i} \phi \partial_{x_3} {Lf}) \big)\Lambda^{s_0-2}_{x'}{Lf}dxd\tau \right|
\\
 \lesssim & \|f\|^2_{\dot{H}^{s_0-2}_x} \big( \vert\kern-0.25ex\vert\kern-0.25ex\vert d\phi-dt\vert\kern-0.25ex\vert\kern-0.25ex\vert_{s_0,2,\Sigma}+ (\|\partial v\|_{L^2_t \dot{B}^{s_0-2}_x}+\|\partial v\|_{L^2_t L^\infty_x})(1+\vert\kern-0.25ex\vert\kern-0.25ex\vert d\phi-dt\vert\kern-0.25ex\vert\kern-0.25ex\vert_{s_0,2,\Sigma}) \big)
 \\
 \lesssim & \|f\|^2_{\dot{H}^{s_0-2}_x} \big( \vert\kern-0.25ex\vert\kern-0.25ex\vert d\phi-dt\vert\kern-0.25ex\vert\kern-0.25ex\vert_{s_0,2,\Sigma}+ \|\partial v\|_{L^2_t \dot{B}^{s_0-2}_x})(1+\vert\kern-0.25ex\vert\kern-0.25ex\vert d\phi-dt\vert\kern-0.25ex\vert\kern-0.25ex\vert_{s_0,2,\Sigma}) \big),
  \end{split}
\end{equation}
where we use the fact that $\|\partial v\|_{L^2_t L^\infty} \leq \|\partial v\|_{L^2_t \dot{B}^{s_0-2}_x}$.
Combining \eqref{Q1} to \eqref{Q6}, we can get
\begin{equation}\label{te31}
\begin{split}
  \vert\kern-0.25ex\vert\kern-0.25ex\vert \Lambda^{s_0-2}_{x'} Lf\vert\kern-0.25ex\vert\kern-0.25ex\vert^2_{0,2,\Sigma} \lesssim & \ \big( \| \partial v\|_{L^2_t\dot{B}^{s_0-2}_{\infty,2}}+\vert\kern-0.25ex\vert\kern-0.25ex\vert d\phi-dt \vert\kern-0.25ex\vert\kern-0.25ex\vert_{s_0,2,\Sigma} \big) \|f\|^2_{{H}^{s_0-2}_x}(1+\vert\kern-0.25ex\vert\kern-0.25ex\vert d\phi-dt \vert\kern-0.25ex\vert\kern-0.25ex\vert_{s_0,2,\Sigma})
  \\
  & + \left|\int^t_0 \int_{\mathbb{R}^3} \Lambda^{s_0-2}_{x'}{LF} \cdot \Lambda^{s_0-2}_{x'}Lfdxd\tau \right|.
\end{split}
\end{equation}
Adding \eqref{te30} and \eqref{te31}, we can derive
\begin{equation*}
\begin{split}
  \vert\kern-0.25ex\vert\kern-0.25ex\vert  Lf\vert\kern-0.25ex\vert\kern-0.25ex\vert^2_{s_0-2,2,\Sigma} \lesssim & \ \big( \| \partial v\|_{L^2_t\dot{B}^{s_0-2}_{\infty,2}}+\vert\kern-0.25ex\vert\kern-0.25ex\vert d\phi-dt \vert\kern-0.25ex\vert\kern-0.25ex\vert_{s_0,2,\Sigma} \big) \|f\|^2_{{H}^{s_0-2}_x}(1+\vert\kern-0.25ex\vert\kern-0.25ex\vert d\phi-dt \vert\kern-0.25ex\vert\kern-0.25ex\vert_{s_0,2,\Sigma})
  \\
  & + \sum_{\sigma\in\{0,s_0-2 \}} \left|\int^t_0 \int_{\mathbb{R}^3} \Lambda^{\sigma}_{x'}{LF} \cdot \Lambda^{\sigma}_{x'}Lfdxd\tau \right|.
\end{split}
\end{equation*}
Therefore, we complete the proof of Lemma \ref{te3}.
\end{proof}
\end{Lemma}
We also need to give the estimate of $\mathrm{curl} \varpi$ along the characteristic hypersurfaces, for $\mathrm{curl} \varpi$ is a nonlinear term in the wave equation of the velocity. Then, $\mathrm{curl} \varpi$ decides the regularity of the characteristic hypersurfaces.
\begin{Lemma}\label{te20}
Suppose that $(v, \boldsymbol{\rho}, \varpi) \in \mathcal{H}$. 
Then
\begin{equation}\label{te201}
\begin{split}
 \vert\kern-0.25ex\vert\kern-0.25ex\vert \mathrm{curl}\varpi\vert\kern-0.25ex\vert\kern-0.25ex\vert_{s_0-1,2,\Sigma} \lesssim  \epsilon_2.
   \end{split}
\end{equation}
\end{Lemma}
\begin{proof}
Consider only $\mathrm{curl} \Omega$ satisfying a good transport equation. So we should transfer the goal on $\Omega$. By product estimates on $H^s(\Sigma^t)$ (see Lemma \ref{cj}) and trace theorem, we have
\begin{equation}\label{C1}
\begin{split}
  \vert\kern-0.25ex\vert\kern-0.25ex\vert \mathrm{curl}\varpi\vert\kern-0.25ex\vert\kern-0.25ex\vert_{s_0-1,2,\Sigma} =&\vert\kern-0.25ex\vert\kern-0.25ex\vert \mathrm{e}^{\boldsymbol{\rho}}\Omega\vert\kern-0.25ex\vert\kern-0.25ex\vert_{s_0-1,2,\Sigma}
  \\
  = & \| \mathrm{e}^{\boldsymbol{\rho}}\Omega\|_{L^2_t H^{s_0-1}_{x'}}+ \| \partial_t( \mathrm{e}^{\boldsymbol{\rho}}\Omega)\|_{L^2_t H^{s_0-2}_{x'}}
  \\
  \lesssim & \| \mathrm{e}^{\boldsymbol{\rho}} \|_{L^\infty_t L^\infty_{x'}} \|\Omega\|_{L^2_t H^{s_0-1}_{x'}}+\| \Lambda^{s_0-1}_{x'}(\mathrm{e}^{\boldsymbol{\rho}})\|_{L^\infty_t L^{2}_{x'}}\|\Omega\|_{L^2_t L^{\infty}_{x'}}
  \\
  & + \| \partial_t {\boldsymbol{\rho}}\|_{L^\infty_t H^{s_0-2}_{x'}}\|\Omega\|_{L^2_t L^{\infty}_{x'}}+ \| \mathrm{e}^{\boldsymbol{\rho}}\|_{ L^\infty_t L^{\infty}_{x'} } \|\partial_t \Omega\|_{L^2_t H^{s_0-2}_{x'}}
  \\
  \lesssim &(1+\vert\kern-0.25ex\vert\kern-0.25ex\vert{\boldsymbol{\rho}}\vert\kern-0.25ex\vert\kern-0.25ex\vert_{s_0,2,\Sigma}) \vert\kern-0.25ex\vert\kern-0.25ex\vert\Omega\vert\kern-0.25ex\vert\kern-0.25ex\vert_{s_0-1,2,\Sigma}.
\end{split}
\end{equation}
By Definition \eqref{d0}, we can see
\begin{equation}\label{OM}
 \vert\kern-0.25ex\vert\kern-0.25ex\vert \Omega\vert\kern-0.25ex\vert\kern-0.25ex\vert_{s_0-1,2,\Sigma}=\| \Omega\|_{L^2_t H^{s_0-1}_{x'}(\Sigma)}+ \| \partial_t \Omega\|_{L^2_t H^{s_0-2}_{x'}(\Sigma)}.
\end{equation}
In the following, we will give the bound of the right terms on \eqref{OM}. We divide it into several steps.

\textbf{Step 1: $\| \Omega\|_{L^2_t H^{s_0-1}_{x'}(\Sigma)}$}. Note \eqref{W1}.
We can simply write it as
\begin{equation}\label{TOE}
\mathbf{T}  \Omega = \partial v \cdot \partial \varpi.
\end{equation}
By changing of coordinates $x_3 \rightarrow x_3-\phi(t,x')$, multiplying $\Omega$ and integrating it on $[-2,2]\times \mathbb{R}^3$, we can obtain
\begin{equation*}
  \begin{split}
  \| \Omega\|^2_{L^2_t L^2_{x'}(\Sigma)} \lesssim & \ \|\partial v\|_{L^1_tL^\infty_x} ( \|\partial \varpi\|_{L^\infty_t L^2_x}+\|\Omega \|_{L^\infty_t L^2_x}) \|\Omega\|_{L^\infty_t L^2_x}.
  \end{split}
\end{equation*}
Using $\Omega=\mathrm{e}^{- \boldsymbol{\rho}}\mathrm{curl}\varpi$, \eqref{402}, and \eqref{403}, we derive that
\begin{equation}\label{OM0}
  \| \Omega \|^2_{L^2_t L^2_{x'}(\Sigma)} \lesssim \epsilon^3_2.
\end{equation}
It remains for us to bound $\| \Lambda^{s_0-1}_{x'}\Omega\|_{L^2_t L^2_{x'}(\Sigma)}$. Let us first estimate $\|  \partial \Omega\|_{L^2_t H^{s_0-2}_{x'}(\Sigma)}$. By Hodge decomposition,
\begin{equation}\label{OM1}
\begin{split}
  \|  \partial \Omega\|_{L^2_t H^{s_0-2}_{x'}(\Sigma)} &= \| L ( \mathrm{curl} \Omega)+H ( \mathrm{div} \Omega)\|_{L^2_t H^{s_0-2}_{x'}(\Sigma)}
 \\
  & \leq \| L ( \mathrm{curl} \Omega)\|_{L^2_t H^{s_0-2}_{x'}(\Sigma)}+\| H ( \mathrm{div} \Omega)\|_{L^2_t H^{s_0-2}_{x'}(\Sigma)},
\end{split}
\end{equation}
where the operators $L$ and $H$ are given by $\partial(-\Delta)^{-1}\mathrm{curl}$ and $\partial(-\Delta)^{-1}\nabla$ respectively. By Sobolev imbedding $H^{\frac{1}{2}+}(\mathbb{R}) \hookrightarrow L^\infty(\mathbb{R})$ and $s-s_0+\frac12>\frac12$, we have
\begin{equation}\label{WR}
  \begin{split}
  \|\Lambda^{s_0-2}_{x'}H ( \mathrm{div} \Omega)\|_{L^2_t L^2_{x'}(\Sigma)}
   \lesssim & \|\Lambda^{s_0-2}_{x'}H ( \mathrm{div} \Omega)\|_{L^\infty_t H^{s-s_0+\frac12}_{x}}
   \\
    \lesssim & \|  \mathrm{div} \Omega\|_{L^\infty_t H^{s-\frac32}_{x}}
  \\
 \lesssim & \|\Omega\|_{L^\infty_t H^{1}_{x}}\|\partial \boldsymbol{\rho} \|_{L^\infty_t H^{s-1}_{x}}.
  \end{split}
\end{equation}
We then update \eqref{WR} by
\begin{equation}\label{DO0}
  \begin{split}
  \|\Lambda^{s_0-2}_{x'}H ( \mathrm{div} \Omega)\|_{L^2_t L^2_{x'}(\Sigma)}
  & \ \lesssim \|\Omega\|_{L^\infty_t H^{1}_{x}}\|\boldsymbol{\rho} \|_{L^\infty_t H^{s}_{x}} \lesssim \|\varpi\|_{L^\infty_t H^{2}_{x}}\|\boldsymbol{\rho} \|_{L^\infty_t H^{s}_{x}}\lesssim \epsilon_2^2.
  \end{split}
\end{equation}
In a similar way, we can deduce
\begin{equation}\label{DO1}
  \begin{split}
   \|H ( \Omega \cdot \partial \boldsymbol{\rho})\|_{L^2_t L^2_{x'}(\Sigma)} &= \|H ( \Omega \cdot \partial \boldsymbol{\rho})\|_{L^2_t H^{s_0-\frac32}_{x}}, \quad
  \\
  & \ \lesssim \| \Omega \cdot \partial \boldsymbol{\rho} \|_{L^\infty_t H^{s_0-\frac32}_{x}}
   \lesssim \| \Omega \|_{L^\infty_t H^1_{x}} \| \partial \boldsymbol{\rho} \|_{L^2_t H^{s_0-1}_{x}}
   \\
  & \ \lesssim \| \varpi \|_{L^\infty_t H^2_{x}} \| \boldsymbol{\rho} \|_{L^\infty_t H^{s_0}_{x}} \lesssim \epsilon_2^2.
  \end{split}
\end{equation}
Adding \eqref{DO0} and \eqref{DO1}, we can conclude that
\begin{equation}\label{DO}
  \vert\kern-0.25ex\vert\kern-0.25ex\vert H ( \mathrm{div} \Omega)\vert\kern-0.25ex\vert\kern-0.25ex\vert_{s_0-2,2,\Sigma} \lesssim \epsilon_2^2.
\end{equation}
It remains for us to estimate $\vert\kern-0.25ex\vert\kern-0.25ex\vert L ( \mathrm{curl} \Omega)\vert\kern-0.25ex\vert\kern-0.25ex\vert_{s_0-2,2,\Sigma}$. Recall the transport equation for $\mathrm{curl}\Omega$:
\begin{equation}\label{TO0}
\begin{split}
  &\mathbf{T} \big( \mathrm{curl} \Omega^i -2 \mathrm{e}^{-\boldsymbol{\rho}} \partial_a \boldsymbol{\rho}  \partial^i \varpi^a \big)
  \\
=& \ \partial^i \big( 2 \mathrm{e}^{-\boldsymbol{\rho}} \partial_n v^a \partial^n \varpi_a \big) + R^i_1+ R^i_2+ R^i_3+ R^i_4+ R^i_5+ R^i_6.
\end{split}
\end{equation}
We set
$$f^i=\mathrm{curl}\Omega^i-2\mathrm{e}^{-{\boldsymbol{\rho}}}\partial_a \boldsymbol{\rho} \partial^i \varpi^a,\quad i=1,2,3,$$ and $$F^i=  \sum^6_{k=1}R^i_k, \qquad  G=  2 \mathrm{e}^{-\boldsymbol{\rho}} \partial_n v^a \partial^n \varpi_a , \quad i=1,2,3.$$
Here $R_1, R_2, \ldots, R_6$ is defined in \eqref{rF}. Then we can write \eqref{TO0} as
\begin{equation}\label{TO}
\mathbf{T} f^i=F^i+ \partial^i G.
\end{equation}
Operating the operator $L$ on \eqref{TO}, and using Lemma \ref{te3}, we can deduce that
\begin{equation}\label{L0}
\begin{split}
  \vert\kern-0.25ex\vert\kern-0.25ex\vert Lf\vert\kern-0.25ex\vert\kern-0.25ex\vert^2_{s_0-2,2,\Sigma} \lesssim  \ & \ \big( \| \partial v\|_{L^2_t\dot{B}^{s_0-2}_{\infty,2}}+\vert\kern-0.25ex\vert\kern-0.25ex\vert d\phi-dt \vert\kern-0.25ex\vert\kern-0.25ex\vert_{s_0,2,\Sigma} \big) \|f\|^2_{{H}^{s_0-2}_x}(1+\vert\kern-0.25ex\vert\kern-0.25ex\vert d\phi-dt \vert\kern-0.25ex\vert\kern-0.25ex\vert_{s_0,2,\Sigma})
  \\
  & +  \sum_{\sigma \in \{0,s_0-2 \}}\left| \int^2_{-2} \int_{\mathbb{R}^3}  \Lambda_{x'}^{\sigma}\widetilde{LF^i} \cdot \Lambda_{x'}^{\sigma} \widetilde{Lf_i}dxd\tau    \right|
  \\
  & +  \sum_{\sigma \in \{0,s_0-2 \}}\left| \int^2_{-2} \int_{\mathbb{R}^3}  \Lambda_{x'}^{\sigma}\left\{(\partial^i+\partial^i \phi \partial^3)\widetilde{LG}\right\} \cdot \Lambda_{x'}^{\sigma} \widetilde{Lf_i}dxd\tau    \right|.
  \end{split}
\end{equation}
We set
\begin{equation*}
\begin{split}
\mathrm{I}&=\sum_{\sigma \in \{0,s_0-2 \}}\left| \int^2_{-2} \int_{\mathbb{R}^3}  \Lambda_{x'}^{\sigma}\widetilde{LF^i} \cdot \Lambda_{x'}^{\sigma} \widetilde{Lf_i}dxd\tau    \right|,
\\
\mathrm{J}&=   \sum_{\sigma \in \{0,s_0-2 \}}\left| \int^2_{-2} \int_{\mathbb{R}^3}  \Lambda_{x'}^{\sigma}\left\{(\partial^i+\partial^i \phi \partial^3)\widetilde{LG}\right\} \cdot \Lambda_{x'}^{\sigma} \widetilde{Lf_i}dxd\tau    \right|.
\end{split}
\end{equation*}
Then we have
For $\sigma \in \{0,s_0-2 \}$, we use H\"older's inequality to give the bound
\begin{equation}\label{J20}
\begin{split}
  \mathrm{I}
  \leq  &   \ \textstyle{\sum_{\sigma \in \{0,s_0-2 \}}}\| \Lambda_{x'}^{\sigma} L F \|_{L^1_t L^2_x} \|\Lambda_{x'}^{\sigma} L(\mathrm{curl}\Omega-2\mathrm{e}^{-{\boldsymbol{\rho}}}\partial_a \boldsymbol{\rho} \partial \varpi^a) \|_{L^\infty_t L^2_x}
  \\
  \lesssim & \ \sum^{3}_{i=1}\sum^6_{j=1}\| R^i_j\|_{L^1_t H^{s_0-2}_x}\|\mathrm{curl}\Omega-2\mathrm{e}^{-{\boldsymbol{\rho}}}\partial_a \boldsymbol{\rho} \partial \varpi^a\|_{L^\infty_t H^{s_0-2}_x}.
\end{split}
\end{equation}
Recall the expressions for $R^i_1, R^i_2, ... R^i_6$ in \eqref{rF}, we get
\begin{equation}\label{R}
  \begin{split}
 &\sum^{3}_{i=1}\sum^6_{j=1}\| R^i_j\|_{L^1_t H^{s_0-2}_x}
 \lesssim   (\|d \boldsymbol{\rho}, d v\|_{L^2_t \dot{B}^{s_0-2}_{\infty,2}}+\|dv,  \partial \boldsymbol{\rho}\|_{L^2_t L^\infty_x}) \| (v, \boldsymbol{\rho}, \varpi) \|_{H^{s_0}}.
  \end{split}
\end{equation}
We also get
\begin{equation}\label{R1}
  \|\mathrm{curl}\Omega-2\mathrm{e}^{-{\boldsymbol{\rho}}}\partial_a \boldsymbol{\rho} \partial \varpi^a\|_{L^\infty_t H^{s_0-2}_x} \lesssim \| \varpi \|_{H^{s_0}}+ \|\boldsymbol{\rho}\|_{H^2}\| \varpi \|_{H^{2}}.
\end{equation}
Inserting \eqref{R} and \eqref{R1} into \eqref{J2} and using \eqref{401}-\eqref{403}, we have
\begin{equation}\label{J2}
\begin{split}
  \mathrm{I} \lesssim &\ \epsilon^2_2.
\end{split}
\end{equation}
As for $\mathrm{J}$, we separate it as
\begin{equation*}
  \mathrm{J} \leq \sum_{\sigma \in \{0,s_0-2 \}}(|J^1|+ |J^2|+|J^3|+|J^4|),
\end{equation*}
where
\begin{equation*}
  \begin{split}
  J_1= &   \int^2_{-2} \int_{\mathbb{R}^3}  (\partial^i+\partial^i \phi \partial^3) \Lambda_{x'}^{\sigma}(\widetilde{LG}) \cdot  \Lambda_{x'}^{\sigma} \widetilde{(L\mathrm{curl}\Omega_i)}dxd\tau ,
  \\
  J_2= &   \int^2_{-2} \int_{\mathbb{R}^3} [\Lambda_{x'}^{\sigma}, \partial^i+\partial^i \phi \partial^3](\widetilde{LG}) \cdot \Lambda_{x'}^{\sigma} \widetilde{(L\mathrm{curl}\Omega_i)} dxd\tau ,
  \\
  J_3= &   \int^2_{-2} \int_{\mathbb{R}^3} (\partial^i+\partial^i \phi \partial^3) \Lambda_{x'}^{\sigma} \widetilde{L \big( 2 \mathrm{e}^{-\boldsymbol{\rho}} \partial_n v^a \partial^n \varpi_a \big)} \cdot \Lambda_{x'}^{\sigma} \widetilde{L \big(-2\mathrm{e}^{-{\boldsymbol{\rho}}} \partial_a \boldsymbol{\rho} \partial_i \varpi^a \big)}  dxd\tau ,
  \\
   J_4= &   \int^2_{-2} \int_{\mathbb{R}^3} [\Lambda_{x'}^{\sigma}, \partial^i+\partial^i \phi \partial^3]  \widetilde{L \big( 2 \mathrm{e}^{-\boldsymbol{\rho}} \partial_n v^a \partial^n \varpi_a \big)} \cdot \Lambda_{x'}^{\sigma} \widetilde{L \big(-2\mathrm{e}^{-{\boldsymbol{\rho}}} \partial_a \boldsymbol{\rho} \partial_i \varpi^a \big)}  dxd\tau .
  \end{split}
  \end{equation*}
Let us bound the above terms one by one. For $J_1$, we have
\begin{equation*}
\begin{split}
  J_1= &  \int^2_{-2} \int_{\mathbb{R}^3}  (\partial^i+\partial^i \phi \partial^3) \Lambda_{x'}^{\sigma}(\widetilde{LG}) \cdot  \Lambda_{x'}^{\sigma} \widetilde{(L\mathrm{curl}\Omega_i)}dxd\tau
  \\
  =& -\int^2_{-2} \int_{\mathbb{R}^3}  \Lambda_{x'}^{\sigma}(\widetilde{LG}) \cdot  (\partial^i+\partial^i \phi \partial^3) \big\{ \Lambda_{x'}^{\sigma} \widetilde{(L\mathrm{curl}\Omega_i)} \big\} dxd\tau
  \\
  =& -\int^2_{-2} \int_{\mathbb{R}^3}  \Lambda_{x'}^{\sigma}(\widetilde{LG}) \cdot  [\partial^i+\partial^i \phi \partial^3, \Lambda_{x'}^{\sigma}]  \widetilde{(L\mathrm{curl}\Omega_i)}  dxd\tau
  \\
  & + \int^2_{-2} \int_{\mathbb{R}^3}  \Lambda_{x'}^{\sigma}(\widetilde{LG}) \cdot    \Lambda_{x'}^{\sigma}  \big\{ (\partial^i+\partial^i \phi \partial^3) \widetilde{(L\mathrm{curl}\Omega_i)} \big\} dxd\tau .
\end{split}
\end{equation*}
Using the fact
\begin{equation*}
  (\partial^i+\partial^i \phi \partial^3) \widetilde{(L\mathrm{curl}\Omega_i)}= (\partial^i L\mathrm{curl}\Omega_i)(x_1,x_2,x_3+\phi(t,x')) =0,
\end{equation*}
then
\begin{equation*}
\begin{split}
  J_1= &  -\int^2_{-2} \int_{\mathbb{R}^3}  \Lambda_{x'}^{\sigma}(\widetilde{LG}) \cdot  [\partial^i+\partial^i \phi \partial^3, \Lambda_{x'}^{\sigma}]  \widetilde{(L\mathrm{curl}\Omega_i)}  dxd\tau.
\end{split}
\end{equation*}
For $\sigma \in \{0,s_0-2 \}$, by H\"older's inequality and commutator estimates in Lemma \ref{ce}, we can derive
\begin{equation}\label{J1}
\begin{split}
  \textstyle{\sum_{\sigma \in \{0,s_0-2 \}}}(|J_1|+|J_2|)= &  \| LG \|_{\dot{H}_x^\sigma} \|\partial \phi\|_{L^2_t L^\infty_x} \|L\mathrm{curl}\Omega\|_{\dot{H}_x^\sigma}
  \\
  \lesssim & \|v\|_{H^2_x} \|\varpi\|_{H^2_x}\|\partial \phi\|_{L^2_t L^\infty_x}\|\varpi\|_{{H}_x^{s_0-1}}.
\end{split}
\end{equation}
For $J_3$, we use Plancherel formula in $\mathbb{R}^3$ such that
\begin{equation}\label{J3}
\begin{split}
  J_3= &\int^2_{-2} \int_{\mathbb{R}^3} \partial^i \Lambda_{x'}^{\sigma} \widetilde{L \big( 2 \mathrm{e}^{-\boldsymbol{\rho}} \partial_n v^a \partial^n \varpi_a \big)} \cdot  \Lambda_{x'}^{\sigma} \widetilde{L \big(-2\mathrm{e}^{-{\boldsymbol{\rho}}} \partial_a \boldsymbol{\rho} \partial_i \varpi^a \big)}  dxd\tau
  \\
  & + \int^2_{-2} \int_{\mathbb{R}^3}   \partial^3 \Lambda_{x'}^{\sigma} \widetilde{L \big( 2 \mathrm{e}^{-\boldsymbol{\rho}} \partial_n v^a \partial^n \varpi_a \big)} \cdot \partial^i \phi  \Lambda_{x'}^{\sigma} \widetilde{L \big(-2\mathrm{e}^{-{\boldsymbol{\rho}}} \partial_a \boldsymbol{\rho} \partial_i \varpi^a \big)}  dxd\tau
  \\
  = &\int^2_{-2} \int_{\mathbb{R}^3} \Lambda^{-\frac{1}{2}}\partial^i \Lambda_{x'}^{\sigma} \widetilde{L \big( 2 \mathrm{e}^{-\boldsymbol{\rho}} \partial_n v^a \partial^n \varpi_a \big)} \cdot \Lambda^{\frac{1}{2}} \Lambda_{x'}^{\sigma} \widetilde{L \big(-2\mathrm{e}^{-{\boldsymbol{\rho}}} \partial_a \boldsymbol{\rho} \partial_i \varpi^a \big)}  dxd\tau
  \\
  & + \int^2_{-2} \int_{\mathbb{R}^3}   \Lambda^{-\frac{1}{2}} \partial^3 \Lambda_{x'}^{\sigma} \widetilde{L \big( 2 \mathrm{e}^{-\boldsymbol{\rho}} \partial_n v^a \partial^n \varpi_a \big)} \cdot \Lambda^{\frac{1}{2}} \left( \partial^i \phi  \Lambda_{x'}^{\sigma} \widetilde{L \big(-2\mathrm{e}^{-{\boldsymbol{\rho}}} \partial_a \boldsymbol{\rho} \partial_i \varpi^a \big)} \right)  dxd\tau.
\end{split}
\end{equation}
For $\sigma \in \{0,s_0-2 \}$, by H\"older inequality, we can obtain
\begin{equation*}
\begin{split}
  & \textstyle{\sum_{\sigma \in \{0,s_0-2 \}}}|J_3|
  \\
  = & \|  \partial v \partial \varpi  \|_{L^\infty_t H^{s_0-\frac{3}{2}}_x}\|  \partial \boldsymbol{\rho} \partial \varpi  \|_{L^\infty_t H^{s_0-\frac{3}{2}}_x}+\|  d \phi  \|_{L^2_t C^{\frac{1}{2}}_x} \|  \partial v \partial \varpi  \|_{L^\infty_t H^{s_0-\frac{3}{2}}_x}\|  \partial \boldsymbol{\rho} \partial \varpi  \|_{L^\infty_t H^{s_0-\frac{3}{2}}_x}
  \\
  \lesssim & \ \| \boldsymbol{\rho}\|_{H^{s_0}} \|v \|_{H^{s_0}}\|\varpi\|_{H^{s_0}}(1+ \vert\kern-0.25ex\vert\kern-0.25ex\vert d\phi-dt\vert\kern-0.25ex\vert\kern-0.25ex\vert_{s_0,2,\Sigma})
\end{split}
\end{equation*}
By using \eqref{401}-\eqref{403}, we get
\begin{equation}\label{J3}
\begin{split}
  \textstyle{\sum_{\sigma \in \{0,s_0-2 \}}}|J_3| \lesssim & \ \epsilon^2_2.
\end{split}
\end{equation}
Substituing \eqref{J1}, ...\eqref{J3} to \eqref{L0} and using \eqref{401}-\eqref{403}, we can update \eqref{L0} by
\begin{equation}\label{L01}
\begin{split}
  \vert\kern-0.25ex\vert\kern-0.25ex\vert Lf\vert\kern-0.25ex\vert\kern-0.25ex\vert^2_{s_0-2,2,\Sigma} \lesssim  \epsilon^2_2.
  \end{split}
\end{equation}
For
\begin{equation*}
  Lf^i=L(\mathrm{curl}\Omega^i)-2L(\mathrm{e}^{-{\boldsymbol{\rho}}}\partial_a \boldsymbol{\rho} \partial^i \varpi^a).
\end{equation*}
we then get
\begin{equation}\label{Rr}
  \vert\kern-0.25ex\vert\kern-0.25ex\vert Lf^i\vert\kern-0.25ex\vert\kern-0.25ex\vert_{s_0-2,2,\Sigma} \geq \vert\kern-0.25ex\vert\kern-0.25ex\vert L\mathrm{curl}\Omega^i\vert\kern-0.25ex\vert\kern-0.25ex\vert_{s_0-2,2,\Sigma}-  2\vert\kern-0.25ex\vert\kern-0.25ex\vert L\left( \mathrm{e}^{-{\boldsymbol{\rho}}}\partial_a \boldsymbol{\rho} \partial^i \varpi^a \right)\vert\kern-0.25ex\vert\kern-0.25ex\vert_{s_0-2,2,\Sigma}.
\end{equation}
By trace theorem and elliptic estimate, we have
\begin{equation}\label{RR}
\begin{split}
 \vert\kern-0.25ex\vert\kern-0.25ex\vert L\left( \mathrm{e}^{-{\boldsymbol{\rho}}}\partial_a \boldsymbol{\rho} \partial^i \varpi^a \right)\vert\kern-0.25ex\vert\kern-0.25ex\vert_{s_0-2,2,\Sigma}
\leq & \ C\| L\left( \mathrm{e}^{-{\boldsymbol{\rho}}}\partial_a \boldsymbol{\rho} \partial^i \varpi^a \right)\|_{H^{s_0-\frac{3}{2}+(s-s_0)}_x}
\\
\leq & \ C\| \mathrm{e}^{-{\boldsymbol{\rho}}} \partial_a \boldsymbol{\rho} \partial^i \varpi^a \|_{H^{s_0-\frac{3}{2}+(s-s_0)}_x}
\\
\leq & \  C\|\partial \boldsymbol{\rho} \|_{H^{s-1}}\|\partial \varpi \|_{H^{1}} \lesssim \epsilon^2_2.
\end{split}
\end{equation}
Sustituing \eqref{Rr} and \eqref{RR} to \eqref{L01}, we therefore get
\begin{equation}\label{CO}
  \vert\kern-0.25ex\vert\kern-0.25ex\vert L ( \mathrm{curl} \Omega)\vert\kern-0.25ex\vert\kern-0.25ex\vert_{s_0-2,2,\Sigma}\lesssim \ \epsilon^2_2 \lesssim \ \epsilon_2.
\end{equation}
If adding \eqref{CO} to \eqref{DO}, then we have
\begin{equation}\label{PO}
  \|\partial \Omega\|_{L^2_t H^{s_0-2}_{x'}(\Sigma)} \lesssim \ \epsilon_2.
\end{equation}
Due to
\begin{equation*}
  \partial_{x'} \Omega = \partial \Omega \cdot (0, \partial_{x'}\phi)^{\mathrm{T}},
\end{equation*}
by H\"older inequality and product estimates in Lemma \ref{cj}, we can deduce the following estimate
\begin{equation}\label{OH}
\begin{split}
  \|\partial_{x'} \Omega\|_{L^2_t H^{s_0-2}_{x'}(\Sigma)} &\lesssim \| \partial \Omega \cdot (0, \partial_{x'}\phi)^{\mathrm{T}}\|_{L^2_t H^{s_0-2}_{x'}(\Sigma)}
  \\
  &\lesssim \| \partial \Omega \|_{L^2_t H^{s_0-2}_{x'}(\Sigma)} \| (0, \partial_{x'}\phi)^{\mathrm{T}}\|_{L^\infty_t H^{s_0-1}_{x'}(\Sigma)}
  \\
  &\lesssim \| \partial \Omega \|_{L^2_t H^{s_0-2}_{x'}(\Sigma)} \| \partial_{x'}\phi \|_{L^\infty_t H^{s_0-1}_{x'}(\Sigma)}
  \\
  &\lesssim \| \partial \Omega \|_{L^2_t H^{s_0-2}_{x'}(\Sigma)} \| \partial_{x'}\phi \|_{s_0, 2, \Sigma} \lesssim \epsilon^2_2.
\end{split}
\end{equation}
Above, we use the trace theorem
$\| \partial_{x'}\phi \|_{L^\infty_t H^{s_0-1}_{x'}(\Sigma)} \lesssim  \vert\kern-0.25ex\vert\kern-0.25ex\vert \partial_{x'}\phi \vert\kern-0.25ex\vert\kern-0.25ex\vert_{s_0, 2, \Sigma}\lesssim  \vert\kern-0.25ex\vert\kern-0.25ex\vert d\phi-dt \vert\kern-0.25ex\vert\kern-0.25ex\vert_{s_0, 2, \Sigma}$.
Through \eqref{OM0} and \eqref{OH}, we can see
\begin{equation}\label{O1}
  \| \Omega \|_{L^2_t H^{s_0-1}_{x'}(\Sigma)} \lesssim \epsilon_2 .
\end{equation}
\textbf{Step 2: $\| \partial_t \Omega \|_{L^2_t H^{s_0-2}_{x'}(\Sigma)}$}. 
By using \eqref{W1} and Sobolev imbeddings, we also have
\begin{equation}\label{OR1}
\begin{split}
  \| \partial_t \Omega \|_{L^2_t H^{s_0-2}_{x'}(\Sigma)}  & \leq    \| - (v \cdot \nabla) \Omega \|_{L^2_t H^{s_0-2}_{x'}(\Sigma)}+ \|\partial v \cdot \partial W \|_{L^2_t H^{s_0-2}_{x'}(\Sigma)}
  \\
  & \leq \| (v \cdot \nabla) \Omega \|_{L^2_t H^{s_0-2}_{x'}(\Sigma)} + \| \partial v \cdot \partial \varpi \|_{L^2_t H^{s_0-2}_{x'}(\Sigma)}
  \\
  & \leq \| v \|_{L^\infty_t H^{s_0-\frac12}_{x'}(\Sigma)} \| \partial \Omega \|_{L^2_t H^{s_0-2}_{x'}(\Sigma)} + \| \partial v \cdot \partial \varpi \|_{L^\infty_t H^{s-\frac{3}{2}}_{x}},
  \\
  & \leq \| v \|_{L^\infty_t H^{s}_{x}(\Sigma)} \| \partial \Omega \|_{L^2_t H^{s_0-2}_{x'}(\Sigma)} + \| \partial v \|_{L^\infty_t H^{s-1}_x}\| \partial \varpi \|_{L^\infty_t H^{1}_{x}}.
\end{split}
\end{equation}
Due to \eqref{OM}, \eqref{O1}, \eqref{OR1}, and \eqref{401}-\eqref{403}, we can show that
\begin{equation}\label{OME}
 \vert\kern-0.25ex\vert\kern-0.25ex\vert \Omega\vert\kern-0.25ex\vert\kern-0.25ex\vert_{s_0-1,2,\Sigma} \lesssim \epsilon_2.
\end{equation}
We finally use \eqref{C1} and \eqref{OME} to conclude that
\begin{equation*}
  \vert\kern-0.25ex\vert\kern-0.25ex\vert \mathrm{curl} \varpi  \vert\kern-0.25ex\vert\kern-0.25ex\vert_{s_0-1,2,\Sigma} \lesssim \epsilon_2.
\end{equation*}
\end{proof}
\begin{Lemma}\label{fre}
{Let $U=(\boldsymbol{\rho}, v)^{\mathrm{T}}$ satisfy the assumption in Lemma \ref{te1}. Then}
\begin{equation}\label{508}
  \vert\kern-0.25ex\vert\kern-0.25ex\vert  2^j(U-S_j U), d S_j U, 2^{-j} d \partial_x S_j U\vert\kern-0.25ex\vert\kern-0.25ex\vert_{s_0-1,2,\Sigma} \lesssim \epsilon_2. 
\end{equation}
\end{Lemma}
\begin{proof}
Let $\Delta_0$ be a standard multiplier of order $0$ on $\mathbb{R}^3$, such that $\Delta_0$ is additionally bounded on $L^\infty(\mathbb{R}^3)$. Clearly,
\begin{equation*}
  (\Delta_0U)_t+ \sum^3_{i=1}A_i(U)(\Delta_0U)_{x_i}= -\sum^3_{i=1}[\Delta_0, A_i(U)]\partial_{x_i}U.
\end{equation*}
Set $G=-\sum^3_{i=1}[\Delta_0, A_i(U)]\partial_{x_i}U$.
Due to Lemma \ref{te1}, we have
\begin{equation}\label{60}
  \vert\kern-0.25ex\vert\kern-0.25ex\vert \Delta_0U\vert\kern-0.25ex\vert\kern-0.25ex\vert^2_{s_0,2,\Sigma} \lesssim \|d U \|_{L^2_t L^{\infty}_x}\| \Delta_0 U\|^2_{L^{\infty}_tH_x^{s}}+\| G\|_{L^{2}_tH_x^{s_0}}\| \Delta_0 U\|_{L^{\infty}_tH_x^{s}}.
\end{equation}
Using commutator estimates, we can prove
\begin{equation*}
  \| G\|_{L^{2}_tH_x^{s_0}} \lesssim \|d U \|_{L^2_t L^{\infty}_x}\| \Delta_0 U\|_{L^{\infty}_tH_x^{s}}.
\end{equation*}
In a result, we can update \eqref{60} by
\begin{equation}\label{60e}
  \vert\kern-0.25ex\vert\kern-0.25ex\vert \Delta_0U\vert\kern-0.25ex\vert\kern-0.25ex\vert^2_{s_0,2,\Sigma} \lesssim \|d U \|_{L^2_t L^{\infty}_x}\| \Delta_0 U\|^2_{L^{\infty}_tH_x^{s}}.
\end{equation}
To control the norm of $2^j(U-S_j U)$, we write
\begin{equation*}
  2^j(U-S_j U)= 2^j \sum_{m\geq j}\Delta_m U,
\end{equation*}
where $\Delta_m U$ satisfies the above conditions for $\Delta_0 U$. Applying \eqref{60e} and replacing $s_0$ to $s_0-1$, we get
\begin{equation*}
\begin{split}
   \vert\kern-0.25ex\vert\kern-0.25ex\vert 2^j(U-S_j U) \vert\kern-0.25ex\vert\kern-0.25ex\vert^2_{s_0-1,2,\Sigma}
  \lesssim & \sum_{m\geq j} \vert\kern-0.25ex\vert\kern-0.25ex\vert 2^m \Delta_m U \vert\kern-0.25ex\vert\kern-0.25ex\vert^2_{s_0-1,2,\Sigma}
  \\
  \lesssim & \|d U \|_{L^2_t L^{\infty}_x} \sum_{m\geq j} (2^m\|\Delta_m U\|^2_{L^{\infty}_tH_x^{s}})
  \\
   \lesssim & \|d U \|_{L^2_t L^{\infty}_x} \| U\|^2_{L^{\infty}_tH_x^{s}} \lesssim \epsilon^2_2.
\end{split}
\end{equation*}
Taking square of it, we can see
\begin{equation*}
   \vert\kern-0.25ex\vert\kern-0.25ex\vert 2^j(U-S_j U) \vert\kern-0.25ex\vert\kern-0.25ex\vert_{s_0-1,2,\Sigma}
  \lesssim  \epsilon_2.
\end{equation*}
Finally, applying \eqref{60} to $\Delta_0=S_{j}$ and $\Delta_0=2^{-j}\partial_x S_{j}$ shows that
\begin{equation*}
  \vert\kern-0.25ex\vert\kern-0.25ex\vert d S_{j}U\vert\kern-0.25ex\vert\kern-0.25ex\vert_{s_0-1,2,\Sigma} +\vert\kern-0.25ex\vert\kern-0.25ex\vert 2^{-j} d \partial_x S_{j}U \vert\kern-0.25ex\vert\kern-0.25ex\vert_{s_0-1,2,\Sigma} \lesssim \epsilon_2.
\end{equation*}
Therefore, the proof of Lemma \ref{fre} is completed.
\end{proof}
We are now ready to give a precise proof of Proposition \ref{r1}.
\begin{proof}[proof of Proposition \ref{r1}]
Note $(v, \boldsymbol{\rho}, \varpi) \in \mathcal{H}$. Using Lemma \ref{fre},
it suffices for us to verify that
\begin{equation*}
\begin{split}
  \vert\kern-0.25ex\vert\kern-0.25ex\vert {\mathbf{g}}^{\alpha \beta}-\mathbf{m}^{\alpha \beta}\vert\kern-0.25ex\vert\kern-0.25ex\vert_{s_0,2,\Sigma_{\theta,r}} \lesssim \epsilon_2.
\end{split}
\end{equation*}
By Lemma \ref{te1} and \eqref{5021}-\eqref{504}, we have
\begin{equation*}
  \sup_{\theta,r}\vert\kern-0.25ex\vert\kern-0.25ex\vert v\vert\kern-0.25ex\vert\kern-0.25ex\vert_{s_0,2,\Sigma_{\theta,r}}+ \sup_{\theta,r}\vert\kern-0.25ex\vert\kern-0.25ex\vert \boldsymbol{\rho}\vert\kern-0.25ex\vert\kern-0.25ex\vert_{s_0,2,\Sigma_{\theta,r}} \lesssim \epsilon_2.
\end{equation*}
Noting the expression of $\mathbf{g}$, by Lemma \ref{te2}, we derive that
\begin{equation*}
\begin{split}
  \|{\mathbf{g}}^{\alpha \beta}-\mathbf{m}^{\alpha \beta}\|_{s_0,2,\Sigma_{\theta,r}} & \lesssim \vert\kern-0.25ex\vert\kern-0.25ex\vert v\vert\kern-0.25ex\vert\kern-0.25ex\vert_{s_0,2,\Sigma_{\theta,r}}+\vert\kern-0.25ex\vert\kern-0.25ex\vert v \cdot v\vert\kern-0.25ex\vert\kern-0.25ex\vert_{s_0,2,\Sigma_{\theta,r}}+\vert\kern-0.25ex\vert\kern-0.25ex\vert c_s^2(\boldsymbol{\rho})-c_s^2(0)\vert\kern-0.25ex\vert\kern-0.25ex\vert_{s_0,2,\Sigma_{\theta,r}}
  \\
  & \lesssim \epsilon_2.
\end{split}
\end{equation*}
Thus, the conclusion of Proposition \ref{r1} holds.
\end{proof}
\subsection{The null frame}
We introduce a null frame along $\Sigma$ as follows. Let
\begin{equation*}
  V=(dr)^*,
\end{equation*}
where $r$ is the defining function of the foliation $\Sigma$, and where $*$ denotes the identification of covectors and vectors induced by $\mathbf{g}$. Then $V$ is the null geodesic flow field tangent to $\Sigma$. Let
\begin{equation}\label{600}
  \sigma=dt(V), \qquad l=\sigma^{-1} V.
\end{equation}
Thus $l$ is the g-normal field to $\Sigma$ normalized so that $dt(l)=1$, hence
\begin{equation}\label{601}
  l=\left< dt,dx_3-d\phi\right>^{-1}_{\mathbf{g}} \left( dx_3-d \phi \right)^*,
\end{equation}
so the coefficients $l^j$ are smooth functions of $v, \rho$ and $d \phi$. Conversely,
\begin{equation}\label{602}
 dx_3-d \phi =\left< l,\partial_{x_3}\right>^{-1}_{\mathbf{g}} l^*,
\end{equation}
so that $d \phi$ is a smooth function of $v, \rho$ and the coefficients of $l$.

Next we introduce the vector fields $e_a, a=1,2$ tangent to the fixed-time slice $\Sigma^t$ of $\Sigma$. We do this by applying Grahm-Schmidt orthogonalization in the metric $\mathbf{g}$ to the $\Sigma^t$-tangent vector fields $\partial_{x_a}+ \partial_{x_a} \phi \partial_{x_3}, a=1, 2$.

Finally, we let
\begin{equation*}
  \underline{l}=l+2\partial_t.
\end{equation*}
It follows that $\{l, \underline{l}, e_1, e_2 \}$ form a null frame in the sense that
\begin{align*}
  & \left<l, \underline{l} \right>_{\mathbf{g}} =2, \qquad \qquad \ \left< e_a, e_b\right>_{\mathbf{g}}=\delta_{ab} \ (a,b=1,2),
  \\
  & \left<l, l \right>_{\mathbf{g}} =\left<\underline{l}, \underline{l} \right>_{\mathbf{g}}=0, \quad \left<l, e_a \right>_{\mathbf{g}}=\left<\underline{l}, e_a \right>_{\mathbf{g}}=0 \ (a=1,2).
\end{align*}
The coefficients of each of the fields is a smooth function of $u$ and $d \phi$, and by assumption we also have the pointwise bound
\begin{equation*}
  | e_1 - \partial_{x_1} | +| e_2 - \partial_{x_2} | + | l- (\partial_t+\partial_{x_3}) | + | \underline{l} - (-\partial_t+\partial_{x_3})|  \lesssim \epsilon_1.
\end{equation*}
After that, we can state the following corollary concerning to the decomposition of curvature tensor.
\begin{Lemma}\label{LLQ}\cite{ST}
Suppose $f$ satisfying $$\mathbf{g}^{\alpha \beta} \partial^2_{\alpha \beta}f=F.$$
Let $(t,x',\phi(t,x'))$ denote the projective parametrisation of $\Sigma$, and for $0 \leq \alpha, \beta \leq 2$, let $/\kern-0.55em \partial_\alpha$ denote differentiation along $\Sigma$ in the induced coordinates. Then, for $0 \leq \alpha, \beta \leq 2$, one can write
\begin{equation*}
  /\kern-0.55em \partial_\alpha /\kern-0.55em \partial_\beta (f|_{\Sigma}) = l(f_2)+ f_1,
\end{equation*}
where
\begin{equation*}
  \| f_2 \|_{L^2_t H^{s_0-1}_{x'}(\Sigma)}+\| f_1 \|_{L^1_t H^{s_0-1}_{x'}(\Sigma)} \lesssim \|df\|_{L^\infty_t H^{s_0-1}}+ \|df\|_{L^2_t L^\infty}+ \| F\|_{L^2_t H^{s_0-1}_x}+ \| F\|_{L^1_t H^{s_0-1}_{x'}(\Sigma)}.
\end{equation*}
\end{Lemma}
\begin{corollary}\label{Rfenjie}
Let $R$ be the Riemann curvature tensor for the metric ${\mathbf{g}}$. Let $e_0=l$. Then for any $0 \leq a, b, c,d \leq 2$, we can write
\begin{equation}\label{603}
  \left< R(e_a, e_b)e_c, e_d \right>_{\mathbf{g}}|_{\Sigma}=l(f_2)+f_1,
\end{equation}
where $|f_1|\lesssim |\mathrm{curl} \varpi|+ |d \mathbf{g} |^2$, $|f_2| \lesssim |d {\mathbf{g}}|$. Moreover, the characteristic estimates
\begin{equation}\label{604}
  \|f_2\|_{L^2_t H^{s_0-1}_{x'}(\Sigma)}+\|f_1\|_{L^1_t H^{s_0-1}_{x'}(\Sigma)} \lesssim \epsilon_2,
\end{equation}
holds. Additionally, for any $t \in [0,T]$,
\begin{equation}\label{605}
  \|f_2(t,\cdot)\|_{C^\delta_{x'}(\Sigma^t)} \lesssim \|d \mathbf{g}\|_{C^\delta_x(\mathbb{R}^3)}.
\end{equation}
\end{corollary}
\begin{proof}
According to the expression of curvature tensor, we have
\begin{equation*}
  \left< R(e_a, e_b)e_c, e_d \right>_{\mathbf{g}}= R_{\alpha \beta \mu \nu}e^\alpha_a e^\beta_b e_c^\mu e_d^\nu,
\end{equation*}
where
\begin{equation*}
  R_{\alpha \beta \mu \nu}= \frac12 \left[ \partial^2_{\alpha \mu} \mathbf{g}_{\beta \nu}+\partial^2_{\beta \nu} \mathbf{g}_{\alpha \mu}-\partial^2_{\beta \mu} \mathbf{g}_{\alpha \nu}-\partial^2_{\alpha \nu} \mathbf{g}_{\beta \mu} \right]+ Q(\mathbf{g}^{\alpha \beta}, d \mathbf{g}_{\alpha \beta}),
\end{equation*}
where $Q$ is a sum of products of coefficients of $\mathbf{g}^{\alpha \beta} $ with quadratic forms in $d \mathbf{g}_{\alpha \beta}$. By using Proposition \ref{r1}, then $Q$ satisfies the bound required of $f_1$. It suffices for us to consider
\begin{equation*}
   \frac12 e^\alpha_a e^\beta_b e_c^\mu e_d^\nu  \left[ \partial^2_{\alpha \mu} \mathbf{g}_{\beta \nu}+\partial^2_{\beta \nu} \mathbf{g}_{\alpha \mu}-\partial^2_{\beta \mu} \mathbf{g}_{\alpha \nu}-\partial^2_{\alpha \nu} \mathbf{g}_{\beta \mu} \right].
\end{equation*}
We therefore look at the term $ e^\alpha_a e_c^\mu \partial^2_{\alpha \mu} \mathbf{g}_{\beta \nu} $, which is typical. By \eqref{503}, Proposition \ref{r1}, and Lemma \ref{te3}, we get
\begin{equation}\label{LLL}
  \| l^\alpha - \delta^{\alpha 0} \|_{s_0,2,\Sigma} +\| \underline{l}^\alpha + \delta^{\alpha 0}-2\delta^{\alpha n} \|_{s_0,2,\Sigma} + \| e^\alpha_a- \delta^{\alpha a} \|_{s_0,2,\Sigma} \lesssim \epsilon_1.
\end{equation}
By \eqref{LLL} and Proposition \ref{r1}, the term $ e_a (e_c^\mu) \partial_{ \mu} \mathbf{g}_{\beta \nu}$ satisfies the bound required of $f_1$, so we consider $e_a(e_c(\mathbf{g}_{\beta \nu}))$. Finally, since the coefficients of $e_c$ in the basis $/\kern-0.55em \partial_\alpha$ have tangential derivatives bounded in $L^2_tH^{s_0-1}_{x'}(\Sigma)$, we are reduced by Lemma \ref{LLQ} to verifying that
\begin{equation*}
 \| \mathbf{g}^{\alpha \beta } \partial^2_{\alpha \beta} \mathbf{g}_{\mu \nu} \|_{L^1_t H^{s_0-1}_{x'}(\Sigma)} \lesssim \epsilon_2.
\end{equation*}
Note
\begin{equation*}
  \mathbf{g}^{\alpha \beta } \partial^2_{\alpha \beta} \mathbf{g}_{\mu \nu}= \square_{\mathbf{g}} \mathbf{g}_{\mu \nu}.
\end{equation*}
By Lemma \ref{te20} and Corollary \ref{vte}, we have
\begin{equation*}
\begin{split}
\| \square_{\mathbf{g}} \mathbf{g}_{\mu \nu}\|_{L^1_t H^{s_0-1}_{x'}(\Sigma)}
 \simeq & \ \| \square_{\mathbf{g}} v\|_{L^2_t H^{s_0-1}_{x'}(\Sigma)}+\| \square_{\mathbf{g}} \boldsymbol{\rho} \|_{L^2_t H^{s_0-1}_{x'}(\Sigma)}
 \\
 \lesssim & \ \| \mathrm{curl} \varpi\|_{L^2_t H^{s_0-1}_{x'}(\Sigma)}+ \| (d{\mathbf{g}})^2\|_{L^1_t H^{s_0-1}_{x'}(\Sigma)}
 \\
 \lesssim & \ \| \mathrm{curl} \varpi\|_{L^2_t H^{s_0-1}_{x'}(\Sigma)}+ \| d{\mathbf{g}} \|_{L^2_t L^\infty_x }\| d{\mathbf{g}}\|_{L^2_t H^{s_0-1}_{x'}(\Sigma)}
 \\
 \lesssim & \ \| \mathrm{curl} \varpi\|_{L^2_t H^{s_0-1}_{x'}(\Sigma)}+ \| dv, d\boldsymbol{\rho} \|_{L^2_t L^\infty_x }\| dv, d\boldsymbol{\rho}\|_{L^2_t H^{s_0-1}_{x'}(\Sigma)}
 \\
 \lesssim & \ \epsilon_2.
\end{split}
\end{equation*}
Above, $\mathrm{curl} \varpi$ and $(d{\mathbf{g}})^2$ is included in $f_1$. We hence complete the proof of Corollary \ref{Rfenjie}.
\end{proof}

\subsection{The estimate of connection coefficients}
Define
\begin{equation*}
  \chi_{ab} = \left<D_{e_a}l,e_b \right>_{\mathbf{g}}, \qquad l(\ln \sigma)=\frac{1}{2}\left<D_{l}\underline{l},l \right>_{\mathbf{g}}, \quad \mu_{0ab} = \left<D_{l}e_a,e_b \right>_{\mathbf{g}}.
\end{equation*}
For $\sigma$, we set the initial data $\sigma=1$ at the time $-2$. Thanks to Proposition \ref{r1}, we have
\begin{equation}\label{606}
  \|\chi_{ab}\|_{L^2_t H^{s_0-1}_{x'}(\Sigma)} + \| l(\ln \sigma)\|_{L^2_t H^{s_0-1}_{x'}(\Sigma)} \lesssim \epsilon_1.
\end{equation}
In a similar way, if we expand $l=l^\alpha /\kern-0.55em \partial_\alpha$ in the tangent frame $\partial_t, \partial_{x'}$ on $\Sigma$, then
\begin{equation}\label{607}
  l^0=1, \quad \|l^1\|_{s_0,2,\Sigma} \lesssim \epsilon_1.
\end{equation}
\begin{Lemma}\label{chi}
Let $\chi$ be defined as before. Then
\begin{equation}\label{608}
  \|\chi_{ab}\|_{L^2_t H^{s_0-1}_{x'}(\Sigma)} \lesssim \epsilon_2.
\end{equation}
Furthermore, for any $t \in [-2,2]$,
\begin{equation}\label{609}
\| \chi_{ab} \|_{C^{\delta}_{x'}(\Sigma^t)} \lesssim \epsilon_2+ \|d \mathbf{g} \|_{C^{\delta}_{x}(\mathbb{R}^3)}.
\end{equation}
\end{Lemma}
\begin{proof}
The well-known transport equation for $\chi$ along null hypersurfaces (see references \cite{KR2} and \cite{ST}) can be described as
\begin{equation*}
  l(\chi_{ab})=\left< R(l,e_a)l, e_b \right>_{\mathbf{g}}-\chi_{ac}\chi_{cb}-l(\ln \sigma)\chi_{ab}+\mu_{0ab} \chi_{cb}+ \mu_{0bc}\chi_{ac}.
\end{equation*}
Due to Corollary \ref{Rfenjie}, we can write the above equation in the following
\begin{equation}\label{610}
  l(\chi_{ab}-f_2)=f_1-\chi_{ac}\chi_{cb}-l(\ln \sigma)\chi_{ab}+\mu_{0ab} \chi_{cb}+ \mu_{0bc}\chi_{ac},
\end{equation}
where
\begin{equation}\label{611}
  \|f_2\|_{L^2_t H^{s_0-1}_{x'}(\Sigma)}+\|f_1\|_{L^1_t H^{s_0-1}_{x'}(\Sigma)} \lesssim \epsilon_2,
\end{equation}
and for any $t \in [0,T]$,
\begin{equation}\label{612}
  \|f_2(t,\cdot)\|_{C^\delta_{x'}(\Sigma^t)} \lesssim \|d \mathbf{g}\|_{C^\delta_x(\mathbb{R}^3)}.
\end{equation}
Let $\Lambda_{x'}^{s_0-1}$ be the fractional derivative operator in the $x'$ variables. Set
\begin{equation*}
  G=f_1-\chi_{ac}\chi_{cb}-l(\ln \sigma)\chi_{ab}+\mu_{0ab} \chi_{cb}+ \mu_{0bc}\chi_{ac}.
\end{equation*}
We have
\begin{equation}\label{613}
  \begin{split}
 & \|\Lambda^{s_0-1}(\chi_{ab}-f_2)(t,\cdot) \|_{L^2_{x'}(\Sigma^t)}
  \\
  \lesssim \ &\| [\Lambda^{s_0-1},l](\chi_{ab}-f_2) \|_{L^1_tL^2_{x'}(\Sigma^t)}+ \| \Lambda^{s_0-1}G \|_{L^1_tL^2_{x'}(\Sigma^t)}.
  \end{split}
\end{equation}
A direct calculation shows that
\begin{equation}\label{614}
  \begin{split}
 \| \Lambda^{s-1}G \|_{L^1_tL^2_{x'}(\Sigma^t)} &\lesssim \|f_1\|_{L^1_tH^{s-1}_{x'}(\Sigma^t)}+ \|\chi\|^2_{L^2_tH^{s-1}_{x'}(\Sigma^t)}
  \\
  & \quad + \|\chi\|_{L^2_tH^{s-1}_{x'}(\Sigma^t)}\cdot\|l(\ln \sigma)\|_{L^2_tH^{s-1}_{x'}(\Sigma^t)}
  \\
  & \quad + \|\mu\|_{L^2_tH^{s-1}_{x'}(\Sigma^t)}\cdot\|\chi\|_{L^2_tH^{s-1}_{x'}(\Sigma^t)},
  \end{split}
\end{equation}
where we use the fact that $H^{s-1}_{x'}(\Sigma^t)$ is an algebra for $s>2$.

We next bound
\begin{align*}
  \| [\Lambda^{s_0-1},l](\chi_{ab}-f_2) \|_{L^2_{x'}(\Sigma^t)} &\leq \| /\kern-0.55em \partial_{\alpha} l^{\alpha} (\chi_{ab}-f_2)(t,\cdot) \|_{H^{s_0-1}_{x'}(\Sigma^t)}
  \\
  & \quad \ + \|[\Lambda^{s_0-1} /\kern-0.55em \partial_{\alpha}, l^{\alpha}](\chi-f_2)(t,\cdot) \|_{L^{2}_{x'}(\Sigma^t)}.
\end{align*}
By Kato-Ponce commutator estimate and Sobolev embeddings, the above is bounded by
\begin{equation}\label{615}
  \|l^1(t,\cdot)\|_{H^{s_0-1}_{x'}(\Sigma^t)} \| \Lambda^{s_0-1}(\chi_{ab}-f_2)(t,\cdot) \|_{L^{2}_{x'}(\Sigma^t)} .
\end{equation}
Gathering \eqref{606}, \eqref{607}, \eqref{611} and \eqref{613}-\eqref{615} together, we thus prove that
\begin{equation*}
  \sup_t \|(\chi_{ab}-f_2)(t,\cdot)\|_{H^{s_0-1}_{x'}(\Sigma^t)}  \lesssim \epsilon_2.
\end{equation*}
From \eqref{610}, we see that
\begin{equation}\label{616}
\begin{split}
  \| \chi_{ab}-f_2\|_{C^{\delta}_{x'}} & \lesssim \| f_1 \|_{L^1_tC^{\delta}_{x'}}+ \|\chi_{ac}\chi_{cb}\|_{L^1_tC^{\delta}_{x'} }+\|l(\ln \sigma)\chi_{ab}\|_{L^1_tC^{\delta}_{x'}}
  \\
  & \quad + \|\mu_{0ab} \chi_{cb}\|_{L^1_tC^{\delta}_{x'} }+\|\mu_{0bc}\chi_{ac}\|_{L^1_tC^{\delta}_{x'} }.
  \end{split}
\end{equation}
By Sobolev imbedding $H^{s_0-1}(\mathbb{R}^2)\hookrightarrow C^{\delta}(\mathbb{R}), \delta \in (0, s_0-2)$ and Gronwall's inequality, we derive that
\begin{equation*}
\| \chi_{ab} \|_{C^{\delta}_{x'}(\Sigma^t)} \lesssim \epsilon_2+ \|d \mathbf{g}\|_{C^{\delta}_{x}(\mathbb{R}^2)}.
\end{equation*}
\end{proof}
\subsection{The proof of Proposition \ref{r2}}
Note
\begin{equation*}
  G(v, \boldsymbol \rho)= \vert\kern-0.25ex\vert\kern-0.25ex\vert d\phi(t,x')-dt\vert\kern-0.25ex\vert\kern-0.25ex\vert_{s,2, \Sigma}.
\end{equation*}
Using \eqref{602} and the estimate of $\vert\kern-0.25ex\vert\kern-0.25ex\vert \mathbf{g}-\mathbf{m} \vert\kern-0.25ex\vert\kern-0.25ex\vert_{s_0,2,\Sigma}$ in Proposition \ref{r1}, then the estimate \eqref{G} follows from the bound
\begin{equation*}
  \vert\kern-0.25ex\vert\kern-0.25ex\vert l-(\partial_t-\partial_{x_3})\vert\kern-0.25ex\vert\kern-0.25ex\vert_{s_0,2,\Sigma} \lesssim \epsilon_2,
\end{equation*}
where it is understood that one takes the norm of the coefficients of $l-(\partial_t-\partial_{x_3})$ in the standard frame on $\mathbb{R}^{3+1}$. The geodesic equation, together with the bound for Christoffel symbols $\|\Gamma^\alpha_{\beta \gamma}\|_{L^2_t L^\infty_x} \lesssim \|d {\mathbf{g}} \|_{L^2_t L^\infty_x}\lesssim \epsilon_2$, imply that
\begin{equation*}
  \|l-(\partial_t-\partial_{x_3})\|_{L^\infty_{t,x}} \lesssim \epsilon_2,
\end{equation*}
so it suffices to bound the tangential derivatives of the coefficients of $l-(\partial_t-\partial_{x_3})$ in the norm $L^2_t H^{s_0-1}_{x'}(\Sigma)$. By using Proposition \ref{r1}, we can estimate the Christoffel symbols
\begin{equation*}
  \|\Gamma^\alpha_{\beta \gamma} \|_{L^2_t H^{s_0-1}_{x'}(\Sigma^t)} \lesssim \epsilon_2.
\end{equation*}
Note that $H^{s_0-1}_{x'}(\Sigma^t)$ is a algebra if $s_0>2$. We then deduce that
\begin{equation*}
  \|\Gamma^\alpha_{\beta \gamma} e_1^\beta l^\gamma\|_{L^2_t H^{s_0-1}_{x'}(\Sigma^t)} \lesssim \epsilon_2.
\end{equation*}
We are now in a position to establish the following bound,
\begin{equation*}
  \| \left< D_{e_a}l, e_b \right>\|_{L^2_t H^{s_0-1}_{x'}(\Sigma^t)}+ \| \left< D_{e_a}l, \underline{l} \right>\|_{L^2_t H^{s_0-1}_{x'}(\Sigma^t)}+\|\left< D_{l}l, \underline{l} \right>\|_{L^2_t H^{s_0-1}_{x'}(\Sigma^t)} \lesssim \epsilon_2.
\end{equation*}
The first term is $\chi_{ab}$ which is bounded by Lemma \ref{chi}. For the second term, noting
\begin{equation*}
  \left< D_{e_a}l, \underline{l} \right>=\left< D_{e_a}l, 2\partial_t \right>=-2\left< D_{e_a}\partial_t,l \right>,
\end{equation*}
then it can be bounded via Proposition \ref{r1}. Similarly, we can control the last term through Proposition \ref{r1}. At this stage, it remains for us to show that
\begin{equation*}
  \| d \phi(t,x')-dt \|_{C^{1,\delta}_{x'}(\mathbb{R}^2)}  \lesssim \epsilon_2+ \| d\mathbf{g}(t,\cdot)\|_{C^{\delta}_x(\mathbb{R}^3)}.
\end{equation*}
To do that, it suffices for us to show that
\begin{equation*}
  \|l(t,\cdot)-(\partial_t-\partial_{x_3})\|_{C^{1,\delta}_{x'}(\mathbb{R}^2)} \lesssim \epsilon_2+ \| d \mathbf{g} (t,\cdot)\|_{C^{\delta}_x(\mathbb{R}^3)}.
\end{equation*}
The coefficients of $e_1$ are small in $C^{\delta}_{x'}(\Sigma^t)$ perturbations of their constant coefficient analogs, so it suffices to show that
\begin{equation*}
 \|\left< D_{e_1}l, e_1 \right>(t,\cdot)\|_{C^{\delta}_{x'}(\Sigma^t)}
 +\|\left< D_{e_1}l, \underline{l} \right>(t,\cdot)\|_{C^{\delta}_{x'}(\Sigma^t)}  \lesssim \epsilon_2+ \| d\mathbf{g}(t,\cdot)\|_{C^{\delta}_x(\mathbb{R}^3)}.
\end{equation*}
Above, the first term is bounded by Lemma \ref{chi}, and the second by using
\begin{equation*}
  \|\left< D_{e_1}\partial_t, l \right>(t,\cdot)\|_{C^{\delta}_{x'}(\Sigma^t)} \lesssim  \| d\mathbf{g}(t,\cdot)\|_{C^{\delta}_x(\mathbb{R}^3)}.
\end{equation*}
Consequently, we complete the proof of Proposition \ref{r2}.
\section{Strichartz estimates for solutions}
In this part, let us introce the Strichartz estimates of linear wave equations.
\begin{proposition}\label{r3}
Suppose that $(v,\boldsymbol{\rho}, \varpi) \in \mathcal{{H}}$ and $G(v,\boldsymbol{\rho})\leq 2 \epsilon_1$.
For any $1 \leq r \leq s_0+1$, and for each $t_0 \in [0,T]$, the linear, non-homogenous equation
	\begin{equation*}
	\begin{cases}
	& \square_g f=\mathbf{T}G+B, \qquad (t,x) \in [0,T]\times \mathbb{R}^3,
	\\
	&f(t_0,\cdot)=f_0 \in H^r(\mathbb{R}^3), \quad \partial_t f(t_0,\cdot)=f_1 \in H^{r-1}(\mathbb{R}^3),
	\end{cases}
	\end{equation*}
admites a solution $f \in C([0,T],H^r) \times C^1([0,T],H^{r-1})$ and the following estimates holds:
\begin{equation}\label{lw0}
\| f\|_{L_t^\infty H^r}+ \|\partial_t f\|_{L_t^\infty H^{r-1}} \lesssim \|f_0\|_{H^r}+ \|f_1\|_{H^{r-1}}+\| G\|_{L^\infty_tH^{r-1}\cap L^1_tH^r}+\|B\|_{L^1_tH^{r-1}}.
\end{equation}
Additionally, the following estimates hold, provided $k<r-1$,
\begin{equation}\label{lw1}
\| \left<\partial \right>^k f\|_{L^2_{t}L^\infty_x} \lesssim \|f_0\|_{H^r}+ \|f_1\|_{H^{r-1}}+\| G\|_{L^\infty_tH^{r-1}\cap L^1_tH^r}+\|B\|_{L^1_tH^{r-1}}.
\end{equation}
\end{proposition}
\begin{proposition}\label{r5}
Suppose that $(v,\boldsymbol{\rho}, \varpi) \in \mathcal{{H}}$ and $G(v,\boldsymbol{\rho})\leq 2 \epsilon_1$. For each $1 \leq r \leq s_0+1$, then the linear, homogenous equation
\begin{equation*}
	\begin{cases}
	& \square_g F=0,
	\\
	&F(t_0,\cdot)=F_0, \quad \partial_t F(t_0,\cdot)=F_1,
	\end{cases}
	\end{equation*}
is well-posed for the initial data $(F_0, F_1)$ in $H^r \times H^{r-1}$. Moreover, the solution satisfies, respectively, the energy estimates
\begin{equation*}
  \| F\|_{H^r}+ \| \partial_t F\|_{H^{r-1}} \leq C \big( \|F_0\|_{H^r}+ \|F_1\|_{H^{r-1}} \big),
\end{equation*}
and the Strichartz estimates
\begin{equation}\label{SL}
  \| \left<\partial \right>^k F\|_{L^2_{t}L^\infty_x} \leq C \big( \|F_0\|_{H^r}+ \|F_1\|_{H^{r-1}} \big).
\end{equation}
\end{proposition}
We will use Proposition \ref{r5} to prove Proposition \ref{r3} by Lemma \ref{LD}.
\begin{proof}[proof of Proposition \ref{r3} by Proposition \ref{r5}]
Let $\mathbf{g}$ be defined in Proposition \ref{r3}. Let $V$ satisfy the the linear, homogenous equation
\begin{equation}\label{Vf}
	\begin{cases}
	& \square_{\mathbf{g}} V=0,
	\\
	&V(t_0,\cdot)=f_0, \quad \mathbf{T} V(t_0,\cdot)=f_1-G(t_0),
	\end{cases}
	\end{equation}
where $(f_0, f_1-G(t_0)) \in H^r \times H^{r-1}$.
Let $Q$ satisfy the the linear, nonhomogenous equation
\begin{equation}\label{Qf}
	\begin{cases}
	& \square_{\mathbf{g}} Q=\mathbf{T}G+B,
	\\
	&Q(t_0,\cdot)=0, \quad \mathbf{T} Q(t_0,\cdot)=G(t_0),
	\end{cases}
	\end{equation}
where $G(0) \in H^{r-1}$. To write $V$ and $Q$ adapting to Proposition \ref{r5}, we can rewrite \eqref{Vf} and \eqref{Qf} by using $\mathbf{T}=\partial_t+v \cdot \nabla$, respectively,
 \begin{equation*}
	\begin{cases}
	& \square_{\mathbf{g}} V=0,
	\\
	&V(t_0,\cdot)=f_0, \quad \partial_t V(t_0,\cdot)=f_1-G(t_0)-v\cdot \nabla f_0(t_0,\cdot),
	\end{cases}
	\end{equation*}
and
\begin{equation*}
	\begin{cases}
	& \square_{\mathbf{g}} Q=\mathbf{T}G+B,
	\\
	&Q(t_0,\cdot)=0, \quad \partial_t Q(t_0,\cdot)=G(t_0),
	\end{cases}
	\end{equation*}
 By superposition principle for linear wave equation, then
$f= V+Q$ satisfying
\begin{equation*}
	\begin{cases}
	& \square_{\mathbf{g}} f=\mathbf{T}G+B, \qquad (t,x) \in [0,T]\times \mathbb{R}^3,
	\\
	&f(t_0,\cdot)=f_0 \in H^r(\mathbb{R}^3), \quad \mathbf{T} f(t_0,\cdot)=f_1 \in H^{r-1}(\mathbb{R}^3).
	\end{cases}
\end{equation*}
To prove \eqref{lw0} and \eqref{lw1}, we transfer the goal to bound $V$ and $Q$. By Proposition \ref{r5}, we can see
\begin{equation*}
  \| V\|_{H^r}+ \| \partial_t V\|_{H^{r-1}} \leq C \big( \|f_0\|_{H^r}+ \|f_1-G(t_0)-v\cdot \nabla V(t_0,\cdot)\|_{H^{r-1}} \big),
\end{equation*}
and
\begin{equation*}
  \| \left<\partial \right>^k V\|_{L^2_{t}L^\infty_x} \leq C \big( \|f_0\|_{H^r}+ \|f_1-G(t_0)-v\cdot \nabla V(t_0,\cdot)\|_{H^{r-1}} \big), \quad k<r-1.
\end{equation*}
Using $s_0> 2$ and $r<s_0+1$, then we can obtain
we can see
\begin{equation}\label{VE}
\begin{split}
  \| V\|_{H^r}+ \| \partial_t V\|_{H^{r-1}} \leq &C \big( \|f_0\|_{H^r}+ \|f_1\|_{H^{r-1}}+\|G(t_0)\|_{H^{r-1}}+\| v\|_{H^{s_0}} \| \nabla f_0(t_0,\cdot)\|_{H^{r-1}} \big)
  \\
  \leq & C \big( \|f_0\|_{H^r}+ \|f_1\|_{H^{r-1}}+\|G\|_{L^\infty_tH^{r-1}}\big),
\end{split}
\end{equation}
and
\begin{equation}\label{SV}
  \| \left<\partial \right>^k V\|_{L^2_{t}L^\infty_x} \leq C \big( \|f_0\|_{H^r}+ \|f_1\|_{H^{r-1}}+\|G\|_{L^\infty_tH^{r-1}} \big), \quad k<r-1.
\end{equation}
By Lemma \ref{LD} and using Proposition \ref{r5} again, we can prove
\begin{equation}\label{QE}
  \| Q\|_{H^r}+ \| \partial_t Q\|_{H^{r-1}} \leq C \big( \|G(t_0)\|_{H^{r-1}}+ \|G\|_{L^1_tH^{r}}+ \|B\|_{L^1_tH^{r-1}}\big),
\end{equation}
and
\begin{equation}\label{SQ}
  \| \left<\partial \right>^k V\|_{L^2_{t}L^\infty_x} \leq C \big( \|G(t_0)\|_{H^{r-1}}+ \|G\|_{L^1_tH^{r}}+ \|B\|_{L^1_tH^{r-1}}\big), \quad k<r-1.
\end{equation}
Adding \eqref{VE} and \eqref{QE}, we get \eqref{lw0}. Adding \eqref{SV} and \eqref{SQ}, we get \eqref{lw1}. Therefore, we finish the proof of Proposition \ref{r3}.
\end{proof}
Based on Proposition \ref{r3}, we can derive the following result.
\begin{proposition}\label{r4}
Suppose $(v, \rho, \varpi) \in \mathcal{H}$ and $G(v, \rho)\leq 2 \epsilon_1$. Let $v_{+}$ be defined in \eqref{dvc}. We have
\begin{equation}\label{strr}
\|d v, d \boldsymbol{\rho}, \partial v_{+}\|_{L^2_t C^\delta_x}+\|\partial v_{+}, d \boldsymbol{\rho}, dv\|_{L^2_t \dot{B}^{s_0-2}_{\infty,2}} \leq \epsilon_2,
\end{equation}
and
\begin{equation}\label{eef}
\| v\|_{L^\infty_tH^s} +\|\boldsymbol{\rho}\|_{L^\infty_tH^s}+\| \varpi\|_{L^\infty_tH^{s_0}}\leq \epsilon_2.
\end{equation}
\end{proposition}
\begin{proof}
Recall $\boldsymbol{\rho}$ and $v_{+}$ satisfying
\begin{align}\label{fcr}
\begin{split}
 &\square_{\mathbf{g}} \boldsymbol{\rho}=\mathcal{D},
\\
&\square_{\mathbf{g}} v_{+}=\mathbf{T}\mathbf{T} \eta+ \mathcal{R}.
\end{split}
\end{align}
Using Lemma \ref{yux}, we get
\begin{equation*}
  \|\mathcal{D}, \mathcal{R}\|_{L^1_tH^{s-1}}  \lesssim  \|d v, d\boldsymbol{\rho}\|_{L^1_t L_x^\infty} \|d v, d{\boldsymbol{\rho}}\|_{L^\infty_tH^{s-1}}.
\end{equation*}
Operating $\Delta_j$ on \eqref{fcr}, we can derive that $\Delta_j \boldsymbol{\rho}$ and $\Delta_j v^i_{+}$ satisfy
\begin{equation*}
\begin{cases}
 &\square_g \Delta_j \boldsymbol{\rho}= \Delta_j \mathcal{D}+ [\square_g, \Delta_j]\boldsymbol{\rho},
 \\
  & \Delta_j \boldsymbol{\rho}|_{t=0}= \Delta_j \boldsymbol{\rho}_0,
\end{cases}
\end{equation*}
and
\begin{equation*}
\begin{cases}
& \square_g \Delta_j v^i_{+}=
\mathbf{T}\Delta_j \mathbf{T} \eta^i+ \Delta_j Q^i+ [\square_g, \Delta_j]v^i_{+}+[\Delta_j, \mathbf{T}]\mathbf{T} \eta^i,
\\
& \Delta_j v^i_{+}|_{t=0}=\Delta_j (v_0-\eta_0), \quad \Delta_j \mathbf{T}v^i_{+}|_{t=0}=\Delta_j \mathbf{T}(v_0-\eta_0).
\end{cases}
\end{equation*}
By using the Strichartz estimate in Proposition \ref{r5} (taking $r=s-s_0+1, k=0$) 
, we obtain
\begin{equation}\label{ise}
\begin{split}
  \| (\Delta_j \boldsymbol{\rho}, \Delta_j v_{+})\|_{L^2_t L^\infty_x} \lesssim \ & \|\Delta_j \boldsymbol{\rho}(0, \cdot)\|_{H^{s-s_0+1}}  + \| \Delta_j (v_0-\eta_0)\|_{H^{s-s_0+1}}+ \|\Delta_j \mathbf{T} \eta(0, \cdot)\|_{H^{s-s_0+1}}
  \\
  & + \| \Delta_j Q \|_{L^1_t H^{s-s_0}}+ \| [\square_g, \Delta_j]\boldsymbol{\rho}\|_{L^1_t H^{s-s_0}}+ \|[\square_g, \Delta_j]v_{+}\|_{L^1_t H^{s-s_0}}
  \\
  & + \| [\Delta_j, \mathbf{T}]\mathbf{T} \eta \|_{L^1_t H^{s-s_0}}+  \|\Delta_j \mathbf{T} \eta \|_{H_x^{s-s_0+1}}+ \| \Delta_j \mathcal{D} \|_{L^1_t H^{s-s_0}}
\end{split}
\end{equation}
Multiplying $2^{(s_0-1)j}$ on \eqref{ise}, and taking square of it and summing it over $j\geq 1$, we get
\begin{equation}\label{fgh}
\begin{split}
  \| \boldsymbol{\rho}, v_{+}\|^2_{L^2_t \dot{B}^{s_0-1}_{\infty, 2}} & \lesssim \| \boldsymbol{\rho}_0\|^2_{H^s}+\| v_0\|^2_{H^s}+  \|\eta_0\|^2_{H^s}+\| \mathbf{T}\eta_0\|^2_{H^s}+ {\| \varpi_0\|^2_{H^{s_0}}}.
\end{split}
\end{equation}
By Lemma \ref{yux}, we also update \eqref{fgh} by
\begin{equation*}
  \| \boldsymbol{\rho}, v_{+}\|^2_{L^2_t \dot{B}^{s_0-1}_{\infty, 2}}  \lesssim \| \boldsymbol{\rho}_0\|^2_{H^s}+\| v_0\|^2_{H^s}+ {\| \varpi_0\|^2_{H^{s_0}}}.
\end{equation*}
In a result, we have
\begin{equation*}
  \| \partial \boldsymbol{\rho}, \partial v_{+}\|_{L^2_t \dot{B}^{s_0-2}_{\infty, 2}} \lesssim \| \boldsymbol{\rho}_0\|_{H^s}+\|v_0\|_{H^s}+ {\| \varpi_0\|_{H^{s_0}}} \lesssim \epsilon_2.
\end{equation*}
To get the Strichartz estimates for $v$, let us recall it's relation with $v_{+}$ and $\boldsymbol{\rho}$. That is,
\begin{equation*}
  v=\eta+v_{+}, \quad \eta=(-\Delta)^{-1} (\mathrm{e}^{\boldsymbol{\rho}}\mathrm{curl} \varpi).
\end{equation*}
We then get
\begin{equation}\label{800}
\begin{split}
  \| \partial v\|_{\dot{B}^{s_0-2}_{\infty, 2}} \leq \| \partial v_{+}\|_{\dot{B}^{s_0-2}_{\infty, 2}}+\| \partial \eta\|_{\dot{B}^{s_0-2}_{\infty, 2}}.
  \end{split}
\end{equation}
To bound $\| \partial v\|_{\dot{B}^{s_0-2}_{\infty, 2}}$, we need first bound $\| \partial \eta\|_{\dot{B}^{s_0-2}_{\infty, 2}}$. Note
\begin{equation*}
  \partial \eta=\partial(-\Delta)^{-1} (\mathrm{e}^{\boldsymbol{\rho}}\mathrm{curl} \varpi).
\end{equation*}
By Sobolev inequality and elliptic estimate, we have
\begin{equation}\label{900}
\begin{split}
  \| \partial \eta\|_{\dot{B}^{s_0-2}_{\infty, 2}} & \lesssim \| \partial(-\Delta)^{-1} (\mathrm{e}^{\boldsymbol{\rho}}\mathrm{curl} \varpi)\|_{\dot{B}^{s_0-\frac{1}{2}}_{2,2}}
   \\
   & \lesssim \|\mathrm{e}^{\boldsymbol{\rho}} \mathrm{curl}\varpi\|_{\dot{H}^{s_0-\frac{3}{2}}}
   \\
 & \lesssim (1+\|\boldsymbol{\rho}_0\|_{H^{s_0}})\|\varpi_0\|_{H^{s_0}} \lesssim \epsilon_2.
 \end{split}
\end{equation}
Combining \eqref{800} and \eqref{900}, we can see
\begin{equation*}
  \|\partial v\|_{L^2_t \dot{B}^{s_0-2}_{\infty,2}}\lesssim \| \partial v_{+}\|_{\dot{B}^{s_0-2}_{\infty, 2}}+\| \partial \eta\|_{\dot{B}^{s_0-2}_{\infty, 2}} \lesssim \epsilon_2.
\end{equation*}
We hence prove that
\begin{equation}\label{Sone}
  \| \partial \boldsymbol{\rho}, \partial v, \partial v_{+}\|_{L^2_t \dot{B}^{s_0-2}_{\infty, 2}} \lesssim \epsilon_2.
\end{equation}
Using the original transport equation \eqref{fc0}:
\begin{equation*}
  \mathbf{T} \boldsymbol{\rho}=-\mathrm{div}v, \quad \mathbf{T} v=-c^2_s \partial \boldsymbol{\rho},
\end{equation*}
we then obtain
\begin{equation}\label{Stwo}
  \| \mathbf{T} \boldsymbol{\rho},\mathbf{T} v\|_{L^2_t \dot{B}^{s_0-2}_{\infty, 2}} \lesssim \| \partial \boldsymbol{\rho}, \partial v\|_{L^2_t \dot{B}^{s_0-2}_{\infty, 2}} \lesssim \epsilon_2.
\end{equation}
Combining \eqref{Sone} and \eqref{Stwo} together, we have
\begin{equation}\label{Sf}
  \| d \boldsymbol{\rho},d v\|_{L^2_t \dot{B}^{s_0-2}_{\infty, 2}} \lesssim  \epsilon_2.
\end{equation}
It also remains for us to bound $\| d \boldsymbol{\rho}, dv\|_{L^2_t L^\infty_x}$ and $\|d v, d \boldsymbol{\rho}, \partial v_{+}\|_{L^2_t C^\delta_x}$. By Littlewood-Palay decomposition, we also have
\begin{equation*}
\begin{split}
  \| d \boldsymbol{\rho}, dv\|_{L^2_t L^\infty_x} \lesssim & \| d \boldsymbol{\rho}, dv\|_{L^2_t \dot{B}^0_{\infty,2}}
  \\
  \lesssim & \left\{2^{j(s_0-2)} (\Delta_j d\boldsymbol{\rho}, \Delta_j dv) \right\}_{l^2_j} \left\{2^{-j(s_0-2)} \right\}_{l^2_j}
  \\
  \lesssim  & \| d \boldsymbol{\rho},d v\|_{L^2_t \dot{B}^{s_0-1}_{\infty, 2}}\lesssim \epsilon_2.
\end{split}
\end{equation*}
Similarly, for $\delta \in (0, s_0-2)$, we get
\begin{equation*}
  \|d v, d \boldsymbol{\rho}, \partial v_{+}\|_{C^\delta_x} \lesssim  \ \|d v, d \boldsymbol{\rho}, \partial v_{+}\|_{\dot{B}^{\delta}_{\infty,2}} + \|d v, d \boldsymbol{\rho}, \partial v_{+}\|_{\dot{B}^{0}_{\infty,2}}.
\end{equation*}
Therefore, we have
\begin{equation*}
  \|d v, d \boldsymbol{\rho}, \partial v_{+}\|_{L^2_t  C^\delta_x} \lesssim \| d \boldsymbol{\rho},d v\|_{L^2_t \dot{B}^{s_0-1}_{\infty, 2}}\lesssim \epsilon_2.
\end{equation*}
Using Theorem \ref{be}, we can obtain \eqref{eef}. At this stage, we have finished the proof of Proposition \ref{r4}. The only task is Proposition \ref{r5}, which will be stated in the next section.
\end{proof}
\section{Appendix: proof of Proposition \ref{r5}}
Following Smith-Tataru's paper \cite{ST}, we will give a proof of Proposition \ref{r5} in this part. We divide the proofs into several steps. The first step is to reduce the problem in the phase space.
\subsection{A reduction to the paradifferential decomposition}
Given a frequency scale $\lambda \geq 1$, we consider the smooth coefficients
\begin{equation*}
  \mathbf{g}_{\lambda}= S_{<\lambda} \mathbf{g}.
\end{equation*}
We can reduce the Proposition \ref{r5} to the following:
\begin{proposition}\label{A1}
Suppose that $(v,\boldsymbol{\rho}, \varpi) \in \mathcal{{H}}$ and $G(v,\boldsymbol{\rho})\leq 2 \epsilon_1$. Suppose $f$ satisfy
\begin{equation}\label{linearA}
  \begin{cases}
  \square_{\mathbf{g}} f=0,\\
  f|_{t=t_0}=f_0, \quad \partial_t f|_{t=t_0}=f_1.
  \end{cases}
\end{equation}
Then for each $(f_0,f_1) \in H^1 \times L^2$ there exists a function $f_{\lambda} \in C^\infty([-2,2]\times \mathbb{R}^3)$, with
\begin{equation*}
  \mathrm{supp} \widehat{f_\lambda(t,\cdot)} \subseteq \{ \xi: \frac{\lambda}{8} \leq |\xi| \leq 8\lambda \},
\end{equation*}
such that
\begin{equation}\label{Yee}
  \begin{cases}
  & \| \square_{\mathbf{g}_\lambda} f_{\lambda} \|_{L^2_t L^2_x} \lesssim \epsilon_0 (\| f_0\|_{H^1}+\| f_1 \|_{L^2} ),
  \\
  & f_\lambda(-2)=P_\lambda f_0, \quad \partial_t f_{\lambda} (-2)=P_{\lambda} f_1.
  \end{cases}
\end{equation}
Additionally, if $r>1$, then the following Strichartz estimates holds:
\begin{equation}\label{Ase}
  \| f_{\lambda} \|_{L^2_t L^\infty_x} \lesssim \epsilon_0^{-\frac{1}{2}} \lambda^{r-1} ( \| f_0 \|_{H^1} + \| f_1 \|_{L^2} ).
\end{equation}
\end{proposition}
\begin{proof}[proof of Proposition \ref{r5} by Proposition \ref{A1}]
We divide the proof into several cases.

(i) $r=1$. Using basic energy estimates for \eqref{linear}, we have
\begin{equation*}
\begin{split}
  \| \partial_t f \|_{L^2_x} + \|\nabla f \|_{L^2_x}  \lesssim & \ (\| f_0\|_{H^1}+ \| f_1\|_{L^2}) \exp(\int^t_0 \| d \mathbf{g} \|_{L^\infty_x} d\tau)
  \\
  \lesssim & \ \| f_0\|_{H^1}+ \| f_1\|_{L^2}.
\end{split}
\end{equation*}
Then the Cauchy problem \eqref{linear} holds a unique solution $f \in C([0,T],H^1)$ and $\partial_t f \in C([0,T],L^2)$. It remains to show that the solution $f$ also satisfies the Strichartz estimate \eqref{SL}.

Without loss of generality, we take $t_0=0$. For any given initial data $(f_0,f_1) \in H^1 \times L^2$, and $t_0 \in [-2,2]$, we take a Littlewood-Paley decomposition
\begin{equation*}
  f_0=\sum_{\lambda}P_{\lambda}f_0, \qquad  f_1=\sum_{\lambda}P_{\lambda}f_1,
\end{equation*}
and for each $\lambda$ we take the corresponding $f_{\lambda}$ as in \eqref{Yee}.  If we set
\begin{equation*}
  f=\sum_{\lambda}f_{\lambda},
\end{equation*}
then $f$ matches the initial data $(f_0,f_1)$ at the time $t=t_0$, and also satisfies the Strichartz estimates \eqref{Ase}. In fact, $f$ is also an approximate solution for $\square_{\mathbf{g}}$ in the sense that
\begin{equation*}
 \| \square_{\mathbf{g}} f \|_{L^2_t L^2_x} \lesssim \epsilon_0(\| f_0 \|_{H^1}+\| f_1 \|_{L^2} ).
\end{equation*}
We can derive the above bound by using the decomposition
\begin{equation*}
  \square_{\mathbf{g}} f=\textstyle{\sum_{\lambda}} \square_{\mathbf{g}_\lambda} f_\lambda+ \textstyle{\sum_{\lambda}} \square_{\mathbf{g}-\mathbf{g}_\lambda} f_\lambda.
\end{equation*}
The first one can be controlled by Proposition \ref{A1}. As for the second one, noting $\mathbf{g}^{00}=-1$, we then have
\begin{equation*}
  \begin{split}
  \textstyle{\sum_{\lambda}} \square_{\mathbf{g}-\mathbf{g}_\lambda} f_\lambda= \textstyle{\sum_{\lambda}} ({\mathbf{g}-\mathbf{g}_\lambda}) \partial  df_\lambda.
  \end{split}
\end{equation*}
It follows from H\"older inequality that
\begin{equation}\label{y1}
  \begin{split}
  \| \textstyle{\sum_{\lambda}} \square_{\mathbf{g}-\mathbf{g}_\lambda} f_\lambda \|_{L^2_x} \lesssim \mathrm{sup}_{\lambda} \left(\lambda \| {\mathbf{g}-\mathbf{g}_\lambda} \|_{L^\infty_x} \right) \left( \sum_{\lambda}  \| df_\lambda \|^2_{L^2_x} \right)^{\frac12}.
  \end{split}
\end{equation}
On the other hand, we known
\begin{equation}\label{y2}
  \begin{split}
   \mathrm{sup}_{\lambda} \left(\lambda \| {\mathbf{g}-\mathbf{g}_\lambda} \|_{L^\infty_x} \right) \lesssim & \ \mathrm{sup}_{\lambda} \big(\lambda \sum_{\mu > \lambda}\| \mathbf{g}_\mu \|_{L^\infty_x} \big)
   \\
   \lesssim & \ \mathrm{sup}_{\lambda} \big(\lambda \sum_{\mu > \lambda}\mu^{-(1+\delta)}\|d \mathbf{g}_\mu \|_{C^\delta_x} \big)
   \\
   \lesssim & \ \|d \mathbf{g} \|_{C^\delta_x}(\textstyle{\sum_{\mu}} \mu^{-\delta}) \lesssim \ \|d \mathbf{g} \|_{C^\delta_x}.
  \end{split}
\end{equation}
Gathering \eqref{y1} and \eqref{y2}, we can prove that
\begin{equation*}
  \begin{split}
  \textstyle{\sum_{\lambda}} \square_{\mathbf{g}-\mathbf{g}_\lambda} f_\lambda \lesssim \epsilon_0(\| f_0 \|_{H^1}+\| f_1 \|_{L^2} ).
  \end{split}
\end{equation*}

For given $F\in L^1_t L^2_x$, we now form the function
\begin{equation*}
  \mathbf{M} F=\int^t_0 f^{\tau}(t,x)d\tau,
\end{equation*}
where $f^{\tau}(t,x)$ is the approximate solution formed above with the Cauchy data
\begin{equation*}
  f^{\tau}(\tau,x)=0, \quad \partial_t f^{\tau}(\tau,x)=F(\tau,\cdot).
\end{equation*}
By calculating
\begin{equation*}
  \square_{\mathbf{g}} \mathbf{M}F=\int^t_0 \square_{\mathbf{g}} f^{\tau}(t,x) d\tau+F,
\end{equation*}
it follows that
\begin{equation*}
  \| \square_{\mathbf{g}} \mathbf{M}F-F\|_{L^2_t L^2_x} \lesssim \| \square_{\mathbf{g}}f^{\tau}\|_{L^1_t L^2_x} \lesssim \epsilon_0 \| F \|_{L^2_t L^2_x}.
\end{equation*}
Using the contraction principle, we can write the solution $f$ in the form
\begin{equation*}
  f= \tilde{f}+ \mathbf{M}F,
\end{equation*}
where $\tilde{f}$ is the approximation solution formed above for initial data $(f_0,f_1)$ specified at time $t=0$, and
\begin{equation*}
  \| F\|_{L^2_t L^2_x} \lesssim \epsilon_0 ( \| f_0 \|_{H^1}+\| f_1\|_{L^2}).
\end{equation*}
The Strichartz estimates now follow since they holds for each $f^\tau$, $\tau \in [0,t]$. By Duhamel's principle, we can also obtain the Strichartz estimates for the linear, nonhomogeneous wave equation
\begin{equation*}
\begin{cases}
  \square_{\mathbf{g}} f= G,
  \\
  f|_{t=0}=f_0, \quad \partial_tf|_{t=0}=f_1.
\end{cases}
\end{equation*}
That is,
\begin{equation*}
 \| \left< \partial \right>^\theta f \|_{L^2_t L^\infty_x} \lesssim  \|f_0\|_{H^1}+ \| f_1 \|_{L^2}+ \|G\|_{L^1_t L^2_x} , \quad   \theta<0.
\end{equation*}
(ii) $1 < r \leq s_0+1$. Based on the above result, we transform the initial data in $H^1 \times L^2$. Operating $\left< \partial \right>^{r-1}$ on \eqref{linearA}, we have
\begin{equation*}
  \square_{\mathbf{g}} \left< \partial \right>^{r-1}f=-[\square_{\mathbf{g}}, \left< \partial \right>^{r-1}]f.
\end{equation*}
Let $\left< \partial \right>^{r-1}f=h$. Then $h$ is a solution to
\begin{equation}\label{qd}
  \begin{cases}
  \square_{\mathbf{g}}h=-[\square_{\mathbf{g}}, \left< \partial \right>^{r-1}]\left< \partial \right>^{1-r}h,
  \\
  (h(t_0), \partial_t h(t_0)) \in H^1 \times L^2.
  \end{cases}
\end{equation}
To handle this case, we then need to estimate the right term as
\begin{equation*}
 \| [\square_{\mathbf{g}}, \left< \partial \right>^{r-1}]\left< \partial \right>^{1-r}h\|_{L^2_t L^2_x} \lesssim \epsilon_0( \| dh\|_{L^\infty_t L^2_x}+  \| \left<\partial \right>^m dh\|_{L^2_t L^\infty_x}), \quad \mathrm{for} \  m>1-s_0.
\end{equation*}
To see this. we apply analytic interpolation to the family
\begin{equation*}
  h \rightarrow   [\square_{\mathbf{g}}, \left< \partial \right>^{r-1}]\left< \partial \right>^{1-r}h.
\end{equation*}
For $\mathrm{Re}z=0$, noting $\mathbf{g}^{00}=-1$, we use the commutator estimate (c.f. (3.6.35) of \cite{KP}) to get
\begin{equation*}
  \| [\mathbf{g}^{\alpha \beta}, \left< \partial \right>^z] \partial^2_{\alpha \beta} h \|_{L^2_x} \lesssim \| d \mathbf{g} \|_{L^\infty_x} \| d h \|_{L^2_x}.
\end{equation*}
For $\mathrm{Re}z=s_0$, we use the Kato-Ponce commutator estimate
\begin{equation}\label{qdy}
  \| [\mathbf{g}^{\alpha \beta}, \left< \partial \right>^z] \left< \partial \right>^{-z} \partial^2_{\alpha \beta} h \|_{L^2_x} \lesssim \| d \mathbf{g} \|_{L^\infty_x} \| d h \|_{L^2_x}+  \| \mathbf{g}^{\alpha \beta}- \eta^{\alpha \beta}\|_{H^{s_0}_x}   \|\left< \partial \right>^{-z} \partial^2_{\alpha \beta} h \|_{L^\infty_x},
\end{equation}
where $\alpha\neq 0$ because $\mathbf{g}^{00}=-1$. Considering
\begin{equation*}
  \| d \mathbf{g} \|_{L^\infty_x} +  \| \mathbf{g}^{\alpha \beta}- \eta^{\alpha \beta}\|_{H^{s_0}_x} \lesssim \epsilon_0,
\end{equation*}
then \eqref{qdy} becomes
\begin{equation*}
  \| [\mathbf{g}^{\alpha \beta}, \left< \partial \right>^z] \left< \partial \right>^{-z} \partial^2_{\alpha \beta} h \|_{L^2_x} \lesssim \epsilon_0( \| d h \|_{L^2_x}+  \|\left< \partial \right>^{-z} \partial^2_{\alpha \beta} h \|_{L^\infty_x}).
\end{equation*}
Let us go back \eqref{qd}. 
Using the discussion in case $r=1$, we then get the Strichartz estimates for $h$:
\begin{equation}\label{eh}
\begin{split}
 \| \left< \partial \right>^\theta h \|_{L^2_t L^\infty_x} \lesssim  & \  \epsilon_0( \|\left< \partial \right>^{r-1}f_0\|_{H^1}+ \| \left< \partial \right>^{r-1} f_1 \|_{L^2}+ \| d h \|_{L^2_x}+  \|\left< \partial \right>^{-r} \partial d h \|_{L^2_t L^\infty_x} ) ,
 \\
 \lesssim & \ \epsilon_0( \|f_0\|_{H^r}+ \| f_1 \|_{H^{r-1}}) , \quad   \theta<0.
\end{split}
\end{equation}
Substituting $\left< \partial \right>^{r-1}f=h$ to \eqref{eh}, and using the energy estimate $$\| f \|_{L^\infty_t H^r} \lesssim  \|f_0\|_{H^r}+ \| f_1 \|_{H^{r-1}},$$
 we can see derive that
\begin{equation*}
\begin{split}
 \| \left< \partial \right>^k f \|_{L^2_t L^\infty_x} \lesssim  & \  \epsilon_0( \|\left< \partial \right>^{r-1}f_0\|_{H^1}+ \| \left< \partial \right>^{r-1} f_1 \|_{L^2} ) ,
 \\
  \lesssim  & \  \epsilon_0( \|f_0\|_{H^r}+ \| f_1 \|_{H^{r-1}} ) , \quad   k<r-1.
\end{split}
\end{equation*}
\end{proof}
\subsection{The construction of wave packets}
\begin{definition}[\cite{ST}]
Let the hypersurface $\Sigma_{\theta,r}$ and the geodesic $\gamma$ be defined in Section 5. A normalized wave packet around $\gamma$ is a function $f$ of the form
\begin{equation*}
  f=\epsilon_0^{\frac12} \lambda^{-1} T_\lambda(uh),
\end{equation*}
where
\begin{equation*}
  u(t,x)=\delta(x_{\theta}-\phi_{\theta,r}(t,x'_\theta)), \quad h=h_0( (\epsilon_0 \lambda)^{\frac12}(x'_\theta-\gamma'_\theta(t))  ).
\end{equation*}
Here, $h_0$ is a smooth function supported in the set $|x'| \leq 1$, with uniform bounds on its derivatives $|\partial^\alpha_{x'} h_0(x')| \leq c_\alpha$.
\end{definition}
We give two notations here. We denote $L(u,h)$ to denote a translation invariant bilinear operator of the form
\begin{equation*}
  L(u,h)(x)=\int K(y,z)u(x+y)h(x+z)dydz,
\end{equation*}
where $K(y,z)$ is a finite measure. If $X$ is a Sobolev spaces, we then denote $X_a$ the same space but with the norm obtained by dimensionless rescaling by $a$,
\begin{equation*}
  \| u \|_{X_a}=\| u(a \cdot)\|_{X}.
\end{equation*}
Since $2(s_0-1)>2$, then for $a<1$ we have $\| u \|_{H^{s_0-1}_a(\mathbb{R}^{2})} \lesssim \| u \|_{H^{s_0-1}(\mathbb{R}^{2})}$.
\subsubsection{A normalized wave packet}
\begin{proposition}\label{np}
Let $u$ be a normalized packet. Then there is another normalized wave packet $\tilde{u}$, and functions $\phi_m(t,x'_\theta), m=0,1,2$, so that
\begin{equation}\label{np1}
  \square_{\mathbf{g}_\lambda} P_\lambda f= L(d \mathbf{g}, d \tilde{P}_\lambda \tilde{f})+ \epsilon_0^{\frac12}\lambda^{-1}P_{\lambda}T_{\lambda}\sum_{m=0,1,2}
\psi_m\delta^{(m)}(x'_\theta-\phi_{\theta,r}),
\end{equation}
where the functions $\psi_m=\psi_m(t,x'_\theta)$ satisfy the scaled Sobolev estimates
\begin{equation}\label{np2}
 \| \psi_m\|_{L^2_t H^{s_0-1}_{a,x'_\theta}} \lesssim \epsilon_0 \lambda^{1-m}, \quad m=0,1,2, \quad a=(\epsilon_0 \lambda)^{-\frac12}.
\end{equation}
\end{proposition}
As a immediate consequence, we can obtain
\begin{corollary}
Let $f$ be a normalized wave packet. Then
\begin{equation}\label{np3}
  \|d P_\lambda f \|_{L^\infty_t L^2_x} \lesssim 1, \quad \| \square_{\mathbf{g}} P_\lambda f \|_{L^2_t L^2_x} \lesssim \epsilon_0.
\end{equation}
\end{corollary}
\begin{proof}[proof of Proposition \ref{np}]
For brevity, we consider the case $\theta=(0,0,1)$. Then $x_\theta=x_3$, and $x'_\theta=x'$. We write
\begin{equation}\label{js0}
  \square_{\mathbf{g}_\lambda} P_\lambda f
  = \lambda^{-1} ( [\square_{\mathbf{g}_\lambda}, P_\lambda T_\lambda]+P_\lambda T_\lambda \square_{\mathbf{g}_\lambda} )(uh).
\end{equation}
For the first term in \eqref{js0}, noting $\mathbf{g}\lambda$ supported at frequency $\leq \frac \lambda8$, then we can write
\begin{equation*}
  [\square_{\mathbf{g}_\lambda}, P_\lambda T_\lambda]=[\square_{\mathbf{g}_\lambda}, P_\lambda T_\lambda]\tilde{P}_\lambda \tilde{T}_\lambda
\end{equation*}
for some multipliers $\tilde{P}_\lambda, \tilde{T}_\lambda$ which have the same properties as $P_\lambda, T_\lambda$. Therefore, by using the kernal bounds for $P_\lambda T_\lambda$, we conclude that
\begin{equation*}
  [\square_{\mathbf{g}_\lambda}, P_\lambda T_\lambda]f=L(d\mathbf{g}, df).
\end{equation*}
For the second term in \eqref{js0}, we use the Leibniz rule
\begin{equation}\label{js1}
  \square_{\mathbf{g}_\lambda}(uh)=h \square_{\mathbf{g}_\lambda} u+(\mathbf{g}^{\alpha \beta}_\lambda+\mathbf{g}^{\beta \alpha}_\lambda)\partial_\alpha u \partial_\beta h + u \square_{\mathbf{g}_\lambda} h.
\end{equation}
We let $\nu$ denote the conormal vector field along $\Sigma$, $\nu=dx_3-d\phi(t,x')$. In the following, we take the greek indices $0 \leq \alpha, \beta \leq 2$.

For the first term in \eqref{js1}, by calculation, we have
\begin{equation*}
  \begin{split}
  \mathbf{g}_{\lambda}^{\alpha \beta} \partial_{\alpha \beta} u =& \mathbf{g}_{\lambda}^{\alpha \beta}(t,x',\phi)\nu_\alpha \nu_\beta \delta^{(2)}_{x_3-\phi}
  -2 (\partial_3 \mathbf{g}_{\lambda}^{\alpha \beta})(t,x',\phi) \nu_\alpha \nu_\beta \delta^{(1)}_{x_3-\phi}
  \\
  &+(\partial^2_3 \mathbf{g}_{\lambda}^{\alpha \beta})(t,x',\phi) \nu_\alpha \nu_\beta \delta^{(0)}_{x_3-\phi}- \mathbf{g}_{\lambda}^{\alpha \beta})(t,x',\phi) \partial_{\alpha \beta}\phi \delta^{(1)}_{x_3-\phi}
  \\
  &+ \partial_3\mathbf{g}_{\lambda}^{\alpha \beta})(t,x',\phi) \partial_{\alpha \beta}\phi \delta^{(0)}_{x_3-\phi}.
  \end{split}
\end{equation*}
Here, $\delta^{(m)}_{x_3-\phi}=(\partial^m \delta)(x_3-\phi) $. By Leibniz rule, we can take
\begin{equation*}
\begin{split}
  \psi_0&=h \big\{ (\partial^2_3 \mathbf{g}_{\lambda}^{\alpha \beta})(t,x',\phi)\nu_\alpha \nu_\beta+ (\partial_3 \mathbf{g}_{\lambda}^{\alpha \beta})(t,x',\phi)\partial_{\alpha \beta}\phi  \big\},
  \\
  \psi_1&=h \big\{ 2(\partial_3 \mathbf{g}_{\lambda}^{\alpha \beta})(t,x',\phi)\nu_\alpha \nu_\beta-  \mathbf{g}_{\lambda}^{\alpha \beta}(t,x',\phi)\partial_{\alpha \beta}\phi \big\},
  \\
  \psi_2&=h( \mathbf{g}_{\lambda}^{\alpha \beta}-\mathbf{g}^{\alpha \beta})\nu_\alpha \nu_\beta,
  \end{split}
\end{equation*}
By using \eqref{G}, Proposition \ref{r1}, and Corollary \ref{vte}, we can conclude that this settings of $\psi_0, \psi_1,$ and $\psi_2$ satisfy the estimates \eqref{np2}. For the second term in \eqref{js0}, we have
\begin{equation*}
  \begin{split}
  (\mathbf{g}^{\alpha \beta}_\lambda+\mathbf{g}^{\beta \alpha}_\lambda)\partial_\alpha u \partial_\beta h=& \frac12\nu_\alpha (\mathbf{g}^{\alpha \beta}_\lambda+\mathbf{g}^{\beta \alpha}_\lambda)(t,x',\phi)\partial_\beta h \delta^{(1)}_{x_3-\phi}
  \\
  &- \frac12\nu_\alpha \partial_3 (\mathbf{g}^{\alpha \beta}_\lambda+\mathbf{g}^{\beta \alpha}_\lambda)(t,x',\phi)\partial_\beta h \delta^{(0)}_{x_3-\phi}.
  \end{split}
\end{equation*}
We then take
\begin{equation*}
  \psi_0= \frac12\nu_\alpha \partial_3 (\mathbf{g}^{\alpha \beta}_\lambda+\mathbf{g}^{\beta \alpha}_\lambda)(t,x',\phi)\partial_\beta h, \quad \psi_1= \frac12\nu_\alpha (\mathbf{g}^{\alpha \beta}_\lambda+\mathbf{g}^{\beta \alpha}_\lambda)(t,x',\phi)\partial_\beta h.
\end{equation*}
By using \eqref{G}, Proposition \ref{r1}, and Corollary \ref{vte}, we can conclude that this settings of $\psi_0$ and $\psi_1$ satisfy the estimates \eqref{np2}.
For the third term in \eqref{js0}, we can take
\begin{equation*}
  \psi_0=\mathbf{g}^{\alpha \beta}_\lambda(t,x',\phi)\partial_{\alpha \beta}h.
\end{equation*}
By using \eqref{G}, Proposition \ref{r1}, and Corollary \ref{vte}, we can conclude that this settings of $\psi_0$ satisfies the estimates \eqref{np2}.
\end{proof}
\subsubsection{Superpositions of wave packets}
\begin{proposition}[\cite{ST}]\label{szy}
Let $f=\sum_{\theta,j}a_{\theta,j}f^{\theta,j}$, where $f^{\theta,j}$ are normalized wave packets supported in $T_{\theta,j}$. Then
\begin{equation}\label{ese}
\begin{split}
  \| d P_\lambda f\|_{L^\infty_t L^2_x} &\lesssim (\sum_{\theta,j} a^2_{\theta,j})^{\frac12},
  \\
\| \square_{\mathbf{g}_{\lambda}} P_\lambda f \|_{L^1_t L^2_x} &\lesssim \epsilon_0 (\sum_{\theta,j} a^2_{\theta,j})^{\frac12}.
\end{split}
\end{equation}
\end{proposition}
\subsubsection{Matching the initial data}
\begin{proposition}[\cite{ST}]\label{szi}
Given any initial data $(f_0,f_1) \in H^1 \times L^2$, there exists a function of the form
\begin{equation*}
  f=\sum_{\theta,j}a_{\theta,j}f^{\theta,j},
\end{equation*}
where the function $f^{\theta,j}$ are normalized wave packets, such that
\begin{equation*}
  P_\lambda f(-2)=P_\lambda f_0, \quad \partial_t P_\lambda f(-2)=P_\lambda f_1.
\end{equation*}
Furthermore,
\begin{equation*}
\sum_{\theta,j}a_{\theta,j}^2 \lesssim \| f_0 \|^2_{H^1}+ \| f_1 \|^2_{L^2}.
\end{equation*}
\end{proposition}
\subsection{Strichartz estimate}
The estimate \eqref{Ase} can be directly derived by
\begin{proposition}[\cite{ST}]\label{szt}
Let
\begin{equation*}
  f=\sum_{K \in \mathcal{T}}a_{K}\chi_{K},
\end{equation*}
where $\sum_{K \in \mathcal{T}}a_{K}^2 \leq 1$. Then
\begin{equation*}
\|f \|_{L^2_t L^\infty_x} \lesssim \epsilon_0^{-\frac12}(\ln \lambda)^3.
\end{equation*}
\end{proposition}
\section*{Acknowledgments} Part of this work was done while the authors were in residence at Institut Mittag-Leffler in Djursholm, Sweden during the fall of 2019, supported by the Swedish Research Council under grant no. 2016-06596. The author Huali Zhang is also supported by Natural Science Foundation of Hunan Province, China(Grant No. 2021JJ40561).


\begin{thebibliography}{4}
\addcontentsline{toc}{section}{References}
\bibitem{AAR}
P.T. Allen, L. Andersson, A. Restuccia. Local well-posedness for membranes in the light
cone gauge, Comm. Math. Phys., 301, 383-410 (2011).


\bibitem{AM}
L. Andersson, V. Moncrief. Elliptic-Hyperbolic Systems and the Einstein equations, Ann. Henri Poincar\'e, 4, 1-34(2003).



\bibitem{BC1}
H. Bahouri and J.Y. Chemin. E\'quations d\'ondes quasilineaires et effet dispersif, Internat. Math. Res. Notices, 21, 1141-1178 (1999).
\bibitem{BC2}
H. Bahouri and J.Y. Chemin. E\'quations d\'ondes quasilineaires et estimation de Strichartz, Amer. J. Math., 121, 1337-1377 (1999).
\bibitem{BCD}
H. Bahouri, J. Y. Chemin, and Rapha\"el Danchin, Fourier analysis and nonlinear partial differential equations,
Grundlehren der Mathematischen Wissenschaften, vol. 343, Springer,
Heidelberg, 2011.

\bibitem{BL}
J. Bourgain, D. Li. Strong ill-posedness of the incompressible Euler
equation in borderline Sobolev spaces, Invent. Math., 201:97-157 (2015).

\bibitem{BL2}
J. Bourgain, L. Dong. Strong ill-Posedness of the 3D incompressible Euler equation in borderline
Spaces, Int. Math. Res. Notices, 00(0), 1-110 (2019).
\bibitem{Chae}
D. Chae. On the well-posedness of the Euler equations in the Triebel-Lizorkin spaces. Comm. Pure Appl. Math. 55(5), 654-678.
\bibitem{Ch2}
D. Chae. On the Euler equations in the critical Triebel-Lizorkin spaces. Arch. Ration. Mech. Anal. 170(3), 185-210 (2003).


\bibitem{C}
Demetrios Christodoulou and Shuang Miao. Compressible flow and Euler's equations, Surveys of Modern Mathematics,
vol. 9, International Press, Somerville, MA; Higher Education Press, Beijing, 2014.

\bibitem{CLS}
Daniel Coutand, Hans Lindblad, and Steve Shkoller. A priori estimates for the free-boundary 3D compressible Euler
equations in physical vacuum, Comm. Math. Phys. 296(2), 559-587 (2010).

\bibitem{CS}
Daniel Coutand and Steve Shkoller. Well-posedness in smooth function spaces for moving-boundary 1-D compressible Euler
equations in physical vacuum, Comm. Pure Appl. Math. 64(3), 328-366 (2011).

\bibitem{CS2}
Daniel Coutand and Steve Shkoller. Well-posedness in smooth function spaces for the moving-boundary three-dimensional compressible Euler equations
in physical vacuum, Arch. Ration. Mech. Anal. 206(2), 515-616 (2012).


\bibitem{DLS}
M. Disconzi, C. Luo, G. Mazzone and J. Speck. Rough sound waves in 3D compressible euler flow with
vorticity, arXiv:1909.02550v1, 100 pages.
\bibitem{EL}
B. Ettinger, H. Lindblad. A sharp counterexample to local existence
 of low regularity solutions to Einstein
 equations in wave coordinates, Ann. Math., 185, 311-330 (2017).

\bibitem{JM}
Juhi Jang and Nader Masmoudi. Well-posedness for compressible Euler equations with physical vacuum singularity, Comm.
Pure Appl. Math. 62(10), 1327-1385 (2009).

\bibitem{HKM}
T. Hughes, T. Kato and J.E. Marsden. Well-posed quasi-linear second-order hyperbolic systems with
applications to nonlinear electrodynamics and general relativity, Arch. Rat. Mech. Anal., 63, 273-
294 (1977).

\bibitem{IT}
M. Ifrim, D. Tataru. The compressible Euler equations in a physical vacuum:a comprehensive Eulerian approach, arXiv:2007.05668, 79pages.


\bibitem{Geba}
D.G. Geba. A local well-posedness result for the quasilinear
wave equation in $\mathbb{R}^{2+1}$,  Comm. Part. Diff. Equ, 29, 323-360 (2004).

\bibitem{KP}
T. Kato, G. Ponce. Commutator estimates and the Euler and Navier-Stokes equations, Comm. Pure Appl. Math., 41(7), 891-907 (1988).

\bibitem{K}
S. Klainerman. A commuting vectorfield approach to Strichartz type inequalities and applications to quasilinear wave equations, Int. Math. Res. Notices, 5, 221-274 (2001).

\bibitem{KR2}
S. Klainerman, I. Rodnianski. Improved local well-posedness for quasilinear wave equations in dimension three, Duke Math. J., 117, 1-124(2003).
\bibitem{KR}
S. Klainerman, I. Rodnianski. Rough solutions of the Einstein vacuum equations, Ann. Math., 161, 1143-1193 (2005).

\bibitem{KR1}
S. Klainerman, I. Rodnianski and J. Szeftel. The bounded $\textrm{L}^2$ curvature conjecture, Invent. Math., 202(1), 91216 (2015).

\bibitem{LW}
T.S Li, L.B. Wang. Global Propagation
of Regular Nonlinear Hyperbolic Waves, Springer, New York, 2009.

\bibitem{L}
H. Lindblad. Counterexamples to local existence for quasilinear wave equations, Math. Res. Letters, 5(5), 605-622 (1998).
\bibitem{LS1}
J. Luk, J. Speck. Shock formation in solutions to the 2D compressible Euler equations in the presence of non-zero vorticity, Invent. Math., 214, 1-169 (2018).
\bibitem{LS2}
J. Luk, J. Speck. The hidden null structure of the compressible Euler equations and a prelude to applications, To appear in Journal of Hyperbolic Differential Equations, arXiv:1610.00743 (2016)
\bibitem{M}
A. Majda. Compressible fluid flow and systems of conservation laws in several space variables, Applied
Mathematical Sciences, 53. Springer, New York, 1984.
\bibitem{ML}
C. Miao, L. Xue, On the global well-posedness of a class of Boussinesq-Navier-Stokes systems, NoDEA Nonlinear Differential Equations Appl. 18 (6), 707-735 (2011).

\bibitem{S}
T.C. Sideris. Formation of Singularities in Three Dimensional
Compressible Fluids, Commun. Math. Phys. 101, 475-485 (1985).

\bibitem{ST}
H.F. Smith, D. Tataru. Sharp local well-posedness results for the nonlinear wave equation, Ann.
Math., 162, 291-366 (2005).

\bibitem{ST0}
H.F. Smith and D. Tataru. Sharp counterexamples for Strichartz estimates for low regularity metrics,
Mathematical Research Letters, 199-204 (2002).

\bibitem{T1}
D. Tataru. Strichartz estimates for operators with nonsmooth coefficients and the nonlinear wave
equation, Am. J. Math., 122, 349-376 (2000).
\bibitem{T2}
D. Tataru. Strichartz estimates for second order hyperbolic operators with nonsmooth coefficients II,
Am. J. Math. 123, 385-423 (2001).
\bibitem{T3}
D. Tataru. Strichartz estimates for second order hyperbolic operators with nonsmooth coefficients III.
J. Am. Math. Soc. 15, 419-442 (2002).

\bibitem{WCB}
C.B. Wang. Sharp local well-posedness for quasilinear wave equations with spherical symmetry, appeared in Journal of the European Mathematical Society.

\bibitem{WQRough}
Q. Wang. Rough Solutions of Einstein vacuum equations in CMCSH Gauge, Comm. Math. Phys. 328, 1275-1340 (2014).

\bibitem{WQ1}
Q. Wang. Causal geometry of rough Einstein CMCSH spacetime, J. Hyperbolic Differ. Equ. 11(3), 563-601 (2014).

\bibitem{WQ2}
Q. Wang. Improved breakdown criterion for Einstein vacuum equations in CMC gauge, Commun. Pure Appl. Math. 65(1), 21-76 (2012).


\bibitem{WQSharp}
Q. Wang. A geometric approach for sharp local well-posedness
of quasilinear wave equations, Ann. PDE, 3:12 (2017).
\bibitem{WQEuler}
Q. Wang. Rough solutions of the 3-D compressible Euler equations, arXiv:1911.05038v1, 110 pages.
\bibitem{Z1}
H.L. Zhang. Local existence theory for 2D compressible Euler equations with low regularity, appeared in Journal of Hyperbolic Differential Equations.
\bibitem{Z2}
H.L. Zhang. On the Cauchy problem of 2D compressible Euler equations: low regularity solutions $\Pi$, arXiv:2012.01060.

\end{thebibliography}
\end{document}